\documentclass[oneside,final]{amsart}

\usepackage[utf8]{inputenc}
\usepackage[backend=biber,style=alphabetic,giveninits=true,noerroretextools,maxbibnames=99]{biblatex}
\usepackage{csquotes}

\DeclareSourcemap{
  \maps[datatype=bibtex]{
     \map{
        \step[fieldsource=doi,final]
        \step[fieldset=isbn,null]
        \step[fieldset=issn,null]
        }
      }
}

\AtEveryBibitem{\clearfield{url}} 
\AtEveryCitekey{\clearfield{url}} 

\usepackage[T1]{fontenc}
\usepackage{geometry}

\usepackage{hyperref}
\usepackage[english]{babel}
\usepackage{enumerate}
\usepackage{geometry}
\usepackage{aliascnt}
\usepackage{amssymb}
\usepackage{amsmath}
\usepackage{amsthm}
\usepackage{verbatim}
\usepackage[normalem]{ulem}
\usepackage{mathtools}
\usepackage{esint} 
\usepackage{nth}
\usepackage{ragged2e}

\usepackage[noabbrev,capitalize]{cleveref}

\usepackage[right]{showlabels}

\usepackage{tikz}
\usepackage{tikzscale}
\usepackage{graphicx}
\usepackage[labelformat=simple]{subcaption}
\usepackage{mathscinet} 
\usepackage[title]{appendix}

\usepackage{color}

\numberwithin{equation}{section}

\usepackage{dsfont} 
\usepackage{hyperref}

\let\etoolboxforlistloop\forlistloop 
\usepackage{autonum}
\let\forlistloop\etoolboxforlistloop 

\allowdisplaybreaks[4]

\theoremstyle{plain}

\newtheorem{proposition}{Proposition}[section]

\newaliascnt{lemma}{proposition} 

\newtheorem{lemma}[lemma]{Lemma}

\aliascntresetthe{lemma} 
\Crefname{lemma}{Lemma}{Lemmas}

\newaliascnt{theorem}{proposition} 
\newtheorem{theorem}[theorem]{Theorem}

\aliascntresetthe{theorem} 

\newaliascnt{corollary}{proposition} 
\newtheorem{corollary}[corollary]{Corollary}

\aliascntresetthe{corollary}

\newaliascnt{hypothesis}{proposition} 
\newtheorem{hypothesis}[lemma]{Hypothesis}

\aliascntresetthe{hypothesis} 

\newaliascnt{setting}{proposition} 
\newtheorem{setting}[setting]{Setting}

\aliascntresetthe{setting} 

\theoremstyle{definition}

\newaliascnt{definition}{proposition} 
\newtheorem{definition}[definition]{Definition}

\aliascntresetthe{definition} 
\Crefname{definition}{Definition}{Definitions}

\newaliascnt{problem}{proposition} 

\aliascntresetthe{problem} 

\newaliascnt{example}{proposition} 

\aliascntresetthe{example} 

\theoremstyle{remark}

\newaliascnt{remark}{proposition} 
\newtheorem{remark}[remark]{Remark}

\aliascntresetthe{remark} 

\newaliascnt{notation}{proposition} 
\newtheorem{notation}[notation]{Notation}

\aliascntresetthe{notation} 

\newtheorem{claim}{Claim}

\makeatletter
\@addtoreset{claim}{proposition}
\makeatother

\def\equationautorefname~#1\null{%
	(#1)\null
}

\newcommand{\R}{\mathbb{R}}
\newcommand{\C}{\mathbb{C}}
\newcommand{\Z}{\mathbb{Z}}
\newcommand{\N}{\mathbb{N}}
\newcommand{\E}{\mathcal{E}}

\newcommand{\curv}{\vec{\kappa}}

\renewcommand{\S}{\mathbb{S}}
\newcommand{\X}{\mathbb{X}}

\newcommand{\A}{\mathcal{A}}

\renewcommand{\H}{\mathbb{H}}

\newcommand{\Ll}{\mathcal{L}}

\newcommand{\diam}{\mathrm{diam}}

\newcommand{\iu}{\mathrm{i}\mkern1mu} 

\newcommand*{\dd}{\mathop{}\!\mathrm{d}}

\def\nicefrac#1#2{%
    \raise.5ex\hbox{$#1$}%
    \kern-.15em/\kern-.05em%
    \lower.25ex\hbox{$#2$}}

\title[\textsc{The length-preserving elastic flow with free boundary on hypersurfaces in $\R^n$}]{\Large The length-preserving elastic flow\\ with free boundary on hypersurfaces in $\R^n$}

\author[A.~Dall'Acqua]{Anna Dall'Acqua}
\address[A.~Dall'Acqua]{Institute for Applied Analysis, Ulm University, Helmholtzstraße 18, 89081 Ulm, Germany.}
\email{anna.dallacqua@uni-ulm.de}
\author[M.~Schlierf]{Manuel Schlierf}
\address[M.~Schlierf]{Institute for Applied Analysis, Ulm University, Helmholtzstraße 18, 89081 Ulm, Germany.}
\email{manuel.schlierf@uni-ulm.de}

\begin{document}

\begin{abstract}
    We study the length-preserving elastic flow of curves in arbitrary codimension with free boundary on hypersurfaces. This constrained gradient flow is given by a nonlocal evolution equation with nonlinear higher-order boundary conditions. We prove global existence and subconvergence to critical points. The proof strategy involves a careful treatment of short-time existence, uniqueness, and parabolic energy estimates. 
\end{abstract}

\begingroup
\def\uppercasenonmath#1{\scshape} 
\maketitle
\endgroup



\section{Introduction}

The study of elastic curves can be traced back to a problem posed by James Bernoulli in 1691 which is concerned with the shape of a suitable elastic beam when a weight is hung from its top. While Bernoulli came up with a solution to his problem, in 1744 Euler was the first to completely characterize the family of curves known as elastica, 
see 
\cite{truesdell1983} for a historical treatment. Today 
the following setup is considered. Let $I$ be a compact one-dimensional connected manifold and $\gamma\colon I\to\R^n$ be an immersion. Then 
\begin{equation}
    \E(\gamma)=\int_{\S^1}|\curv|^2\dd s
\end{equation}
is the \emph{elastic energy} of $\gamma$ where $\curv=\partial_s^2\gamma$, $\partial_s=\frac{1}{|\gamma'|}\partial_x$ and $\dd s=|\gamma'|\dd x$ is the Riemannian measure on $I$ induced by $\gamma^*\langle\cdot,\cdot\rangle$, the pull-back of the Euclidean metric on $\R^n$. Moreover, $\Ll(\gamma)=\int_I\dd s$ denotes the length of $\gamma$. The critical points of $\E+\lambda\Ll$ where $\lambda\in\R$ are called \emph{elastica} and the study of associated variational problems is still a vibrant topic of research. Due to the behavior of $\E$ with respect to dilations, we trivially have 
\begin{equation}
    \inf\{\E(\gamma)\mid \gamma\in W^{2,2}( I ,\R^n)\text{ is an immersion}\}=0
\end{equation}
and the infimum is not attained if $I=\S^1$. However, setting $I=\S^1$ and minimizing with a fixed length constraint, round circles with the given length are the only minimizers. Indeed, this is a classical argument involving Fenchel's Theorem \cite{fenchel1929,borsuk1947} and the Cauchy--Schwarz inequality. Even more, for $n=2$, it is known that the only closed planar elastica are given by round circles or the figure-eight, up to similarity \cite{langersinger1985}. Among closed curves, variational techniques can be used to show that the elastic energy encodes topological information. In fact, if $I=\S^1$ and the product $\E(\gamma)\Ll(\gamma)$ is below an explicitly known threshold, then $\gamma$ is an embedding. This was first proved in \cite{mullerrupp2023} for $n=2$ and improved to $n>2$ in \cite{miura2023}. 

When minimizing $\E$ with fixed length among open curves, it turns out that the properties of the solutions sensitively depend on the prescribed boundary data. Uniqueness and the asymptotic shape of minimizers when suitably varying clamped boundary conditions, i.e.\ fixed end-points and tangents at the end-points, is treated in \cite{miura2020}. In \cite{dallacquadeckelnick2024}, the minimization with clamped boundary conditions is carried out among graphs, and qualitative properties of the minimizers are studied using a Noether-type identity. See \cite{yoshizawa2022} for a treatment of \emph{pinned} boundary conditions. Moreover, a recent development is the minimization of $\E+\lambda\Ll$ for so-called \emph{networks of curves} where $\lambda\geq 0$ and corresponds to coupled systems with nonlinear, higher-order boundary conditions \cite{dallacquapluda2017,dallacquanovagapluda2020,delninpludapozzetta2021}.

\subsection{The elastic flow}

A further line of research is concerned with the evolution modeling how a deformed elastic rod regains its equilibrium shape: The mathematically most efficient way is described by the gradient flow of the (possibly penalized) elastic energy, the so-called \emph{elastic flow}. The associated evolution equation is a quasi-linear, fourth-order, degenerate parabolic PDE and possibly includes nonlocal terms (for example, if the length is kept fixed). In particular, no maximum principles are available and the study of asymptotic behavior, that is, proving global existence and (sub-) convergence to critical points, is analytically challenging. 

To be more precise, a family of immersions $\gamma\colon[0,T)\times I\to\R^n$ is an elastic flow if
\begin{align}
    \partial_t\gamma &= - \big(2(\partial_s^\bot)^2\curv + |\curv|^2\curv -\lambda(t) \curv \big)\\
    &=\vcentcolon -\big(\nabla\E(\gamma(t,\cdot)) + \lambda(t)\nabla\Ll(\gamma(t,\cdot))\big) \quad\text{on $[0,T)\times I$},\label{eq:intro-elastic-flow-def}
\end{align}
with suitable boundary conditions if $\partial I\neq \emptyset$. Here $\partial_s^\bot=\partial_s - \langle \partial_s\ \cdot\ ,\partial_s\gamma\rangle\partial_s\gamma$. Moreover, $\gamma$ is a 
\begin{enumerate}
    \item \emph{free} elastic flow if $\lambda(t)\equiv 0$;
    \item \emph{length-penalized} elastic flow if $\lambda(t)\equiv \lambda>0$;
    \item \emph{length-preserving} elastic flow if $\lambda(t)=\lambda(\gamma(t,\cdot))$ with
    \begin{equation}\label{eq:def-lambda}
        \lambda(\gamma)=\frac{\int_{I}\langle 2(\partial_s^\bot)^2\curv+|\curv|^2\curv,\curv\rangle\dd s}{\int_{I}|\curv|^2\dd s} =  \frac{\int_{I}\langle \nabla\E(\gamma),\curv\rangle\dd s}{\E(\gamma)}.
    \end{equation}
    In this case, we compute $\partial_t\Ll(\gamma(t,\cdot))=0$.
\end{enumerate}

In \cite{dziukkuwertschaetzle2002}, a proof of global existence and subconvergence to critical points is given for all three types of elastic flows of closed curves, introducing a proof strategy based on parabolic energy estimates obtained with geometric versions of Gagliardo--Nirenberg-type interpolation inequalities. 
These energy techniques rely on a crucial step involving integration by parts --- and thus, they sensitively depend on the underlying boundary conditions. Clearly, the periodic boundary conditions of closed curves do not produce any boundary terms. 

In \cite{lin2012}, the strategy of \cite{dziukkuwertschaetzle2002} was first successfully modified and applied to a (free or length-penalized) elastic flow of open curves, namely the elastic flow with clamped boundary conditions. Also, see \cite{dallacquapozzi2014} for the (free or length-penalized) elastic flow with Navier boundary conditions corresponding to the minimization of pinned curves. As it turns out, the corresponding length-preserving elastic flows come with further significant difficulties in carrying out energy estimates and are treated separately in \cite{dallacqualinpozzi2014,dallacqualinpozzi2017}. Moreover, for the length-preserving clamped elastic flow, \cite{ruppspener2020} gives a detailed study of short-time existence and uniqueness for initial data in the energy space $W^{2,2}(I,\R^n)$.

The energy estimates carried out for the above elastic flows only yield sequential compactness for $t\nearrow \infty$, i.e.\ subconvergence to critical points (after suitable reparametrization). Proving full convergence can be achieved by proving and applying a \L ojasiewicz--Simon gradient inequality for the elastic energy; see \cite{mantegazzapozzetta2021,dallacquapozzispener2016,pozzetta2022} and \cite{ruppspener2020} for a length-preserving case.

Once global existence and convergence of the elastic flow is established, since no maximum principle is available, studying qualitative properties of elastic flows leads to further challenging questions. In this context, optimal energy thresholds for preserving the embeddedness of length-penalized and length-preserving elastic flows are treated in \cite{miuramuellerrupp2025}. There is also recent work on a problem posed by Huisken on the possibility of a ``migration phenomenon'' for elastic flows \cite{miura2023migrating-ef-1,miura2024migrating-ef-2}. Moreover, a careful analysis of the classification of elastica leads to the recent result \cite[Theorem~1.4]{muelleryoshizawa2024} proving the eventual embeddedness of the elastic flow with natural boundary conditions below a critical energy threshold.

As indicated above, nonlinear, especially higher-order boundary conditions for elastic flows of open curves result in a more involved analysis when establishing energy estimates as in \cite{dziukkuwertschaetzle2002}. This becomes apparent in \cite{dallacqualinpozzi2019,garckemenzelpluda2020}, where elastic flows of networks are studied. In particular, these come with a third-order boundary condition at the points where several curves in the network are joined.

\subsection{Free boundary elastic flows and main results}

In this article, we study the length-preserving elastic flow of open curves satisfying certain \emph{free boundary conditions}. Similarly to the study of elastic networks, the free boundary problem comes with a set of nonlinear higher-order boundary conditions, and thus with new challenges when studying the associated elastic flow.

More precisely, we fix $I=[-1,1]$, $n\geq 2$ and a complete and embedded hypersurface $M\subseteq\R^n$ and an associated smooth unit normal field $\xi\colon M\to\S^{n-1}$. If $M$ is noncompact, we additionally impose a condition suitably bounding the geometry of $M$, see \Cref{hyp:unif-tub-nbhd} below for more details. 

\begin{figure}[htb!]
    \vspace*{-1.2em}
    \centering
    \includegraphics[width=0.31\textwidth]{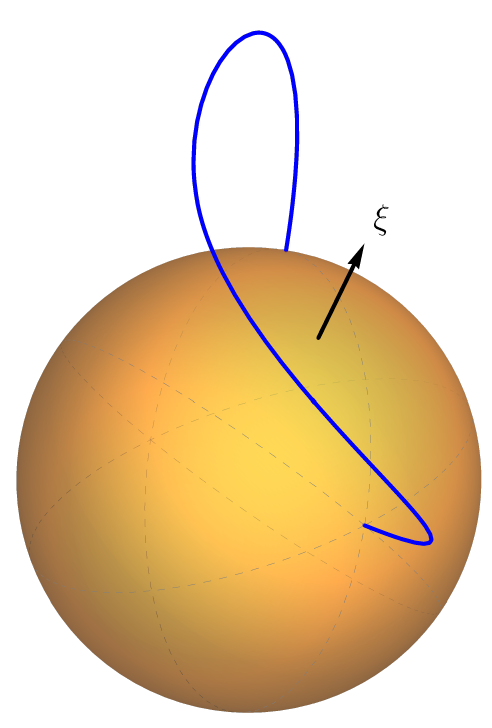}
    \vspace*{-0.6em}
    \caption{Curve $\gamma$ with orthogonal free boundary on a round sphere $M\subseteq \R^3$ with unit normal $\xi$. Here $\langle \partial_s\gamma(\pm1),\xi(\gamma(\pm1))\rangle = \mp1$.}\label{fig:curv-with-free-bdry-on-sphere}
\end{figure}

In some special cases, there have been recent results on the asymptotics of elastic flows with free boundary conditions. In \cite{diana2024}, the case $n=2$ with $M=\R\times\{0\}$ is studied. The author computes the Navier boundary conditions satisfied by the critical points $\gamma\colon I\to\R^2$ of $\E+\lambda\Ll$, prescribing the constraint $\gamma(\pm1)\in\R\times\{0\}$ and studies the associated 
length-penalized elastic flow, prescribing this set of free Navier boundary conditions. The main result is \cite[Theorem~5.1]{diana2024}, proving global existence under some qualitative assumptions about the behavior of the flow. The recent article \cite{gazwanimccoy2024} also works in the planar case with $M\subseteq\R^2$ given by a cone whose tip lies at $(0,0)$, studying the three types of elastic flows. The authors prescribe \emph{orthogonal free boundary conditions} $\gamma(\pm1)\in M\setminus\{(0,0)\}$ and $\partial_s\gamma(\pm 1)\perp T_{\gamma(\pm1)}M$. The elastic flows are studied under the additional third-order boundary condition $\partial_s^\bot\curv(\pm1)=0$ and global existence is established assuming that the flow stays away from the cone tip $(0,0)$.

Prescribing \emph{orthogonal} free boundary conditions is well-established for variational problems of Willmore-type, see \cite{alessandronikuwert2016,kuwertlamm2021,kuwertmueller2024,metsch2024} for the two-dimensional case. In this article, we also impose orthogonal free boundary conditions on $M$, thus generalizing the setting of \cite{gazwanimccoy2024}. We treat the case of arbitrary codimension with a more general free-boundary hypersurface $M$. In particular, in contrast to \cite{diana2024,gazwanimccoy2024}, recall that we do not assume that $M$ is flat.  

More precisely, for an immersion $\gamma\colon[-1,1]\to\R^n$, we consider the conditions
\begin{align}
    \gamma(\pm1)&\in M,\quad \partial_s\gamma(\pm1) \bot T_{\gamma(\pm1)}M \quad\text{and}\\
    \partial_s^{\bot}\curv(\pm1)&=-\langle\partial_s\gamma(\pm1),\xi(\gamma(\pm1))\rangle S_{\gamma(\pm1)}\curv(\pm1) \label{eq:intro-free-bcs}
\end{align}
where $S_p$ is the shape operator of $M$ at $p\in M$, see \Cref{sec:geom-prelim} below 
and 
\Cref{fig:curv-with-free-bdry-on-sphere}.

Our first main result concerns the existence of length-preserving elastic flows satisfying these free boundary conditions. In \Cref{app:fct-spaces}, we introduce the function spaces used below.

\begin{theorem}\label{thm:existence-of-solution}
    Let $p>5$ and $\gamma_0\in W^{4(1-\frac1p),p}([-1,1],\R^n)$ be an immersion with $\E(\gamma_0)>0$, satisfying \eqref{eq:intro-free-bcs}. 
    There exist $T>0$ and a family of immersions $\gamma\colon[0,T)\times[-1,1]\to\R^n$ with $\gamma\in C^\infty((0,T)\times [-1,1])$ such that $[0,T)\mapsto W^{4(1-\frac1p),p}([-1,1])$, $t\mapsto \gamma(t,\cdot)$ is continuous and such that the following holds. We have
    \begin{equation}\label{eq:geo-ev-eq}
        \begin{cases}
            \partial_t^\bot\gamma(t,x) = -(\nabla\E(\gamma(t,\cdot))(x)-\lambda(\gamma(t,\cdot)) \cdot \curv(t,x))&\text{on $(0,T)\times[-1,1]$}\\
            \gamma(0,\cdot)=\gamma_0&\text{in $W^{4(1-\frac1p),p}([-1,1])$}\\
            \gamma(t,\cdot)\text{ satisfies \eqref{eq:intro-free-bcs}}&\text{for all $0\leq t<T$}
        \end{cases}
    \end{equation}
    where $\lambda(\gamma)$ is given by \eqref{eq:def-lambda} and $\partial_t^\bot=\partial_t-\langle\partial_t\ \cdot\ ,\partial_s\gamma\rangle\partial_s\gamma$. 
\end{theorem}

Notice that, since $p>5$, $W^{4(1-\frac1p),p}(I)\hookrightarrow C^{3+(1-\frac5p)}(I)$ by the Sobolev inequality. Therefore, it is also necessary that $\gamma_0$ satisfies \eqref{eq:intro-free-bcs}. This makes $p>5$ a natural choice for the free boundary conditions \eqref{eq:intro-free-bcs}.

Well-posedness for geometric flows is rather delicate, and a detailed proof is often omitted in past work studying the asymptotics of elastic flows, for example \cite{dziukkuwertschaetzle2002,lin2012,dallacquapozzi2014}. In fact, the degeneracy of the equation comes from the fact that \eqref{eq:geo-ev-eq} is invariant with respect to reparametrization, that is, for any solution $\gamma$ and any smooth family of diffeomorphisms $\phi(t,\cdot)$ of $I$ with $\phi(0,\cdot)=\mathrm{id}_I$, also $(t,x)\mapsto \gamma(t,\phi(t,x))$ is a solution of \eqref{eq:geo-ev-eq}. While this introduces challenges when studying short-time existence and uniqueness (in a suitable sense), this observation later drastically aids the proof of global existence. Given a non-extensible solution $\gamma$ as in \Cref{thm:existence-of-solution} with vanishing tangential speed, our careful existence analysis will allow us to choose a suitable parametrization for which energy estimates for derivatives of the curvature yield control of full derivatives of the (reparametrized) solution. Since the initial datum only needs to be bounded in $W^{4(1-\frac1p),p}(I)$ for some $p>5$, 
this may simplify the arguments in proving global existence compared to \cite{dallacqualinpozzi2014,dallacqualinpozzi2017,dallacqualinpozzi2019}. 

The proof strategy for \Cref{thm:existence-of-solution} can be summarized as follows. First, we choose a suitable \emph{analytic equation} for \eqref{eq:geom-ev-eq}. That is, we modify the boundary conditions \eqref{eq:intro-free-bcs} as well as the right hand side in the PDE to obtain a well-posed quasi-linear parabolic system of $n$ equations whose solutions are instantaneously smooth. This yields solutions of the geometric problem \eqref{eq:geo-ev-eq}.

For studying the asymptotic behavior, we consider solutions of \eqref{eq:geo-ev-eq} with zero tangential component of the speed (i.e.\ $\langle \partial_t \gamma, \partial_s \gamma\rangle=0$). Existence of such solutions can be obtained from \Cref{thm:existence-of-solution} as follows, especially using the fact that the boundary conditions \eqref{eq:intro-free-bcs} are invariant with respect to reparametrization.

\begin{corollary}\label{cor:ex-no-tang-velocity}
    Consider a solution $\gamma\colon[0,T)\times[-1,1]\to\R^n$ of \eqref{eq:geo-ev-eq} as in \Cref{thm:existence-of-solution} and let $\varepsilon\in(0,T)$. Then there exists a smooth family of diffeomorphisms $\phi\colon[0,T)\times[-1,1]\to[-1,1]$ such that $\phi(t,\cdot)=\mathrm{id}_{[-1,1]}$ for all $0\leq t<\frac12\varepsilon$ and such that, writing $\tilde\gamma(t,x)=\gamma(t,\phi(t,x))$, $\tilde\gamma$ again solves \eqref{eq:geo-ev-eq} on $[0,T)\times[-1,1]$ and, additionally,
    \begin{equation}
        \partial_t\tilde\gamma(t,x) = -(\nabla\E(\tilde\gamma(t,\cdot))(x)-\lambda(\tilde\gamma(t,\cdot)) \cdot \curv_{\tilde\gamma}(t,x))\quad\text{on  $(\varepsilon,T)\times[-1,1]$}.
    \end{equation}
\end{corollary}

For solutions of \eqref{eq:geo-ev-eq} with zero tangential speed, we study the existence of reparametrizations to solutions of the aforementioned analytic equation. This is the key tool in translating useful well-posedness properties from the analytic equation to the (degenerate parabolic) problem with zero tangential speed.

Combining suitable energy estimates for derivatives of the curvature with the theory of short-time existence and uniqueness, we obtain the following for the asymptotics of the length-preserving elastic flow with orthogonal free boundary conditions on $M$.

\begin{theorem}\label{thm:intro-main-asymptotics}
    Let $\gamma\colon[0,T)\times [-1,1]\to\R^n$ be a smooth maximal solution of
    \begin{equation}\label{eq:ev-eq-intro}
        \begin{cases}
            \partial_t\gamma(t,x) = -(\nabla\E(\gamma(t,\cdot))(x)-\lambda(\gamma(t,\cdot)) \cdot \curv(t,x))&\text{on  $(0,T)\times[-1,1]$}\\
            \gamma(0,\cdot)=\gamma_0&\text{in $[-1,1]$}\\
            \gamma(t,\cdot)\text{ satisfies \eqref{eq:intro-free-bcs}}&\text{for all $0\leq t<T$}
        \end{cases}
    \end{equation}
    with $\lambda(\gamma)$ as in \eqref{eq:def-lambda} such that, for $L_0=\Ll(\gamma_0)$, 
    \begin{equation}\label{eq:intro-unif-non-curvature}
        \inf_{t\in[0,T)} \big( L_0-|\gamma(t,1)-\gamma(t,-1)| \big) > 0.
    \end{equation}    
    Then $T=\infty$ and, for each $\alpha\in(0,\frac12)$ and $t_j\nearrow \infty$, $\gamma(t_j,\cdot)$ converges to a smooth elastica in the $C^{4+\alpha}$-topology after reparametrization, translation, and passing to a subsequence. If $M$ is compact, $\gamma(t_j,\cdot)$ converges to a smooth elastica satisfying \eqref{eq:intro-free-bcs} in the $C^{4+\alpha}$-topology, after passing to a subsequence and reparametrization.
\end{theorem}

\begin{figure}[htb!]
    \vspace*{-2em}
    \centering
    \begin{tikzpicture}

        \node (A) at (0, -1) {\includegraphics[width=6cm]{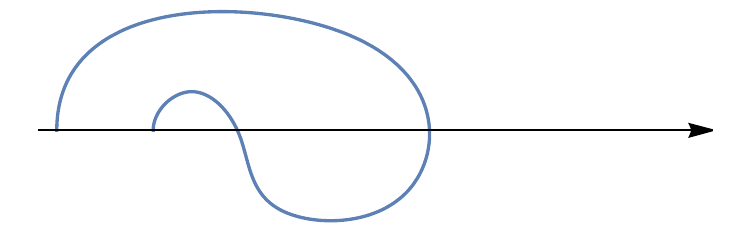}};
        \node (B) at (7.5, 0) {\includegraphics[width=6cm]{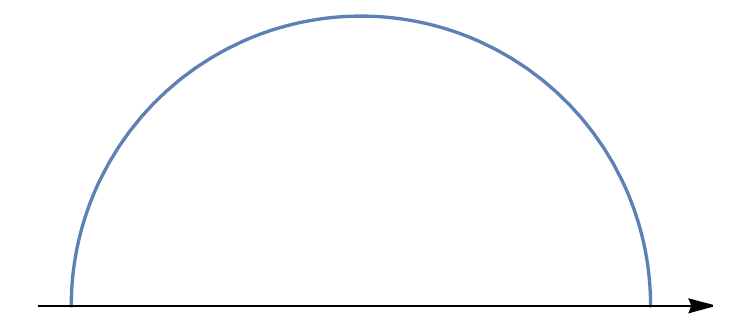}};
        
        \draw[->] (3,-0.25) -- (4.5,-0.25);
        
    \end{tikzpicture}
    \caption{An initial datum (left) and the associated limit (right).}\label{fig:A_examples}
\end{figure}

Moreover, in \Cref{thm:special-case} below, we study the special case where $M=\R\times\{0\}\subseteq\R^2$ is the $x$-axis in the plane. In this case, we always obtain global existence and full convergence for $t\to\infty$ up to reparametrization. Even more, we can classify the possible limits: they are either segments of (multiple covered) round circles or figure-eight elastica (see \cite[Theorem~0.1]{langersinger1985}). An example is illustrated in \Cref{fig:A_examples}, where the flow ``migrates'' in the sense that it is completely contained in the upper half-plane after some finite time.

In general however, the result in \Cref{thm:intro-main-asymptotics} is conditional in the sense that we make the uniform non-flatness assumption \eqref{eq:intro-unif-non-curvature} along a maximal evolution $\gamma$. In fact, this assumption is automatically satisfied in many cases. If $M$ is flat, \eqref{eq:intro-unif-non-curvature} is always true and if $M$ is compact, we show that \eqref{eq:intro-unif-non-curvature} holds if there exists no straight line $\bar\gamma$ of length $L_0$ satisfying \eqref{eq:intro-free-bcs} with $\langle\partial_s\bar\gamma(\pm1),\xi(\bar\gamma(\pm1))\rangle = \langle\partial_s\gamma_0(\pm1),\xi(\gamma_0(\pm1))\rangle$. And such $\bar\gamma$ never exists if $M=\partial\Omega$ for a convex bounded domain $\Omega$, $\xi$ is outward-pointing and $\langle\partial_s\gamma_0(-1),\xi(\gamma_0(-1))\rangle=1$, for example, see \Cref{fig:curv-with-free-bdry-on-sphere}. For an example where $M=\partial\Omega$ with a non-convex domain $\Omega$ where such a straight line $\bar\gamma$ exists only for one critical value $L_0^*(M)>0$ of length, see \Cref{fig:example-nonconvex}.
\begin{figure}[htb!]
    \centering
    \includegraphics[width=0.35\textwidth]{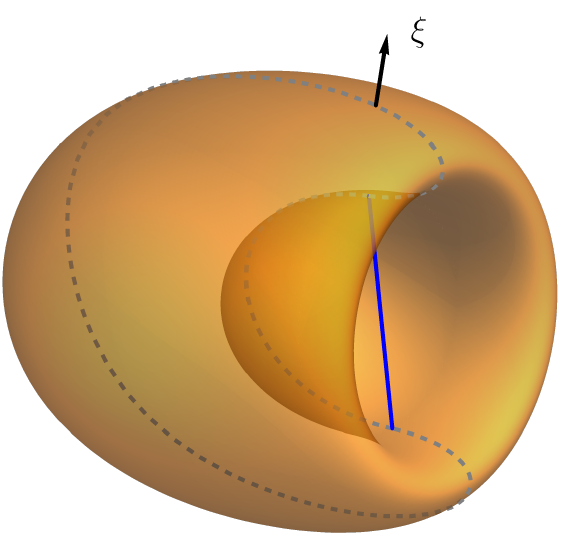}
    \caption{An example where a straight line $\bar\gamma$ of length $L_0$ satisfying \eqref{eq:intro-free-bcs} with $\langle\partial_s\bar\gamma(\pm1),\xi(\bar\gamma(\pm1))\rangle=\mp1$ exists for exactly one length $L_0=L_0^*(M)>0$.}\label{fig:example-nonconvex}
\end{figure}

Next, we comment on our proof strategy for \Cref{thm:intro-main-asymptotics}. As explained earlier, similarly to the analysis of elastic flows with periodic, clamped, or Navier boundary conditions, a result such as \Cref{thm:intro-main-asymptotics} relies on time-uniform energy estimates for derivatives of the curvature. In order to establish the theorem, writing $\zeta$ for the $\inf$ in \eqref{eq:intro-unif-non-curvature}, we prove that
\begin{equation}\label{eq:intro-en-est}
    \|\partial_s^3\curv\|_{L^2(\dd s)} \leq C(L_0,\E(\gamma_0),\zeta,n)\quad\text{for all $0\leq t<T$}.
\end{equation}
Estimate \eqref{eq:intro-en-est} is achieved by carefully choosing a vector field $N$ with the highest-order term $(\partial_s^\bot)^3\curv$ such that $N(t,\pm1)=0$ for all $0\leq t<T$ --- thus allowing integration by parts in terms appearing in $\partial_t\|N\|_{L^2(\dd s)}^2$. Here, only a priori control on the $L^\infty$-norm of the nonlocal Lagrange multiplier $\lambda$ is needed since $N$ can be chosen without any terms involving $\lambda$, i.e., no derivatives of $\lambda$ appear (in contrast to \cite{dallacqualinpozzi2017}).

Note that a similar strategy was successfully applied to prove \emph{global existence} for the evolution of elastic networks in \cite{dallacqualinpozzi2019,garckemenzelpluda2020}. However, with a novel argument, we additionally prove subconvergence to \emph{elastica} in \Cref{thm:intro-main-asymptotics} only using time-uniform bounds on $\|\partial_s^3\curv\|_{L^2(\dd s)}$. While combining \eqref{eq:intro-en-est} with interpolation and Sobolev embeddings clearly allows passing to \emph{some} limit along subsequences of times (up to reparametrization), previously used arguments for proving that this limit is an \emph{elastica} are based on showing $\|\partial_t\gamma\|_{L^2(\dd s)}^2\to 0$ for $t\nearrow \infty$ and require higher-order analogs of \eqref{eq:intro-en-est} (see, for instance, \cite[Proof of Thm.~3.2]{dziukkuwertschaetzle2002}, \cite[p.~6427]{lin2012} or \cite[p.~651]{dallacquapozzi2014}). To this end, we make use of an $L^\infty$-version of \cite[Lemma~4.10]{ruppspener2020} which already relates the $L^2(-1,1)$-norm of the time derivative of the constant-speed reparametrizations of $\gamma$ to $\|\partial_t\gamma\|_{L^2(\dd s)}$.

\subsection{Structure of the article}

The article is structured as follows.
We first review the precise geometric setup and collect some preliminary observations in \Cref{sec:prelim}. Furthermore, we adapt some established shorthand notation and deduce higher-order identities following from \eqref{eq:intro-free-bcs}. In \Cref{sec:well-posedness}, \Cref{thm:existence-of-solution,cor:ex-no-tang-velocity} are proved. Especially, we formulate an analytic version of \eqref{eq:geo-ev-eq}, prove existence and uniqueness in suitable parabolic function spaces for which linear maximal $L^p$-regularity theory is available, and show instantaneous smoothing of the equation. In particular, \Cref{sec:wp-reparametrizations} discusses the existence of reparametrizations relating the (well-posed) analytic problem and \eqref{eq:ev-eq-intro}. \Cref{sec:long-time} starts with a review of available interpolation inequalities and a priori estimates for the nonlocal Lagrange multiplier, explaining the assumption \eqref{eq:intro-unif-non-curvature} and eventually leading to \eqref{eq:intro-en-est}. We then proceed to prove \Cref{thm:intro-main-asymptotics} and, afterwards, give a detailed discussion of \eqref{eq:intro-unif-non-curvature} in \Cref{sec:non-flatness}. Finally, we discuss the aforementioned special planar case where $M$ is flat in \Cref{thm:special-case}.

\section{Preliminaries}\label{sec:prelim}

\subsection{Assumptions on and basic properties of the hypersurface $M$}\label{sec:geom-prelim}

Consider a complete and embedded hypersurface $M\subseteq\R^n$ where $n\geq 2$. For $x\in M$, $T_xM\subseteq\R^n$ denotes the tangent space and $N_xM=(T_xM)^\bot$ is the normal space. There exists a smooth unit normal field 
\begin{equation}\label{eq:defxi}
\xi\colon M\to \S^{n-1}, \quad \xi(x)\in N_xM . 
\end{equation}
Throughout this article, we assume the following hypothesis.

\begin{hypothesis}\label{hyp:unif-tub-nbhd}
    There exists $\delta>0$ such that $\Phi^M\colon M\times (-\delta,\delta)\to B_\delta(M)$, $(x,r)\mapsto x+r\xi(x)$ is a smooth diffeomorphism. Moreover, for each $k\in\N$, 
    \begin{equation}\label{eq:unif-bd-derivatives}
        \sup_{(x,r)\in M\times (-\delta,\delta)} |D^k\Phi^M(x,r)|,\ \sup_{y\in B_\delta(M)} |D^k (\Phi^M)^{-1}(y)| < \infty.
    \end{equation}
\end{hypothesis}

In the following, 
$(\Pi^M,d^M)\vcentcolon=(\Phi^M)^{-1}$, i.e.\ $\Pi^M$ is the nearest point projection in the $\delta$-tubular neighborhood $B_\delta(M)$ of $M$ and $d^M$ denotes the signed distance to $M$.

\begin{remark}
	Examples of hypersurfaces satisfying \Cref{hyp:unif-tub-nbhd} 
    are compact hypersurfaces without boundary or \emph{uniformly embedded} hypersurfaces, see \cite[Def.~2.21, Rem.~2.22 and Thm.~2.31]{Eldering2013}.
\end{remark}
Writing $A$ for the vectorial second fundamental form of $M$, the shape operator $S$ is determined by the choice of the unit normal $\xi$ via
\begin{equation}
  \langle S_pX,Y\rangle = \langle A(X,Y),\xi\rangle\quad\text{for all $p\in M$ and $X,Y\in T_pM$}.
\end{equation} 
\begin{remark}[Extension of functions defined on $M$]
    With a slight abuse of notation, we view all objects defined on $M$ as extended onto $\R^n$ with the following construction. Choose a smooth function $\tilde\eta\in C_c^\infty(\R,[0,1])$ with $\chi_{[-\frac13\delta,\frac13\delta]}\leq \tilde\eta\leq \chi_{(-\frac23\delta,\frac23\delta)}$ and define $\eta\in C^{\infty}(\R^n,[0,1])$ by 
    \begin{equation}
        \eta(x)=\begin{cases}
            \tilde\eta(d^M(x)) &\text{for $x\in B_\delta(M)$}\\
            0&\text{for $x\notin B_\delta(M)$}.
        \end{cases}
    \end{equation}
    Then $\eta\in C^\infty(\R^n,[0,1])$ satisfies
    \begin{equation}\label{eq:bounds-eta}
        \chi_{B_{\frac13\delta}(M)}\leq \eta \leq \chi_{B_{\frac23\delta}(M)}\quad\text{and}\quad \sup_{x\in\R^n} |D^k\eta(x)|<\infty \text{ for all $k\in\N_0$},
    \end{equation}
    using \eqref{eq:unif-bd-derivatives} to derive uniform bounds for $|D^kd^M|$.
    If $V$ is some vector space and $T\colon M\to V$ is smooth, then we extend $T$ to $\R^n$ (without renaming) by 
    \begin{equation}
        T(x)=\begin{cases}
            \eta(x)T(\Pi^M(x))&\text{for $x\in B_\delta(M)$}\\
            0&\text{for $x\notin B_\delta(M)$}.
        \end{cases}
    \end{equation}
    Using \eqref{eq:bounds-eta} and \eqref{eq:unif-bd-derivatives}, this extension satisfies $\|T\|_{C^k(\R^n)}\leq C(k,M) \cdot \|T\|_{C^k(M)}$ for all $k\in\N_0$. 
\end{remark}

With the above definitions, we obtain
\begin{equation}\label{eq:shape-op}
  D_X \xi(p) = - S_pX\quad\text{for all $p\in M$ and $X\in T_pM$}.
\end{equation}

\begin{remark}[Extending the shape operator]\label{rem:extshape}
    With the extension of $\xi$ as constructed in the previous remark, and using \eqref{eq:shape-op}, we extend the shape operator to a function $\R^n\ni x\mapsto S_x\in L(\R^n,\R^n)$,
    \begin{equation}
        S_xX\vcentcolon= -D_X\xi(x)\quad\text{for all $x\in\R^n$ and $X\in\R^n$}.
    \end{equation}
    Using that $\xi=\pm \nabla d^M$ on $M$, \eqref{eq:unif-bd-derivatives} and the previous remark 
    yield for this extension of the shape operator that
    \begin{equation}\label{eq:unif-bound-shape-op}
        \sup_{x\in\R^n} |D^k S_x| < \infty\quad\text{for all $k\in\N_0$}.
    \end{equation}
\end{remark}

\subsection{Evolution equations and natural boundary conditions}

Let $\gamma\colon[-1,1]\to\R^n$ be an immersion. As usual, one denotes by $\partial_s=\frac{1}{|\partial_x\gamma|}\partial_x$ differentiation by arc-length, and by $\dd s=|\partial_x\gamma|\dd x$ integration by the Riemannian measure induced by the metric $\gamma^*\langle\cdot,\cdot\rangle$ on $[-1,1]$. Then $\curv=\partial_s^2\gamma$ is the \emph{curvature} of $\gamma$ and
\begin{equation}
    \E(\gamma) = \int_{-1}^1 |\curv|^2\dd s, \qquad \Ll(\gamma)=\int_{-1}^{1}\dd s
\end{equation}
are its \emph{elastic energy} and length, respectively. 
Furthermore, for any vector field $\Phi\colon[-1,1]\to\R^n$ along $\gamma$, write $\Phi^{\bot}=\Phi-\langle \Phi,\partial_s\gamma\rangle\partial_s\gamma$ for its normal projection along $\gamma$.

We collect here the evolution equations of some geometric quantities, see for instance \cite[Lemma 2.1]{dziukkuwertschaetzle2002}.

\begin{lemma}\label{lem:ev-eq} 
    Let $\gamma\colon[0, T)\times [-1,1] \to \mathbb{R}^{n}$ be a family of immersions and decompose $\partial_t\gamma=V+\varphi\partial_s\gamma$ where $\varphi=\langle\partial_t\gamma,\partial_s\gamma\rangle$ such that $V$ is normal along $\gamma$. Given any smooth normal field $N$ along $\gamma$, the following formulas hold.
    \begin{align}
    \label{a}
    \partial_{t}(\dd s)&=(\partial_s\varphi - \langle \vec{\kappa}, V \rangle) \dd s \\
    \label{b}
    \partial_{t} \partial_{s}- \partial_{s}\partial_{t} &= (\langle \vec{\kappa}, V \rangle-\partial_s\varphi) \partial_{s}\\
    \label{c}
    \partial_{t} \partial_s\gamma &= \partial_{s}^{\perp} V + \varphi\curv\\
    \label{d}
    \partial_{t} N&= \partial_{t}^{\perp} N - \langle \partial_s^{\bot} V+\varphi\curv,N\rangle\partial_s\gamma  \\
    \label{e}
    \partial_{t}^{\perp} \vec{\kappa}& = (\partial_{s}^{\perp})^2 V + \langle \vec{\kappa}, V\rangle  \vec{\kappa} + \varphi\partial_s^\bot\curv \\
    \label{f}
    (\partial_{t}^{\perp}\partial_{s}^{\perp}-\partial_{s}^{\perp}\partial_{t}^{\perp}) N& = (\langle\curv,V\rangle-\partial_s\varphi)\partial_s^{\bot}N+\langle\curv,N\rangle\partial_s^{\bot}V-\langle\partial_s^{\bot}V,N\rangle\curv \, .
    \end{align}
\end{lemma}

Let $\gamma\colon[0,T)\times[-1,1]\to\R^n$ be a family of immersions, $V=(\partial_t\gamma)^{\bot}$ and $\varphi=\langle\partial_t\gamma,\partial_s\gamma\rangle$. Using \eqref{a} and \eqref{e}, we obtain
\begin{align}
  \frac{\dd}{\dd t} \E(\gamma) &= \int_{-1}^{1} \langle 2(\partial_s^{\bot})^2\curv+|\curv|^2\curv,V\rangle\dd s +  \Bigl[ 2\langle \curv,\partial_s^{\bot}V\rangle - 2 \langle \partial_s^{\bot}\curv,V \rangle + \varphi |\curv|^2 \Bigr]_{-1}^1 \\
  &=\vcentcolon \int_{-1}^{1} \langle \nabla \E(\gamma),V\rangle\dd s +  \Bigl[ 2 \langle \curv,\partial_s^{\bot}V\rangle - 2 \langle \partial_s^{\bot}\curv,V \rangle + \varphi |\curv|^2 \Bigr]_{-1}^1.\label{eq:first-var}
\end{align}

We provide some justification for the third-order boundary condition in \eqref{eq:intro-free-bcs}. Consider a smooth family of immersions $\gamma\colon[0,T)\times[-1,1]\to\R^n$ satisfying the following orthogonal free boundary conditions on $M$:
\begin{equation}\label{eq:orth-bc}
    \begin{cases}
        \gamma(t,\pm1)\in M&\text{for all $t\in[0,T)$}\\
        \partial_s\gamma(t,\pm1)\bot T_{\gamma(t,\pm1)}M&\text{for all $t\in[0,T)$}.
    \end{cases}
\end{equation}

Now define $\tau(t,\pm1)=\langle\partial_s\gamma(t,\pm1),\xi(\gamma(t,\pm1))\rangle$ for $0\leq t<T$. Since $M$ has codimension one, \eqref{eq:orth-bc} yields $|\tau|\equiv 1$. In particular, by continuity in $t$, $\tau(t,\pm1)=\tau_0(\pm1)\vcentcolon=\tau(0,\pm1)$ for all $0\leq t<T$, i.e., it is fixed by the initial datum. Then \eqref{eq:orth-bc} yields  
\begin{equation}\label{eq:dsgamma-at-boundary}
    \partial_s\gamma(t,\pm1) = \tau_0(\pm1)\xi(\gamma(t,\pm1))\quad\text{for all  $t\in[0,T)$}.
\end{equation}
Equation \eqref{eq:orth-bc} further implies $\partial_t\gamma(t,\pm1)\in T_{\gamma(t,\pm1)}M$ and thus
\begin{equation}\label{eq:varphi 0 rand}
\varphi(t, \pm1)=\langle \partial_t\gamma(t,\pm1),\partial_s\gamma(t,\pm1)\rangle=0 \quad \text{for all $t\in[0,T)$} .
\end{equation}
Then, differentiating \eqref{eq:dsgamma-at-boundary} in $t$, using \eqref{eq:shape-op} and \eqref{c}, we get 
\begin{equation}\label{eq:rel-V}
  \partial_s^{\bot}V(t,\pm1)=-\tau_0(\pm1)S_{\gamma(t,\pm1)}V.
\end{equation}
Combining \eqref{eq:rel-V} and \eqref{eq:first-var}, one obtains the natural boundary condition 
\begin{equation}\label{eq:nat-bc}
  \partial_s^{\bot}\curv(t,\pm1) = -\tau_0(\pm1)S_{\gamma(t,\pm1)}\curv(t,\pm1)
\end{equation}
associated to the \emph{orthogonal free boundary conditions} \eqref{eq:orth-bc} for the elastic energy.

\begin{definition}
    If $\gamma\colon[0,T)\times[-1,1]\to\R^n$ is a smooth family of immersions with
    \begin{equation}\label{eq:ev-eq}
    \begin{cases}
        \partial_t\gamma(t,x) = -(\nabla\E(\gamma(t,\cdot))(x)-\lambda(\gamma(t,\cdot)) \cdot \curv(t,x))&\text{for all  $(t,x)\in[0,T)\times[-1,1]$}\\
        \gamma(0)=\gamma_0&\text{in $[-1,1]$}\\
        \gamma(t,\pm1)\in M&\text{for all $t\in[0,T)$}\\
        \partial_s\gamma(t,\pm1)\bot T_{\gamma(t,\pm1)}M&\text{for all $t\in[0,T)$}\\
        \partial_s^{\bot}\curv(t,\pm1) = -\tau_0(\pm1)S_{\gamma(t,\pm1)}\curv(t,\pm1) &\text{for all $t\in[0,T)$},
    \end{cases}
    \end{equation}
    then $\gamma$ is called a \emph{fixed-length free-boundary elastic flow} (FLFB-EF) starting at $\gamma_0$. Here, $\lambda(\gamma)$ is given by \eqref{eq:def-lambda} 
    and $\gamma_0\colon[-1,1]\to\R^n$ is a smooth immersion with 
    \begin{align}
        \gamma_0(\pm1)&\in M,\quad \partial_s\gamma_0(\pm1)=\tau_0(\pm1)\xi(\gamma_0(\pm1)),\\
        \partial_s^{\bot}\curv_{\gamma_0}(\pm1) &= -\tau_0(\pm1)S_{\gamma_0(\pm1)}\curv_{\gamma_0}(\pm1)\quad \text{and}\quad\E(\gamma_0)>0.
    \end{align}
\end{definition}

\begin{remark}[On the Lagrange multiplier]
    Integrating by parts in \eqref{eq:def-lambda} , we have
    \begin{equation}\label{eq:lambda-ibp-formula}
        \lambda(\gamma)=\E(\gamma)^{-1} \Big(\int_{-1}^1 -2|\partial_s^\bot\curv|^2+|\curv|^4\dd s + 2\Big[\langle\partial_s^\bot\curv,\curv\rangle\Big]_{-1}^1\Big).
    \end{equation}
    By the third-order boundary condition in \eqref{eq:ev-eq}, the boundary term is quadratic in the curvature. 
    This order-reduction is important  when establishing global existence for \eqref{eq:ev-eq}. 
\end{remark}

\begin{remark}\label{rem:en-mon}
    Consider a FLFB-EF 
    $\gamma\colon[0,T)\times[-1,1]\to \R^n$.  
    Using \eqref{a} and \eqref{eq:def-lambda}, we compute
    \begin{align}
        \frac{\dd}{\dd t} \Ll(\gamma(t)) = - \int_{-1}^1 \langle \curv,\partial_t\gamma\rangle\dd s = \int_{-1}^1\langle \nabla\E(\gamma),\curv\rangle\dd s - \lambda(\gamma) \E(\gamma) = 0.
     \end{align}
     Therefore, by \eqref{eq:first-var}, \eqref{eq:varphi 0 rand}, \eqref{eq:rel-V} and \eqref{eq:nat-bc},
    \begin{align}
        \frac{\dd}{\dd t} \E(\gamma(t)) &= \int_{-1}^1 \langle \nabla \E(\gamma),\partial_t\gamma\rangle \dd s = - \int_{-1}^1 |\partial_t\gamma|^2\dd s + \lambda(\gamma) \int_{-1}^1\langle\curv,\partial_t\gamma\rangle\dd s\\
        &= - \int_{-1}^1 |\partial_t\gamma|^2\dd s - \lambda(\gamma)  \frac{\dd}{\dd t}\Ll(\gamma(t)) = - \int_{-1}^1 |\partial_t\gamma|^2\dd s \leq 0.
    \end{align}
\end{remark}

Finally, the following identity one obtains via integration by parts is crucial in the energy methods used to prove global existence and (sub-) convergence.

\begin{lemma}[{\cite[Lemma~2.3]{dallacquapozzi2014}}]\label{lem:ibp}
    For a family of immersions $\gamma\colon[0,T)\times[-1,1]\to\R^n$ with $\partial_t\gamma=\partial_t^\bot\gamma$ and a vector field $N\colon[0,T)\times[-1,1]\to\R^n$ which is normal along $\gamma$,
    \begin{align}
        \frac{\dd}{\dd t} \frac12\int_{-1}^1|N|^2\dd s + \int_{-1}^1 |(\partial_s^{\bot})^2N|^2\dd s = &\int_{-1}^1 \langle Y,N\rangle\dd s - \frac12\int_{-1}^1 |N|^2\langle\curv, \partial_t\gamma\rangle\dd s\\
        & - \Big[ \langle N,(\partial_s^{\bot})^3N\rangle - \langle \partial_s^{\bot} N,(\partial_s^{\bot})^2N\rangle \Big]_{-1}^1
    \end{align}
    where $Y = \partial_t^{\bot} N + (\partial_s^{\bot})^4N$.
\end{lemma}

\subsection{Preliminaries for the tensor notation used in the energy estimates}


Since maximum principle techniques are not available for higher-order flows, for the analysis of the long-time behavior, we rely on energy methods. In literature, a notation is employed to abbreviate terms appearing in such energy estimates for related flows. This notation is tailored to Gagliardo--Nirenberg-type interpolation estimates, see \cite{dziukkuwertschaetzle2002,lin2012,dallacquapozzi2014}. We slightly adapt this notation to fit into the framework of our problem. 

Consider a family of immersions $\gamma\colon[0,T)\times[-1,1]\to\R^n$ and let $a,b,c\in\N_0$ be integers with $b\geq 1$. Then $P_{b}^{a,c}\colon[-1,1]\to\R^n$ denotes a vector field along $\gamma$ which can be written in the form
\begin{equation}
    T_{\gamma(t,x)}(\partial_s\gamma,\dots,\partial_s\gamma,(\partial_s^{\bot})^{i_1}\curv,\dots,(\partial_s^{\bot})^{i_b}\curv)
\end{equation}
where $\sum_{j=1}^bi_j=a$, $\max_{1\leq j\leq b}i_j=c$ and $T\colon\R^n\to \mathcal{T}^{b+l}_{1}(\R^n)$, $p\mapsto T_p\in\mathcal{T}_1^{b+l}(\R^n)$ is a smooth tensor field on $\R^n$ with uniform bounds
\begin{equation}\label{eq:hyp-tensor-unif-bounds}
    C(P^{a,c}_b,k)\vcentcolon=\sup_{p\in\R^n} |D^kT_p| < \infty\quad\text{for all $k\in\N_0$}.
\end{equation}
Here, $l\in\N_0$ denotes the number of entries of $\partial_s\gamma$ which we do not keep track of in the notation. 
Moreover, for integers $A,B,C\in\N_0$ with $B\geq 1$, $\mathbb{P}_B^{A,C}$ denotes a finite sum of terms of type $P_{b}^{a,c}$ with $a+\frac12 b\leq A+\frac12 B$ and $c\leq C$.

\begin{lemma}\label{lem:tens-alg}
    Consider a FLFB-EF $\gamma\colon[0,T)\times [-1,1]\to\R^n$. Suppose that 
    \begin{equation}
        P_{b}^{a,c}=T_{\gamma(t,x)}(\partial_s\gamma,\dots,\partial_s\gamma,(\partial_s^{\bot})^{i_1}\curv,\dots,(\partial_s^{\bot})^{i_b}\curv)
    \end{equation}
    is normal along $\gamma$ where $T\in\mathcal{T}_1^{b+l}(\R^n)$ is a smooth tensor field on $\R^n$ as above. For any $m\in\N$, we have
    \begin{align}
        (\partial_s^{\bot})^m P_{b}^{a,c} 
        &= \sum_{j=1}^b T_{\gamma(t,x)}(\partial_s\gamma,\dots,\partial_s\gamma,(\partial_s^{\bot})^{i_1+m\delta_{1,j}}\curv,\dots,(\partial_s^{\bot})^{i_b+m\delta_{b,j}}\curv) + \mathbb{P}_{b}^{a+m,c+m-1}\label{eq:tens-alg-1}
    \end{align}
    and
    \begin{align}
        \partial_t^{\bot} P_{b}^{a,c} =\  &
        \sum_{j=1}^b T_{\gamma(t,x)}(\partial_s\gamma,\dots,\partial_s\gamma,(\partial_t^{\bot})^{\delta_{1,j}}(\partial_s^{\bot})^{i_1}\curv,\dots,(\partial_t^{\bot})^{\delta_{b,j}}(\partial_s^{\bot})^{i_b}\curv) \\
        &+ \mathbb{P}_{b+1}^{a+3,\max\{c,3\}} + \lambda \mathbb{P}^{a+1,\max\{c,1\}}_{b+1}.\label{eq:tens-alg-2}
    \end{align}
\end{lemma}
\begin{proof}
    First note that, for a normal vector field $N$ along $\gamma$,
    \begin{equation}\label{eq:tens-alg-3}
        \partial_s^{\bot}N=\partial_s N + \langle N,\curv\rangle \partial_s\gamma.
    \end{equation}
    Moreover, by \cite[Lemma~2.7]{dallacqualinpozzi2017}, for any $k\in\N$,
    \begin{equation}\label{eq:tens-alg-4}
        (\partial_s^{\bot})^k\curv = \partial_s^k\curv + \mathbb P_{2}^{k-1,k-1}.
    \end{equation}
    Thus, also using the product rule for tensors, for $m\in\N$, we obtain
    \begin{align}
        &(\partial_s^{\bot})^m P_{b}^{a,c} = \mathbb{P}_{b+1}^{a+m-1,c+m-1}\\
        &\quad+  \sum_{\sigma_1=1}^b\dots\sum_{\sigma_m=1}^b T_{\gamma(t,x)}(\partial_s\gamma,\dots,\partial_s\gamma,\partial_s^{i_1+\sum_{\nu=1}^m\delta_{1,\sigma_{\nu}}}\curv,\dots,\partial_s^{i_b+\sum_{\nu=1}^m\delta_{b,\sigma_{\nu}}}\curv)\\
        &= \mathbb{P}_{b+1}^{a+m-1,c+m-1}\\
        &\quad+  \sum_{\sigma_1=1}^b\dots\sum_{\sigma_m=1}^b T_{\gamma(t,x)}(\partial_s\gamma,\dots,\partial_s\gamma,(\partial_s^{\bot})^{i_1+\sum_{\nu=1}^m\delta_{1,\sigma_{\nu}}}\curv,\dots,(\partial_s^{\bot})^{i_b+\sum_{\nu=1}^m\delta_{b,\sigma_{\nu}}}\curv).
    \end{align}
    Using \eqref{eq:ev-eq} and \eqref{eq:first-var},
    \begin{equation}\label{eq:DefV}
    V\vcentcolon=\partial_t\gamma=- 2 (\partial_s^{\bot})^{2}\curv -|\curv|^2 \curv + \lambda \curv=\mathbb{P}_1^{2,2}+\lambda P_1^{0,0}.
    \end{equation} 
    Thus, for a normal vector field $N$ along $\gamma$, \eqref{c} and \eqref{d} yield
    \begin{equation}
        \partial_t^{\bot}N=\partial_t N + \langle N,\partial_s^{\bot}V\rangle\partial_s\gamma = \partial_t N + N * (\mathbb P_1^{3,3} + \lambda \mathbb P^{1,1}_1)
    \end{equation}
    where $*$ denotes some constant-coefficient multilinear map. Combined with \eqref{c} and the product rule for tensors, we get \eqref{eq:tens-alg-2}.
\end{proof}

\begin{lemma}\label{lem:control-Y}
    Let $\gamma\colon[0,T)\times[-1,1]\to\R^n$ be a FLFB-EF. Then for $m\in\N_0$ 
    \begin{align}
        \big(\partial_t^{\bot} + (\partial_s^{\bot})^4\big)(\partial_s^{\bot})^m\curv &= \lambda (\partial_s^{\bot})^{m+2}\curv + \mathbb{P}_{3}^{m+2,m+2} + \lambda \mathbb{P}_{3}^{m,m}.
    \end{align}
    Moreover, for any term $P_b^{a,c}$ which is normal along $\gamma$, we have
    \begin{align}
        \big(\partial_t^{\bot} + (\partial_s^{\bot})^4\big)P_{b}^{a,c} &= \mathbb{P}_{b}^{a+4,c+3} + \lambda \mathbb{P}_{b}^{a+2,c+2}.
    \end{align}
\end{lemma}
\begin{proof}
    The first equation is \cite[Lemma~2.5, Equation (2.14)]{dallacqualinpozzi2017}. With \Cref{lem:tens-alg}, writing $P_{b}^{a,c}=T_{\gamma(t,x)}(\partial_s\gamma,\dots,\partial_s\gamma,(\partial_s^{\bot})^{i_1}\curv,\dots,(\partial_s^{\bot})^{i_b}\curv)$,
    \begin{align}
        &\qquad\big(\partial_t^{\bot} + (\partial_s^{\bot})^4\big)P_{b}^{a,c} \\
        &= \sum_{j=1}^b T_{\gamma(t,x)}(\partial_s\gamma,\dots,\partial_s\gamma,\big[\partial_t^{\bot}+(\partial_s^{\bot})^4\big]^{\delta_{1,j}}(\partial_s^{\bot})^{i_1}\curv,\dots,\big[\partial_t^{\bot}+(\partial_s^{\bot})^4\big]^{\delta_{b,j}}(\partial_s^{\bot})^{i_b}\curv) \\
        &\qquad + \mathbb{P}_{b+1}^{a+3,\max\{c,3\}} + \mathbb{P}_{b}^{a+4,c+3} + \lambda \mathbb{P}_{b+1}^{a+1,\max\{c,1\}}.
    \end{align}
    With these computations, the second equation in the claim follows from the first one.
\end{proof}

\subsection{Higher-order identities following from the free boundary conditions}

\begin{lemma}\label{lem:n3}
    Consider a FLFB-EF $\gamma\colon[0,T)\times[-1,1]\to\R^n$. Then
    \begin{equation}
        (\partial_s^{\bot})^3\curv + \tau_0S_{\gamma}(\partial_s^{\bot})^2\curv + \langle \curv,\partial_s^{\bot}\curv\rangle \curv  = 0\quad\text{on $[0,T)\times\{\pm1\}$}.
    \end{equation}
\end{lemma}
\begin{proof}
    The claim follows using \eqref{eq:DefV} and simplifying \eqref{eq:rel-V} accordingly. Indeed, all terms involving $\lambda$ cancel due to \eqref{eq:nat-bc}.
\end{proof}
\begin{lemma}\label{lem:n5}
    Let $\gamma\colon[0,T)\times[-1,1]\to\R^n$ be a FLFB-EF. Then there exists a vector field $N_5\colon[0,T)\times[-1,1]\to\R^n$ which is normal along $\gamma$ with
    \begin{equation}
        N_5 = (\partial_s^{\bot})^5\curv + \mathbb{P}_1^{4,4}+\lambda\mathbb{P}_3^{1,1}
    \end{equation}
    satisfying $N_5(t,x)=0$ for all $0\leq t<T$ and $x\in\{\pm1\}$.
\end{lemma}
\begin{proof}
    Differentiating the boundary condition \eqref{eq:nat-bc} with respect to time yields
    \begin{align}
      \partial_t^{\bot}\partial_s^{\bot}\curv = -\tau_0S_{\gamma}\partial_t^{\bot}\curv-\tau_0DS_{\gamma}(\curv,V)\label{eq:5o-1}.
    \end{align}
    Using \eqref{f} and \eqref{e} (with $\varphi=0$),
    \begin{align}
      \partial_t^{\bot}\partial_s^{\bot}\curv &= \partial_s^{\bot}\Bigl((\partial_s^{\bot})^2V+\langle\curv,V\rangle \curv\Bigr) + \langle \curv,V\rangle \partial_s^{\bot}\curv + |\curv|^2\partial_s^{\bot}V - \langle \partial_s^{\bot}V,\curv\rangle\curv\\
      &=(\partial_s^{\bot})^3V+\langle\partial_s^{\bot}\curv,V\rangle\curv + 2\langle\curv,V\rangle\partial_s^{\bot}\curv + |\curv|^2\partial_s^{\bot}V.
    \end{align}
    By \eqref{e} and \eqref{eq:nat-bc}
    \begin{align}
      -\tau_0S_{\gamma}\partial_t^{\bot}\curv = -\tau_0S_{\gamma}(\partial_s^{\bot})^2 V +\langle\curv,V\rangle\partial_s^{\bot}\curv.
    \end{align}
    So \eqref{eq:5o-1} turns into
    \begin{equation}\label{eq:5o-2}
      (\partial_s^{\bot})^3V+\langle\partial_s^{\bot}\curv,V\rangle\curv + \langle\curv,V\rangle\partial_s^{\bot}\curv + |\curv|^2\partial_s^{\bot}V = -\tau_0S_{\gamma}(\partial_s^{\bot})^2 V -\tau_0DS_{\gamma}(\curv,V)
    \end{equation}
    on $[0,T)\times\{\pm1\}$. In particular, by \eqref{eq:DefV} 
    and using
    \Cref{lem:n3}, we obtain on $[0,T)\times\{\pm1\}$ that
    \begin{equation}\label{eq:5o-3}
        (\partial_s^{\bot})^5\curv = -\tau_0S_{\gamma}(\partial_s^{\bot})^4\curv + \mathbb{P}_3^{3,2} + \lambda \mathbb{P}_{3}^{1,1} = \mathbb{P}_1^{4,4}+\lambda\mathbb{P}_3^{1,1},
    \end{equation}
    so we can take $N_5$ as the difference of the right- and left hand side of \eqref{eq:5o-3}.
\end{proof}

\section{Short-time existence and uniqueness}\label{sec:well-posedness}

In order to establish existence and uniqueness for \eqref{eq:ev-eq}, we consider a related evolution equation and related boundary conditions that are not invariant 
with respect to reparametrizations. This system is then parabolic.

Throughout this section, we sometimes abbreviate $I=[-1,1]$ and work with an integrability parameter $p>5$.

\begin{setting}\label{set:anpbm}
Consider an immersion $\gamma_0\in W^{4(1-\frac1p),p}([-1,1],\R^n)$ with $\gamma_0(\pm1)\in M$. There exist $\delta_0\in(0,\frac13\delta)$ (with $\delta$ as in \Cref{hyp:unif-tub-nbhd}) and smooth frames $(T_i^{\pm1})_{i=1}^{n-1}$ of the tangent bundle of $M$ over $B_{\delta_0}(\gamma_0(\pm1))\cap M$, extended onto $B_{\delta_0}(\gamma_0(\pm1))$ as in \Cref{sec:geom-prelim}. Moreover, there exists $\delta_1>0$ such that any map $\gamma\in C^1([-1,1],\R^n)$ with $\|\gamma-\gamma_0\|_{C^1}<\delta_1$ is an immersion.
\end{setting}

\subsection{Reformulation of the boundary conditions}

A simple computation yields
\begin{equation}\label{eq:dscurv}
    \partial_s^{\bot}\curv = \Big(\frac{\partial_x^3\gamma}{|\partial_x\gamma|^3}-3\frac{\langle \partial_x^2\gamma,\partial_x\gamma\rangle}{|\partial_x\gamma|^5}\partial_x^2\gamma\Big)^\bot.
\end{equation}
This motivates the following.

\begin{definition}
We say that an immersion $\gamma\in C^3([-1,1],\R^n)$ with $|\gamma(\pm1)-\gamma_0(\pm1)|<\delta_0$ satisfies the \emph{analytic free boundary conditions} if 
\begin{equation}\label{eq:ana-fbc}
    \begin{cases}
        d^M(\gamma(\pm1))=0\\
        \langle \partial_x\gamma(\pm1),T_i^{\pm1}(\gamma(\pm1))\rangle =0 \quad\text{for all $1\leq i\leq n-1$}\\
        \frac{\partial_x^3\gamma(\pm1)}{|\partial_x\gamma(\pm1)|^3}-3\frac{\langle \partial_x^2\gamma(\pm1),\partial_x\gamma(\pm1)\rangle}{|\partial_x\gamma(\pm1)|^5}\partial_x^2\gamma(\pm1) + \tau_0(\pm1) S_{\gamma(\pm1)}\curv_{\gamma}(\pm1) = 0.
    \end{cases}
\end{equation}
\end{definition}

\begin{proposition}[Relation between the boundary conditions]\label{lem:geo-vs-ana-fbc}
    Let 
    $\gamma\in W^{4(1-\frac1p),p}(I)$ be an immersion with $|\gamma(\pm1)-\gamma_0(\pm1)|<\delta_0$. We have the following.
    \begin{enumerate}[(a)]
        \item If $\gamma$ satisfies \eqref{eq:ana-fbc}, then also also \eqref{eq:orth-bc},\eqref{eq:nat-bc} (and \eqref{eq:intro-free-bcs}) hold.
        \item If $\gamma$ satisfies \eqref{eq:orth-bc},\eqref{eq:nat-bc}, then there exists a smooth diffeomorphism $\theta\colon[-1,1]\to[-1,1]$ such that the reparametrization $\tilde\gamma=\gamma\circ\theta\in W^{4(1-\frac1p),p}(I)$ of $\gamma$ satisfies the analytic free boundary conditions \eqref{eq:ana-fbc}. Moreover, $\frac12\leq \theta'\leq\frac32$ and, if $A>0$ with 
        \begin{equation}
            \Big|-\frac{\langle\partial_x^3\gamma(\pm1),\partial_x\gamma(\pm1)\rangle}{|\partial_x\gamma(\pm1)|^2} + 3 \frac{\langle\partial_x^2\gamma(\pm1),\partial_x\gamma(\pm1)\rangle^2}{|\partial_x\gamma(\pm1)|^4}\Big| \leq A,
        \end{equation}
        then $\|\partial_x^k\theta\|_{C^0([-1,1])}\leq C(k,A)$ for all $k\in\N$ with $k\geq 2$.
    \end{enumerate}
\end{proposition}
In the proof of \Cref{lem:geo-vs-ana-fbc}(b), we use the following elementary construction for $\theta$, also see \cite[Lemma~5.2]{garckemenzelpluda2020}.
\begin{lemma}\label{lem:phi-constr}
    Let $A>0$. For all $a_{\pm1}\in\R$ with $|a_{-1}|,|a_1|\leq A$, there exists a smooth diffeomorphism $\theta\colon[-1,1]\to[-1,1]$ satisfying $\frac12\leq\theta'\leq\frac32$, 
    \begin{equation}\label{eq:phi-constr-0}
        \theta(\pm1)=\pm1,\quad\theta'(\pm1)=1,\quad\theta''(\pm1)=0\quad \theta'''(\pm1)=a_{\pm1}
    \end{equation} 
    as well as $\|\partial_x^k\theta\|_{C^0([-1,1])}\leq C(k,A)$ for all $k\in\N$.
\end{lemma}
\begin{proof}
    Consider the third-order polynomials $p_{\pm1}(x)=\pm 1+(x-(\pm1))+\frac16 a_{\pm1}(x-(\pm1))^3$ and note that
    \begin{equation}
        p_{\pm1}(\pm1)=\pm1,\quad p_{\pm1}'(\pm1)=1,\quad p_{\pm1}''(\pm1)=0\quad p_{\pm1}'''(\pm1)=a_{\pm1}.
    \end{equation} 
    Moreover, choosing 
    \begin{equation}\label{eq:phi-constr-1}
        \delta=\delta(A) = \min\Big\{\frac{1}{6},\frac{1}{\sqrt{2 A}}\Big\},
    \end{equation}
    we have $|p_{\pm1}'(x)-1|\leq \frac{1}{4}$ and $|p_{\pm1}(x)-(\pm1)| < \frac14$ for all $|x-(\pm1)|\leq \delta$. Define $f\colon (-1-\delta,1+\delta)\to\R$,
    \begin{equation}
        f(x)=\begin{cases}
            p_{\pm1}(x)&\text{for $|x-(\pm1)|\leq \delta$}\\
            p_{-1}(-1+\delta)\frac{(1-\delta)-x}{2-2\delta} + p_1(1-\delta)\frac{x-(-1+\delta)}{2-2\delta} &\text{for $x\in(-1+\delta,1-\delta)$}.
        \end{cases}
    \end{equation}
    Notice that, using $\frac32\leq|p_{-1}(-1+\delta)-p_1(1-\delta)|\leq 2$ and $\delta\leq\frac16$, 
    \begin{equation}\label{eq:phi-constr-3}
        \frac34\leq f'\leq \frac54\quad \text{a.e.\ on $(-1-\delta,1+\delta)$}.
    \end{equation}
    Denote by $\psi_{\varepsilon}(x)=\varepsilon^{-1}\psi_1(x/\varepsilon)$ the standard Dirac-sequence and let $f_\varepsilon = f * \psi_\varepsilon$. Thus, we get for $0<\varepsilon<\delta$ and $x\in[-1,1]$
    \begin{equation}\label{eq:phi-constr-2}
        |f_\varepsilon(x)-f(x)|\leq \int_{x-\varepsilon}^{x+\varepsilon} \psi_{\varepsilon}(x-y)|f(y)-f(x)|\dd y \leq \frac54\varepsilon \int_{-\varepsilon}^{\varepsilon} \psi_\varepsilon(y)\dd y = \frac54\varepsilon.
    \end{equation}
    Next choose $\eta\in C^\infty(\R,\R)$ with $\chi_{[-1+\frac23\delta,1-\frac23\delta]}\leq\eta\leq \chi_{[-1+\frac13\delta,1-\frac13\delta]}$ and $|\partial_x^k\eta|\leq C(k)\delta^{-k}$ for all $k\in\N$ and fix $\varepsilon_0=\varepsilon_0(A)$ given by $\varepsilon_0 = \frac{\delta}{5 C(1)}$, cf.\ \eqref{eq:phi-constr-1}. Finally, define
    \begin{equation}
        \theta(x)=(1-\eta(x))f(x)+\eta(x) f_{\varepsilon_0}(x).
    \end{equation}
    By construction, $\theta$ is smooth and \eqref{eq:phi-constr-0} is satisfied. Moreover, using $f\in W^{1,\infty}((-1-\delta,1+\delta))$ and \eqref{eq:phi-constr-3}, for $x\in[-1,1]$, we have
    \begin{equation}
        f_{\varepsilon_0}'(x)=\int_{x-\varepsilon_0}^{x+\varepsilon_0}\psi_{\varepsilon_0}(x-y)\partial_yf(y)\dd y \in \Big[\frac34,\frac54\Big].
    \end{equation}
    Using $\theta'=(1-\eta)f'+\eta f_{\varepsilon_0}' + \eta'(f_{\varepsilon_0}-f)$, \eqref{eq:phi-constr-2}, $|\eta'|\leq C(1)/\delta$ and the choice of $\varepsilon_0$, we get
    \begin{equation}
        \theta' \geq \frac34 - \frac{5}{4}\varepsilon_0\frac{C(1)}{\delta} = \frac12\quad\text{and}\quad \theta'\leq \frac54 + \frac{5}{4}\varepsilon_0\frac{C(1)}{\delta} = \frac32.
    \end{equation}
    In particular, $\theta(\pm1)=\pm 1$ yields that $\theta\colon[-1,1]\to[-1,1]$ is a (smooth) diffeomorphism. The remainder of the claim now follows using $\partial_x^k\theta = \int_{x-\varepsilon_0}^{x+\varepsilon_0}\partial_x^k\psi_{\varepsilon_0}(x-y)f(y)\dd y$ and $\|\partial_x^k\psi_{\varepsilon_0}\|_{C^0}\leq C(k,A)$ for all $k\geq 2$.
\end{proof} 
\begin{proof}[Proof of \Cref{lem:geo-vs-ana-fbc}]
    Clearly, \eqref{eq:orth-bc} is equivalent to the first two conditions in \eqref{eq:ana-fbc} by the choice of the frame in \Cref{set:anpbm}. So we only need to check \emph{(a)} and \emph{(b)} for the third-order conditions. First suppose that $\gamma$ satisfies \eqref{eq:ana-fbc}. Then $\partial_s\gamma(\pm1)\bot T_{\gamma(\pm1)}M$ so that $\curv(\pm1)\in T_{\gamma(\pm1)}M$. Thus also $S_{\gamma(\pm1)}\curv(\pm1)\in T_{\gamma(\pm1)}M$ which yields that 
    \begin{equation}
        \frac{\partial_x^3\gamma(\pm1)}{|\partial_x\gamma(\pm1)|^3}-3\frac{\langle \partial_x^2\gamma(\pm1),\partial_x\gamma(\pm1)\rangle}{|\partial_x\gamma(\pm1)|^5}\partial_x^2\gamma(\pm1) = - \tau_0(\pm1) S_{\gamma(\pm1)}\curv(\pm1) \bot \partial_s\gamma(\pm1).
    \end{equation}
    Combined with \eqref{eq:dscurv}, this implies \eqref{eq:nat-bc}. 
    
    Conversely, suppose that \eqref{eq:orth-bc},\eqref{eq:nat-bc} hold. By \Cref{lem:phi-constr}, there exists a smooth diffeomorphism $\theta\colon[-1,1]\to[-1,1]$ satisfying \eqref{eq:phi-constr-0} where we choose
    \begin{align}
        a_{\pm1}=\theta'''(\pm1) &= -\frac{\langle\partial_x^3\gamma,\partial_x\gamma\rangle}{|\partial_x\gamma|^2} + 3 \frac{\langle\partial_x^2\gamma,\partial_x\gamma\rangle^2}{|\partial_x\gamma|^4}\Big|_{y=\pm1} .
    \end{align}
    The diffeomorphism satisfies $\frac12\leq \theta'\leq\frac32$ on $[-1,1]$ as well as $\|\partial_x^k\theta\|_{C^0}\leq C(k,A)$ for all $k\geq 2$. Then we compute for $\tilde\gamma=\gamma\circ\theta$ that
    \begin{align}
        \langle \frac{\partial_x^3\tilde\gamma}{|\partial_x\tilde\gamma|^3} - 3 \frac{\langle\partial_x^2\tilde\gamma,\partial_x\tilde\gamma\rangle}{|\partial_x\tilde\gamma|^5}\partial_x^2\tilde\gamma,\partial_x\tilde\gamma \rangle\Big|_{y=\pm1} = 0,
    \end{align}
    Since $\tilde\gamma$ satisfies the third-order condition \eqref{eq:nat-bc} by its invariance with respect to reparametrizations, \eqref{eq:dscurv} therefore yields that $\tilde\gamma$ also satisfies \eqref{eq:ana-fbc}.
    %
    %
    %
\end{proof}

\subsection{The analytic problem}

For an arbitrary $T>0$, consider
\begin{align}
    \mathbb X_{T,p} &= W^{1,p}(0,T;L^p([-1,1],\R^n))\cap L^p(0,T;W^{4,p}([-1,1],\R^n)),\\
    \prescript{}{0}{\mathbb X}_{T,p} &= \{\varphi\in \mathbb X_{T,p}\mid \mathrm{tr}_{t=0}(\varphi)=0\}.
\end{align}
More details on these spaces as well as useful mapping properties and embeddings are collected in \Cref{app:fct-spaces}. In particular, $\mathrm{tr}_{t=0}(\varphi) \in W^{4(1-\frac1p),p}([-1,1],\R^n)$ for $\varphi \in \mathbb X_{T,p}$ by \Cref{prop:xtp-mappings-and-embeddings}\eqref{it:buc-emb} and \eqref{eq:SobSlobSpaces}.

Let $\gamma_0\in W^{4(1-\frac1p),p}([-1,1],\R^n)$ be an immersion with $\E(\gamma_0)>0$ and $T_1>0$. Using $p>5$, by \cite[Theorem~2.1]{denkhieberpruess2007} there exists $\bar\gamma\in\X_{T_1,p}$ solving
\begin{equation}\label{eq:def-bargamma}
    \begin{cases}
        \partial_t\bar\gamma+\partial_x^4\bar\gamma = 0&\text{in $L^p((0,T_1)\times (-1,1))$}\\
        \mathrm{tr}_{t=0}\bar\gamma = \gamma_0&\text{in $W^{4(1-\frac1p),p}([-1,1])$}\\
        \mathrm{tr}_{\partial I}\bar\gamma(t,\cdot) = \mathrm{tr}_{\partial I}\gamma_0&\text{in $W^{1-\frac{1}{4p},p}(0,T_1;\R^n\times\R^n)$}\\
        \mathrm{tr}_{\partial I}\partial_x^3\bar\gamma(t,\cdot) = \mathrm{tr}_{\partial I}\partial_x^3\gamma_0&\text{in $W^{\frac14-\frac{1}{4p},p}(0,T_1;\R^n\times\R^n)$}
    \end{cases}
\end{equation}
such that, for a constant $C=C(T_1,p)$,
\begin{equation}\label{eq:bargamma-gamma0-estimate}
    \|\bar\gamma\|_{\X_{T_1,p}}\leq C(T_1,p)\cdot \|\gamma_0\|_{W^{4(1-\frac1p),p}([-1,1])}.
\end{equation} 
Proceeding similarly as in the proof of \Cref{prop:wp-lin-system} in \Cref{app:wp-lin-system}, one can verify that the boundary conditions in \eqref{eq:def-bargamma} satisfy the Lopatinskii--Shapiro condition. 

\begin{notation}\label{not:wp-gamma0}
    We write $T_0=T_0(\gamma_0)$ for the maximal time in $(0,T_1]$ such that $\bar\gamma(t,\cdot)$ satisfies
    \begin{equation}
        |\partial_x\bar\gamma(t,x)|\geq \frac12\min_I |\partial_x\gamma_0|\quad\text{for all $0\leq t\leq T_0$ and $x\in I$}
    \end{equation}
    and $\E(\bar\gamma(t,\cdot))>0$ for all $0\leq t<T_0(\gamma_0)$. Moreover, we denote $E\gamma_0\vcentcolon=\bar\gamma|_{[0,T_0]\times I}$.
\end{notation}
For $0<T<T_0$ and $\varphi\in \prescript{}{0}{\mathbb X}_{T,p}$, define $\gamma_\varphi=\bar\gamma|_{[0,T]\times I}+\varphi$. Then $\gamma_\varphi\in\mathbb X_{T,p}$, $\mathrm{tr}_{t=0}(\gamma_\varphi)=\gamma_0$. With $\delta_0$, $\delta_1$ determined by $\gamma_0$ as in \Cref{set:anpbm}, consider 
\begin{align}\label{eq:SetU}
    \mathcal U^{\gamma_0}_{T,p} = \{\varphi \in \prescript{}{0}{\mathbb X}_{T,p}\mid\ &|\gamma_\varphi(t,\pm1)-\gamma_0(\pm1)|<\delta_0,\ \E(\gamma_\varphi(t,\cdot))>0 \text{ and }\\
    &\|\gamma_\varphi(t,\cdot) -\gamma_0\|_{C^1([-1,1])}<\delta_1\text{ for all $0\leq t\leq T$}\},
\end{align}
an open subset of $\prescript{}{0}{\mathbb X}_{T,p}$.

\begin{definition}\label{def:ana-prob}
    Suppose $\gamma_0\in W^{4(1-\frac1p),p}(I)$ is an immersion with $\E(\gamma_0)>0$ satisfying \eqref{eq:ana-fbc}. Then, for $0<T < T_0(\gamma_0)$, $\gamma\in\mathbb X_{T,p}$ is said to solve the \emph{analytic problem} with initial datum $\gamma_0$ if $\gamma-(E\gamma_0)|_{[0,T]\times I}\in \mathcal U^{\gamma_0}_{T,p}$ and 
    \begin{equation}\label{eq:ana-prob}
        \begin{cases}
            \partial_t\gamma + \mathcal A(\gamma) - \lambda(\gamma)\curv_{\gamma}= 0&\text{in $L^p((0,T)\times (-1,1))$}\\
            d^M(\gamma(t,\pm1))=0&\text{for all $0\leq t<T$}\\
            \langle \partial_x\gamma(t,\pm1),T_i^{\pm1}(\gamma(t,\pm1))\rangle =0 &\text{for all $0\leq t<T$, $1\leq i\leq n-1$}\\
            \frac{\partial_x^3\gamma(t,\pm1)}{|\partial_x\gamma(t,\pm1)|^3}-3\frac{\langle \partial_x^2\gamma(t,\pm1),\partial_x\gamma(t,\pm1)\rangle}{|\partial_x\gamma(t,\pm1)|^5}\partial_x^2\gamma(t,\pm1) \\
            \quad+ \tau_0(\pm1) S_{\gamma(t,\pm1)}\curv_{\gamma}(t,\pm1) = 0&\text{for all $0\leq t<T$}
        \end{cases}
    \end{equation}
    where $\lambda$ is given by \eqref{eq:lambda-ibp-formula} and 
    \begin{equation}
    \mathcal A(\gamma) = 2\frac{\partial_x^4\gamma}{|\partial_x\gamma|^4} - 12 \frac{\langle \partial_x^2\gamma,\partial_x\gamma\rangle}{|\partial_x\gamma|^6}\partial_x^3\gamma - 8 \frac{\langle \partial_x^3\gamma,\partial_x\gamma\rangle}{|\partial_x\gamma|^6}\partial_x^2\gamma  -5\frac{|\partial_x^2\gamma|^2}{|\partial_x\gamma|^6}\partial_x^2\gamma + 35 \frac{\langle\partial_x^2\gamma,\partial_x\gamma\rangle^2}{|\partial_x\gamma|^8}\partial_x^2\gamma.
\end{equation}
\end{definition}

The differential operator $\mathcal{A}$ defined above satisfies $\nabla\E(\gamma) = (\mathcal A(\gamma))^\bot$, see  \cite[Proposition~A.1(iv)]{ruppspener2020}.

\begin{remark}\label{rem:sol-of-anaprob-is-sol-of-geo-prob}
    Let $\gamma$ solve the analytic problem with initial datum $\gamma_0$ as in \Cref{def:ana-prob}. Then $\gamma$ solves \eqref{eq:geo-ev-eq}. Indeed, the fact that \eqref{eq:intro-free-bcs} follows from \eqref{eq:ana-fbc} is already proved in \Cref{lem:geo-vs-ana-fbc}. Moreover, the choice of $\mathcal A$ yields that
    \begin{align}
        \partial_t^\bot\gamma = -\big(\mathcal (A(\gamma))^\bot -\lambda(\gamma)\curv_\gamma\big) = -\big(\nabla\E(\gamma)-\lambda(\gamma)\curv_\gamma\big).
    \end{align}
\end{remark}

\subsection{Existence and uniqueness of solutions to the analytic problem}

In the following, $p>5$ is considered to be fixed and dependencies of constants on $p$ are not explicitly mentioned. We prove here the following result.

\begin{theorem}\label{thm:wp-ana-prob}
    Let $p>5$ and $\gamma_0\in W^{4(1-\frac1p),p}(I,\R^n)$ be an immersion satisfying  $\E(\gamma_0)>0$ and \eqref{eq:ana-fbc}. There exist $T=T(\gamma_0)>0$ and a unique $\gamma\in \mathbb X_{T,p}$ solving the analytic problem with initial datum $\gamma_0$. Moreover, the time $T(\gamma_0)$ depends continuously on $\gamma_0$ in $W^{4(1-\frac1p),p}(I)$ in the sense of \Cref{def:cts-dependence-on-gamma0} below.
\end{theorem}
\begin{definition}\label{def:cts-dependence-on-gamma0}
    In the following, we say that a constant $C$ in an inequality \emph{depends continuously on $\gamma_0$ in $W^{4(1-\frac1p),p}(I)$} if, for all $\delta>0$, there exists $\zeta>0$ such that, for all $\tilde\gamma_0\in W^{4(1-\frac1p),p}(I)$ satisfying \eqref{eq:ana-fbc} and $\|\tilde\gamma_0-\gamma_0\|_{W^{4(1-\frac1p),p}(I)}<\zeta$, the inequality is also satisfied with $\gamma_0$ replaced by $\tilde\gamma_0$ and $C$ replaced by $C+\delta$. 

    Similarly, we say that a time-parameter $\bar T=\bar T(\gamma_0)>0$ \emph{depends continuously on $\gamma_0$ in $W^{4(1-\frac1p),p}(I)$} if, for any $0<\delta<\bar T(\gamma_0)$, there exists $\zeta>0$ such that, for any $\tilde\gamma_0\in W^{4(1-\frac1p),p}(I)$ satisfying \eqref{eq:ana-fbc} with $\|\tilde\gamma_0-\gamma_0\|_{W^{4(1-\frac1p),p}(I)}<\zeta$, one can choose $\bar T(\tilde\gamma_0)\geq \bar{T}(\gamma_0)-\delta$.
\end{definition}

Consider $T_1\in(0,\infty)$, 
fix an immersion $\gamma_0\in W^{4(1-\frac1p),p}([-1,1])$ as in \Cref{thm:wp-ana-prob} 
and 
$T_0=T_0(\gamma_0)$ and $\bar{\gamma}=E\gamma_0$ as in \eqref{eq:def-bargamma} and \Cref{not:wp-gamma0}. Recall that $\bar\gamma(t,\cdot)$ is an immersion with $\E(\bar\gamma(t,\cdot))>0$ for all $0\leq t < T_0$. Finally, for $0<T<T_0$ and $\varphi\in\mathcal U^{\gamma_0}_{T,p}$, write $\gamma_\varphi\vcentcolon=E\gamma_0|_{[0,T]\times I}+\varphi$.

With the notation and spaces introduced above (see also 
\Cref{app:fct-spaces} and in particular \Cref{prop:xtp-mappings-and-embeddings}), consider the operator $\mathcal T \vcentcolon= \mathcal T_{T}$ mapping
\begin{align} \label{eq:Ftp}
    \mathcal U^{\gamma_0}_{T,p} 
    \to\prescript{}{0}{\mathbb F}_{T,p}\vcentcolon= \ & L^p((0,T)\times I,\R^n) 
    \times \prescript{}{0}{W}^{1-\frac{1}{4p},p}(0,T;\R\times\R)\\
    &\quad\times
        \prescript{}{0}{W}^{\frac34-\frac{1}{4p},p}(0,T;\R^{n-1}\times \R^{n-1}) \times \prescript{}{0}{W}^{\frac14-\frac{1}{4p},p}(0,T;\R^n\times\R^n)
\end{align}
given by
\begin{align}
    \mathcal T(\varphi) = \Big( &\partial_t\gamma_\varphi + \A(\gamma_\varphi) - \lambda(\gamma_\varphi)\curv_{\gamma_\varphi},\mathcal B_0(\varphi),\mathcal B_1(\varphi),\mathcal B_3(\varphi)\Big)
\end{align}
where 
\begin{align}
    \mathcal{B}_0(\varphi)&=\mathrm{tr}_{\partial I} (d^M\circ\gamma_\varphi),\quad \mathcal{B}_1(\varphi) = \begin{pmatrix}
        \langle T_1^{\pm1}\circ(\mathrm{tr}_{\partial I}\gamma_\varphi),\mathrm{tr}_{\partial I}\partial_x\gamma_\varphi\rangle\\
        \vdots\\
        \langle T_{n-1}^{\pm1}\circ(\mathrm{tr}_{\partial I}\gamma_\varphi),\mathrm{tr}_{\partial I}\partial_x\gamma_\varphi\rangle
    \end{pmatrix},\\
    \mathcal B_3(\varphi) &= \mathrm{tr}_{\partial I} \Big[ \frac{\partial_x^3\gamma_\varphi}{|\partial_x\gamma_\varphi|^3}-3\frac{\langle \partial_x^2\gamma_\varphi,\partial_x\gamma_\varphi\rangle}{|\partial_x\gamma_\varphi|^5}\partial_x^2\gamma_\varphi  + \tau_0 S_{\gamma_\varphi}\curv_{\gamma_\varphi} \Big].
\end{align}
We first consider the Lagrange multiplier.

\begin{lemma}\label{lem:tilde-lambda-bounds}
    For any $\sigma\in (\frac1p,\frac14-\frac{1}{4p})$, we have that $\Lambda\colon \mathcal U^{\gamma_0}_{T,p}\to C^{\sigma-\frac1p}([0,T],\R)$,
    \begin{equation}
        \Lambda(\varphi)|_t = \E(\gamma_\varphi(t,\cdot))^{-1}\Big(\int_{-1}^1 -2|\partial_s^\bot\curv_{\gamma_\varphi}|^2+|\curv_{\gamma_\varphi}|^4\dd s\Big|_t + 2\Big[\langle\partial_s^\bot\curv_{\gamma_\varphi}(t,y),\curv_{\gamma_\varphi}(t,y)\rangle\Big]_{y=-1}^1\Big) 
    \end{equation}
    is $C^1$. Moreover, there is a constant $C=C(\gamma_0,T_0)$ such that $\|\Lambda'(0)\|_{L(\prescript{}{0}{\mathbb X}_{T,p},C^0([0,T],\R))}\leq C(T_0,\gamma_0)$ 
    and $C(T_0,\gamma_0)$ depends continuously on $\gamma_0$ in $W^{4(1-\frac1p),p}(I)$.
\end{lemma}
\begin{proof}
Since $T < T_0(\gamma_0)$, by \Cref{prop:xtp-mappings-and-embeddings}\eqref{it:hoeld-emb}, we have that $\Lambda$ is well-defined. There exist smooth $\R^{1\times n}$-valued polynomials $l_{0,i}$ such that 
\begin{equation}
    \partial_{t,0}\E(\bar\gamma+t\varphi)=\int_{-1}^1 \sum_{i=1}^2 l_{0,i}(|\partial_x\bar\gamma(t,x)|^{-1},\partial_x\bar\gamma(t,x),\partial_x^2\bar\gamma(t,x)) \partial_x^i\varphi(t,x) \dd x.
\end{equation}
Thus, we obtain
    \begin{align}
        \Lambda'(0)\varphi|_t &= -\frac{\Lambda(0)}{\E(\bar\gamma(t,\cdot))} \int_{-1}^1 \sum_{i=1}^2 l_{0,i}(|\partial_x\bar\gamma(t,x)|^{-1},\partial_x\bar\gamma(t,x),\partial_x^2\bar\gamma(t,x)) \partial_x^i\varphi(t,x) \dd x\\
        &+\E(\bar\gamma(t,\cdot))^{-1}\int_{-1}^1 \sum_{i=1}^3 l_{1,i}(|\partial_x\bar\gamma(t,x)|^{-1},\partial_x\bar\gamma(t,x),\partial_x^2\bar\gamma(t,x),\partial_x^3\bar\gamma(t,x)) \partial_x^i\varphi(t,x) \dd x \\
        &+ \E(\bar\gamma(t,\cdot))^{-1}\sum_{i=1}^3 l_{2,i}(|\partial_x\bar\gamma(t,y)|^{-1},\partial_x\bar\gamma(t,y),\partial_x^2\bar\gamma(t,y),\partial_x^3\bar\gamma(t,y)) \partial_x^i\varphi(t,y)\Big|_{y=-1}^1
    \end{align}
    for smooth $\R^{1\times n}$-valued polynomials $l_{1,i}$ and $l_{2,i}$. A similar computation yields that $\Lambda$ is $C^1$. Note that $\bar\gamma\in C^{\sigma-\frac1p}([0,T_0],C^{3+(1-4\sigma-\frac1p)}([-1,1],\R^n))$ by \Cref{prop:xtp-mappings-and-embeddings} and  $|\partial_x\bar\gamma|\geq\frac12\min_I|\partial_x\gamma_0|$ by \Cref{not:wp-gamma0}. 
    
    The claim now follows from the above since the embedding-constant of $$\prescript{}{0}{\mathbb X}_{T,p}\hookrightarrow C^{\sigma-\frac1p}([0,T],C^{3+(1-4\sigma-\frac1p)}([-1,1],\R^n))$$ is $T$-independent for $T\leq T_0$, see \Cref{prop:xtp-mappings-and-embeddings}\eqref{it:hoeld-emb}. Hereby the continous dependence of the constant on $\gamma_0$ in $W^{4(1-\frac1p),p}(I)$ is due to \eqref{eq:bargamma-gamma0-estimate} applied to $E\gamma_0-E\tilde\gamma_0$.
\end{proof}

The following lemma summarizes some simple computations.

\begin{lemma}\label{lem:linearizations-of-analytic-system}
    The operators $\mathcal T$, $\mathcal B_0$, $\mathcal B_1$ and $\mathcal B_3$ are $C^1$ and, for $\varphi\in\prescript{}{0}{\mathbb X}_{T,p}$, we have
    \begin{align}
        \mathcal T'(0)\varphi = \Big( &\partial_t\varphi + \frac2{|\partial_x\bar\gamma|^4}\partial_x^4\varphi + \sum_{i=1}^3 a_{i}(|\partial_x\bar\gamma|^{-1},\partial_x\bar\gamma,\dots,\partial_x^3\bar\gamma)\partial_x^i\varphi \\
        &- \frac{8\partial_x^4\bar\gamma}{|\partial_x\bar\gamma|^6}\langle\partial_x\bar\gamma,\partial_x\varphi\rangle  - \Lambda(0) \cdot \sum_{i=1}^2b_{i}(|\partial_x\bar\gamma|^{-1},\partial_x\bar\gamma,\partial_x^2\bar\gamma)\partial_x^i\varphi  - \Lambda'(0)\varphi\cdot \curv_{\bar\gamma},\\
        & \mathcal B_0'(0)\varphi, \mathcal B_1'(0)\varphi, \mathcal{B}_3'(0)\varphi\Big)
    \end{align}
    where the $a_i$, $b_i$ are polynomials 
    with values in $\R^{n\times n}$ respectively. Moreover,
    \begin{equation}
        \mathcal{B}_0'(0)\varphi=\mathrm{tr}_{\partial I}\langle\xi\circ\bar\gamma,\varphi\rangle, \quad \mathcal{B}_1'(0)\varphi =  \begin{pmatrix} 
            \langle T_1^{\pm1}(\mathrm{tr}_{\partial I}\bar\gamma),\mathrm{tr}_{\partial I}\partial_x\varphi\rangle+\mathrm{tr}_{\partial I} [c_1\circ \bar\gamma \cdot \varphi]\\
            \vdots\\
            \langle T_{n-1}^{\pm1}(\mathrm{tr}_{\partial I}\bar\gamma),\mathrm{tr}_{\partial I}\partial_x\varphi\rangle+\mathrm{tr}_{\partial I} [c_{n-1}\circ \bar\gamma \cdot \varphi]
        \end{pmatrix} 
    \end{equation}
    where $c=(c_1,\dots,c_{n-1})$ is a smooth $\R^{{n-1}\times n}$-valued function, and finally,
    \begin{equation}
        \mathcal{B}_3'(0)\varphi = \mathrm{tr}_{\partial I}\Big[ \frac{\partial_x^3\varphi}{|\partial_x\bar\gamma|^3} + \sum_{i=0}^2d_{i,\bar\gamma}(|\partial_x\bar\gamma|^{-1},\partial_x\bar\gamma,\dots,\partial_x^3\bar\gamma)\partial_x^i\varphi\Big]
    \end{equation}
    where the $d_{i,\cdot}$ are smooth families of polynomials with values in $\R^{n\times n}$.
\end{lemma}

Motivated by these computations, we consider a suitable linearization of the analytic problem and show its well-posedness. The motivation for the particular choice of linearization and the choice of boundary conditions for $\bar\gamma$ in \eqref{eq:def-bargamma} becomes apparent when studying the well-posedness proof in detail.

\begin{proposition}\label{prop:wp-lin-system}
    Let $T\in(0,T_0)$. For all $f\in L^p((0,T)\times I)$, $h_0\in \prescript{}{0}{W}^{1-\frac{1}{4p},p}(0,T;\R\times\R)$, $h_1\in \prescript{}{0}{W}^{\frac34-\frac{1}{4p},p}(0,T;\R^{n-1}\times \R^{n-1})$ and $h_3\in \prescript{}{0}{W}^{\frac14-\frac{1}{4p},p}(0,T;\R^n\times\R^n)$, there exists a unique solution $\varphi\in \prescript{}{0}{\mathbb X}_{T,p}$ of
    \begin{equation}\label{eq:lin-system}
        \begin{cases}
            \partial_t\varphi + \frac2{|\partial_x\bar\gamma|^4}\partial_x^4\varphi + \sum_{i=1}^3 a_{i}(|\partial_x\bar\gamma|^{-1},\partial_x\bar\gamma,\dots,\partial_x^3\bar\gamma)\partial_x^i\varphi \\ \ - \Lambda(0) \cdot \sum_{i=1}^2b_{i}(|\partial_x\bar\gamma|^{-1},\partial_x\bar\gamma,\partial_x^2\bar\gamma)\partial_x^i\varphi = f 
            &\text{in $L^p((0,T)\times I)$}\\
            \mathcal{B}_j'(0)\varphi = h_j\text{ for $j=0,1,3$}&\text{on $[0,T]$}.
        \end{cases}
    \end{equation}
    Moreover, for a constant $C=C(T_0,\gamma_0)$,
    \begin{align}
        \|\varphi\|_{\X_{T,p}} \leq C &\Big( \|f\|_{L^p((0,T)\times I)} + \|h_0\|_{W^{1-\frac{1}{4p},p}(0,T)}\\
        &\qquad+\|h_1\|_{W^{\frac34-\frac{1}{4p},p}(0,T)}+\|h_3\|_{W^{\frac{1}{4}-\frac{1}{4p},p}(0,T)} \Big).\label{eq:time-indep-maxreg-estimate}
    \end{align}
    Moreover, the constant $C(T_0,\gamma_0)$ depends continuously on $\gamma_0$ in $W^{4(1-\frac1p),p}(I)$.
\end{proposition}
For better readability, the proof of \Cref{prop:wp-lin-system} is postponed to \Cref{app:wp-lin-system}.

\begin{proposition}[Local existence and uniqueness]\label{prop:ana-prob-local-wp}
    There are $\bar r=\bar r(\gamma_0)>0$, $\bar T=\bar T(\gamma_0)>0$ such that, for all $T\in (0,\bar T)$, there exists $\varphi\in \bar B_{\prescript{}{0}{\mathbb X}_{T,p}}(0,\bar r)$ which is uniquely determined in $\bar B_{\prescript{}{0}{\mathbb X}_{T,p}}(0,\bar r)$ with $\mathcal T_{T}(\varphi)=0$. 
    
    Moreover, $\bar T$ depends continuously on $\gamma_0$ in $W^{4(1-\frac1p),p}(I)$ in the sense of \Cref{def:cts-dependence-on-gamma0}.
\end{proposition}
\begin{proof}
    Consider the $C^1$-map $\mathcal{F}\colon\mathcal U^{\gamma_0}_{T,p}\to \prescript{}{0}{\mathbb{F}}_{T,p}$ (see \eqref{eq:SetU} and \eqref{eq:Ftp}) given by
    \begin{align}
        \mathcal F(\varphi) =&\ \mathcal T(\varphi) - \mathcal T'(0)\varphi  + \Big(- \frac{8\partial_x^4\bar\gamma}{|\partial_x\bar\gamma|^6}\langle\partial_x\bar\gamma,\partial_x\varphi\rangle  - \Lambda'(0)\varphi \cdot\curv_{\bar\gamma}, 0,0,0\Big)
    \end{align}
    as well as the linear operator $\mathcal{S}\colon \prescript{}{0}{\X_{T,p}}\to \prescript{}{0}{\mathbb{F}}_{T,p}$
    \begin{align}
        \mathcal S(\varphi) = - \Big(&\partial_t\varphi + \frac2{|\partial_x\bar\gamma|^4}\partial_x^4\varphi + \sum_{i=1}^3 a_{i}(|\partial_x\bar\gamma|^{-1},\partial_x\bar\gamma,\dots,\partial_x^3\bar\gamma)\partial_x^i\varphi \\ & - \Lambda(0) \cdot \sum_{i=1}^2b_{i}(|\partial_x\bar\gamma|^{-1},\partial_x\bar\gamma,\partial_x^2\bar\gamma)\partial_x^i\varphi,\mathcal{B}_0'(0)\varphi,\mathcal{B}_1'(0)\varphi,\mathcal{B}_3'(0)\varphi\Big).
    \end{align}
    Using \Cref{lem:linearizations-of-analytic-system}, we have $\mathcal S(\varphi)=\mathcal F(\varphi)$ if and only if $\mathcal T(\varphi)=0$.
    
    We will follow the strategy outlined in \cite[Section~9.2]{pruesssimonett2016} to apply Banach's fixed point theorem to $\mathcal S^{-1}\mathcal F$. To this end, we first verify the following properties:
    \begin{itemize}
        \item[(MR)] For each $T\in (0,T_0)$, $\mathcal S\colon \prescript{}{0}{\mathbb X}_{T,p}\to \prescript{}{0}{\mathbb F}_{T,p}$ is an isomorphism and
        \begin{equation}
            \|\mathcal S^{-1}\|_{L(\prescript{}{0}{\mathbb F}_{T,p},\prescript{}{0}{\mathbb X}_{T,p})} \leq C(T_0,\gamma_0).
        \end{equation}
        Notice that this condition is verified by \Cref{prop:wp-lin-system} where the $T$-independent bound on $\|\mathcal S^{-1}\|$ follows from \eqref{eq:time-indep-maxreg-estimate}.
        \item[(NL)] For each $T\in(0,T_0)$, $\mathcal F$ is of class $C^1$ and
        \begin{enumerate}[(i)]
            \item $\|\mathcal F(0)\|_{\mathbb F_{T,p}}\to 0$ as $T\searrow 0$;
            \item $\|\mathcal F'(0)\|_{L(\prescript{}{0}{\mathbb X}_{T,p},\prescript{}{0}{\mathbb F}_{T,p})}\to 0$ as $T\searrow 0$.
        \end{enumerate}
    \end{itemize}
    First, we discuss \emph{(NL)(i)}. Notice that $\mathcal F(0)=\mathcal T(0)$ and that $\|\mathcal T(0)\|_{{\mathbb F}_{T,p}}\to 0$ for $T\searrow 0$ simply follows from the dominated convergence theorem.
    
    Next, we verify \emph{(NL)(ii)}. Notice that
    \begin{equation}
        \mathcal F'(0)\varphi = \Big(- \frac{8\partial_x^4\bar\gamma}{|\partial_x\bar\gamma|^6}\langle\partial_x\bar\gamma,\partial_x\varphi\rangle  - \Lambda'(0)\varphi \cdot\curv_{\bar\gamma}, 0,0,0\Big).
    \end{equation}
    By \Cref{prop:xtp-mappings-and-embeddings}\eqref{it:hoeld-emb}, $\bar\gamma\in \mathbb X_{T_0,p}\cap C^{0}([0,T_0],C^{3}([-1,1]))$ so that $\frac{\partial_x^4\bar\gamma}{|\partial_x\bar\gamma|^6} \partial_x\bar\gamma\in L^p((0,T_0)\times(-1,1))$. Using the embedding $\prescript{}{0}{\mathbb X}_{T,p}\hookrightarrow C^0([0,T],C^1([-1,1],\R^n))$ with time-uniform embedding-constant (as a special case of \Cref{prop:xtp-mappings-and-embeddings}\eqref{it:hoeld-emb}), for $\varphi\in\prescript{}{0}{\mathbb X}_{T,p}$, we have
    \begin{align}
        \|\frac{\partial_x^4\bar\gamma}{|\partial_x\bar\gamma|^6}\langle\partial_x\bar\gamma,\partial_x\varphi\rangle \|_{L^p((0,T)\times(-1,1))} \leq \|\frac{\partial_x^4\bar\gamma}{|\partial_x\bar\gamma|^6} \partial_x\bar\gamma\|_{L^p((0,T)\times(-1,1))}\cdot C(T_0)   \cdot \|\varphi\|_{\mathbb{X}_{T,p}}.
    \end{align}
    Together, $\|\varphi\mapsto \frac{\partial_x^4\bar\gamma}{|\partial_x\bar\gamma|^6}\langle\partial_x\bar\gamma,\partial_x\varphi\rangle \|_{L(\prescript{}{0}{\mathbb X}_{T,p},L^p((0,T)\times I))} \to 0$
    for $T\searrow 0$ by the dominated convergence theorem. Moreover, using \Cref{lem:tilde-lambda-bounds}, for $\varphi\in\prescript{}{0}{\mathbb X}_{T,p}$,
    \begin{align}
        \|\Lambda'(0)\varphi \cdot\curv_{\bar\gamma}\|_{L^p((0,T)\times(-1,1))}&\leq C(\bar\gamma) \|\Lambda'(0)\|_{L(\prescript{}{0}{\mathbb X}_{T,p},C^0([0,T]))} \|\varphi\|_{\mathbb X_{T,p}} \cdot \|1\|_{L^p((0,T)\times(-1,1))}\\
        &\leq  \Big(C(\bar\gamma) \cdot C(T_0,\gamma_0)\cdot   (2T)^{\frac1p}\Big) \cdot \|\varphi\|_{\mathbb X_{T,p}}
    \end{align}
    so that \emph{(NL)(ii)} follows.

    In the following claim, we improve \emph{(NL)(ii)} to uniform neighborhoods of $0$ in the respective $\prescript{}{0}{\mathbb X}_{T,p}$ spaces.

    \textbf{Claim 1.} \emph{For $T\in(0,T_0)$, let 
    \begin{equation}
        \eta(T,r)=\sup\{\|\mathcal F'(\psi)\|_{_{L(\prescript{}{0}{\mathbb X}_{T,p},\prescript{}{0}{\mathbb F}_{T,p})}}\mid \psi\in\mathcal U^{\gamma_0}_{T,p}\text{ with }\|\psi\|_{\mathbb X_{T,p}}<r\}.
    \end{equation}
    For each $\varepsilon>0$, choosing $\bar T>0$ such that $\|\mathcal F'(0)\|_{L(\prescript{}{0}{\mathbb X}_{\bar T,p},\prescript{}{0}{\mathbb F}_{\bar T,p})}<\frac12\varepsilon$, there exists $\bar r=\bar r(\bar T)>0$ such that $\eta(T,r)<\varepsilon$ for all $0<T\leq\bar T$ and $0<r\leq \bar r$.}

    Let $\tilde\varepsilon>0$. By \emph{(NL)(ii)}, there exists $\bar T>0$ with $\|\mathcal F'(0)\|_{L(\prescript{}{0}{\mathbb X}_{T,p},\prescript{}{0}{\mathbb F}_{T,p})}<\frac{1}{2}\tilde\varepsilon$ for all $T\in (0, \bar T]$. Using that $\mathcal{F}$ is $C^1$, there exists $\bar r>0$ such that $\|\mathcal F'(\psi)\|_{_{L(\prescript{}{0}{\mathbb X}_{\bar T,p},\prescript{}{0}{\mathbb F}_{\bar T,p})}}<\tilde\varepsilon$ for all $\psi\in B_{\prescript{}{0}{\mathbb X}_{\bar T,p}}(0,\bar r)\subseteq \mathcal U^{\gamma_0}_{\bar T,p}$. 
    
    Let $0<T\leq \bar T$ and $\psi\in\mathcal U^{\gamma_0}_{T,p}$ with $\|\psi\|_{\mathbb X_{T,p}}<\frac{\bar r}{2}$. Using \Cref{lem:0xtp-extension-lemma}, there exists $\bar \psi \in \mathcal U^{\gamma_0}_{\bar T,p}$ with $\|\bar \psi\|_{\mathbb X_{\bar T,p}}<\bar r$ and $\bar\psi|_{[0,T]\times I}=\psi$. Moreover, again using \Cref{lem:0xtp-extension-lemma}, we have
    \begin{equation}
        \|\mathcal F'(\psi)\|_{_{L(\prescript{}{0}{\mathbb X}_{T,p},\prescript{}{0}{\mathbb F}_{T,p})}} \leq 2 \|\mathcal F'(\bar \psi)\|_{_{L(\prescript{}{0}{\mathbb X}_{\bar T,p},\prescript{}{0}{\mathbb F}_{\bar T,p})}} \leq 2 \tilde\varepsilon.
    \end{equation}
    That is, $\eta(T,r)<2\tilde\varepsilon$ for all $0<T\leq \bar T$ and $0<r<\frac{\bar r}{2}$. So \emph{Claim 1} follows.

    Now proceeding as in \cite[Section~9.2]{pruesssimonett2016}, using \emph{Claim 1} and properties \emph{(MR)} and \emph{(NL)(i)}, there exist $\bar T=\bar T(\gamma_0)$ and $\bar r=\bar r(\gamma_0)$ such that, for all $0<T\leq \bar T$, 
    \begin{equation}
        \mathcal S^{-1}\circ \mathcal F \colon \bar B_{\prescript{}{0}{\mathbb X}_{T,p}}(0,\bar r) \to \bar B_{\prescript{}{0}{\mathbb X}_{T,p}}(0,\bar r)
    \end{equation}
    is a contraction. In particular, Banach's fixed point theorem yields the claim.

    Finally, the continuous dependence of $\bar T(\gamma_0)$ on $\gamma_0$ in $W^{4(1-\frac1p),p}(I)$ can be verified as follows. The steps where the time $T$ needs to be reduced in the fixed-point argument above 
    can be tracked back to controlling the dependence on $\gamma_0$ of the constant in \emph{(MR)} and of the uniform-in-time-smallness of the respective norms in \emph{(NL)(i)} and \emph{(NL)(ii)}. The continuous dependence of the constant in \emph{(MR)} is part of \Cref{prop:wp-lin-system}. The uniform-in-time continuous dependence of the smallness of the norms in \emph{(NL)(i)} and \emph{(NL)(ii)} can be achieved by elementary estimates similarly as the ones carried out above, using \Cref{prop:xtp-mappings-and-embeddings} and \eqref{eq:bargamma-gamma0-estimate} for $E\gamma_0-E\tilde\gamma_0$.
\end{proof}


\begin{proof}[Proof of \Cref{thm:wp-ana-prob}]    
    Existence and the continuous dependence of the existence time on $\gamma_0$ follow from \Cref{prop:ana-prob-local-wp}. Uniqueness can be shown with the same argument as in \cite[Theorem~2.12]{ruppspener2020} where the proof relies on the key fact that $\bar r$ does not depend on $T\in(0,\bar T)$ in \Cref{prop:ana-prob-local-wp}.
\end{proof}

\subsection{Instantaneous smoothing for the analytic problem}

In literature, when showing the instantaneous smoothing property for an \emph{analytic problem} such as \eqref{eq:ana-prob} related to similar geometric flows (see \cite[Section~4]{garckemenzelpluda2020}), using that all terms in the equation and boundary conditions have a multilinear structure, the smoothing property is immediate with \Cref{prop:xtp-mappings-and-embeddings} and classical parabolic Schauder theory \cite{solonnikov1965MathPhysics}. However, we cannot directly apply this reasoning to \eqref{eq:ana-prob} since the boundary condition $d^M(\gamma(t,\pm1))=0$ is in general fully nonlinear. We first establish time-regularity to circumvent this issue.

Using a parameter-trick as in \cite[Article~B, Theorem~3.2]{rupp2022PhD}, also see \cite{escherpruesssimonett2003} and \cite[Section~9.4.1]{pruesssimonett2016}, we obtain regularity \emph{in time} as follows.

\begin{proposition}\label{prop:time-reg}
    Let $p>5$ and $\gamma_0\in W^{4(1-\frac1p),p}(I)$ be an immersion with $\E(\gamma_0)>0$ and \eqref{eq:ana-fbc}. Let $T>0$ and $\gamma\in\mathbb X_{T,p}$ solve the analytic problem with the initial datum $\gamma_0$. Then
    \begin{equation}
        \gamma\in C^{\infty}((0,T
        );W^{4(1-\frac1p),p}(I,\R^n)).
    \end{equation}
\end{proposition}

A major part in the proof of \Cref{prop:time-reg} is the following consequence of \Cref{prop:wp-lin-system}.

\begin{lemma}\label{lem:loc-wp-of-full-lin}
    In the setting of \Cref{prop:time-reg}, there exists $T'=T'(T,\gamma)\in(0,T]$ such that the following is true. 
    
    For all $f\in L^p((0,T')\times I)$, $h_0\in \prescript{}{0}{W}^{1-\frac{1}{4p},p}(0,T';\R\times\R)$, $h_1\in \prescript{}{0}{W}^{\frac34-\frac{1}{4p},p}(0,T';\R^{n-1}\times \R^{n-1})$ and $h_3\in \prescript{}{0}{W}^{\frac14-\frac{1}{4p},p}(0,T';\R^n\times\R^n)$, there exists a unique solution $\varphi\in \prescript{}{0}{\mathbb X}_{T',p}$ of
    \begin{align}
        \partial_t\varphi + \frac2{|\partial_x\gamma|^4}\partial_x^4\varphi + \sum_{i=1}^3 a_{i}(|\partial_x\gamma|^{-1},\partial_x\gamma,\dots,\partial_x^3\gamma)\partial_x^i\varphi - \frac{8\partial_x^4\gamma}{|\partial_x\gamma|^6}\langle\partial_x\gamma,\partial_x\varphi\rangle - \lambda'(\gamma)\varphi\cdot \curv_{\gamma}   \\
            \qquad - \lambda(\gamma) \cdot \sum_{i=1}^2b_{i}(|\partial_x\gamma|^{-1},\partial_x\gamma,\partial_x^2\gamma)\partial_x^i\varphi = f
            \qquad \text{in $L^p((0,T')\times I)$}
    \end{align}
    with boundary conditions
    \begin{equation}
        \begin{cases}
        
            \mathrm{tr}_{\partial I}\langle\xi\circ\gamma,\varphi\rangle=h_0&\text{on $[0,T']$}\\
            \langle T_i^{\pm1}(\mathrm{tr}_{\partial I}\gamma),\mathrm{tr}_{\partial I}\partial_x\varphi\rangle+\mathrm{tr}_{\partial I} [c_i\circ \gamma \cdot \varphi]= h_1^{(i)}&\text{on $[0,T']$, for $1\leq i\leq n-1$}\\
            \mathrm{tr}_{\partial I}\Big[ \frac{\partial_x^3\varphi}{|\partial_x\gamma|^3} + \sum_{i=0}^2d_{i,\gamma}(|\partial_x\gamma|^{-1},\partial_x\gamma,\dots,\partial_x^3\gamma)\partial_x^i\varphi\Big] = h_3&\text{on $[0,T']$}     
        \end{cases}
    \end{equation}
    where we use the notation of \Cref{lem:linearizations-of-analytic-system}. Moreover, one has an estimate
    \begin{align}
        \|\varphi\|_{\X_{T',p}} \leq C &\Big( \|f\|_{L^p((0,T')\times I)} + \|h_0\|_{W^{1-\frac{1}{4p},p}(0,T')}\\
        &\qquad+\|h_1\|_{W^{\frac34-\frac{1}{4p},p}(0,T')}+\|h_3\|_{W^{\frac{1}{4}-\frac{1}{4p},p}(0,T')} \Big).
    \end{align}
\end{lemma}
\begin{proof}
    By \Cref{prop:wp-lin-system}, writing $\bar\gamma=E\gamma_0$, we have that for all $\tilde T\in (0,T]$, the linear operator 
    \begin{align}
        \mathcal L\colon \prescript{}{0}{\mathbb X}_{\tilde T,p} \to \prescript{}{0}{\mathbb F}_{\tilde T,p},\quad
        \varphi&\mapsto \Big(\partial_t\varphi + \frac2{|\partial_x\bar\gamma|^4}\partial_x^4\varphi + \sum_{i=1}^3 a_{i}(|\partial_x\bar\gamma|^{-1},\partial_x\bar\gamma,\dots,\partial_x^3\bar\gamma)\partial_x^i\varphi \\
        &\qquad- \lambda(\bar\gamma) \cdot \sum_{i=1}^2b_{i}(|\partial_x\bar\gamma|^{-1},\partial_x\bar\gamma,\partial_x^2\bar\gamma)\partial_x^i\varphi, \\
        &\qquad\mathrm{tr}_{\partial I}\langle\xi\circ\bar\gamma,\varphi\rangle, \begin{pmatrix} 
            \langle T_i^{\pm1}(\mathrm{tr}_{\partial I}\bar\gamma),\mathrm{tr}_{\partial I}\partial_x\varphi\rangle+\mathrm{tr}_{\partial I} [c_i\circ \bar\gamma \cdot \varphi]
        \end{pmatrix}_{i=1}^{n-1},\\
        &\qquad\mathrm{tr}_{\partial I}\Big[ \frac{\partial_x^3\varphi}{|\partial_x\bar\gamma|^3} + \sum_{i=0}^2d_{i,\bar\gamma}(|\partial_x\bar\gamma|^{-1},\partial_x\bar\gamma,\dots,\partial_x^3\bar\gamma)\partial_x^i\varphi\Big]\Big)
    \end{align}
    is invertible and 
    \begin{equation}\label{eq:loc-wp-fullLin-1}
        \|\mathcal L^{-1}\|_{L(\prescript{}{0}{\mathbb F}_{\tilde T,p},\prescript{}{0}{\mathbb X}_{\tilde T,p})} \leq C(T,\gamma_0) \quad\text{for all $0<\tilde T\leq T$}.
    \end{equation}
    Furthermore, consider $\mathcal K\colon \prescript{}{0}{\mathbb X}_{\tilde T,p} \to \prescript{}{0}{\mathbb F}_{\tilde T,p}$ given by 
    \begin{align}
        &\mathcal K(\varphi) =  \big[- \frac{8\partial_x^4\gamma}{|\partial_x\gamma|^6}\langle\partial_x\gamma,\partial_x\varphi\rangle - \lambda'(\gamma)\varphi\cdot \curv_{\gamma} ,0,0,0\big]\\
        &- \Big[(\frac2{|\partial_x\bar\gamma|^4}-\frac2{|\partial_x\gamma|^4})\partial_x^4\varphi + \sum_{i=1}^3 (a_{i}(|\partial_x\bar\gamma|^{-1},\dots,\partial_x^3\bar\gamma)-a_{i}(|\partial_x\gamma|^{-1},\dots,\partial_x^3\gamma))\partial_x^i\varphi \\
        &\qquad- \sum_{i=1}^2 (\lambda(\bar\gamma)b_{i}(|\partial_x\bar\gamma|^{-1},\partial_x\bar\gamma,\partial_x^2\bar\gamma)-\lambda(\gamma)b_{i}(|\partial_x\gamma|^{-1},\partial_x\gamma,\partial_x^2\gamma))\partial_x^i\varphi, \\
        &\mathrm{tr}_{\partial I}\langle\xi\circ\bar\gamma-\xi\circ\gamma,\varphi\rangle,\\
        &\begin{pmatrix} 
            \langle T_i^{\pm1}(\mathrm{tr}_{\partial I}\bar\gamma)-T_i^{\pm1}(\mathrm{tr}_{\partial I}\gamma),\mathrm{tr}_{\partial I}\partial_x\varphi\rangle+\mathrm{tr}_{\partial I} [(c_i\circ \bar\gamma-c_i\circ \gamma) \cdot \varphi]
        \end{pmatrix}_{i=1}^{n-1},\\
        &\mathrm{tr}_{\partial I}\Big( (\frac{1}{|\partial_x\bar\gamma|^3}-\frac{1}{|\partial_x\gamma|^3})\partial_x^3\varphi + \sum_{i=0}^2(d_{i,\bar\gamma}(|\partial_x\bar\gamma|^{-1},\dots,\partial_x^3\bar\gamma)-d_{i,\gamma}(|\partial_x\gamma|^{-1},\dots,\partial_x^3\gamma))\partial_x^i\varphi\Big)\Big].
    \end{align}
    In order to prove the claim, we show that there exists $T'\in (0,T)$ such that $\mathcal L+\mathcal K\colon \prescript{}{0}{\mathbb X}_{ T',p} \to \prescript{}{0}{\mathbb F}_{ T',p}$ is invertible. Using \eqref{eq:loc-wp-fullLin-1} and a Neumann-series argument, this follows from 
    \begin{equation}\label{eq:loc-wp-fullLin-2}
        \|\mathcal K\|_{L(\prescript{}{0}{\mathbb X}_{\tilde T,p},\prescript{}{0}{\mathbb F}_{\tilde T,p})}\to 0\quad\text{as $\tilde T\searrow 0$}.
    \end{equation}
    We now verify \eqref{eq:loc-wp-fullLin-2}. First observe that, for the first term in the above definition of $\mathcal K(\varphi)$, we have
    \begin{equation}
        \|\varphi\mapsto - \frac{8\partial_x^4\gamma}{|\partial_x\gamma|^6}\langle\partial_x\gamma,\partial_x\varphi\rangle - \lambda'(\gamma)\varphi\cdot \curv_{\gamma}\|_{L(\prescript{}{0}{\mathbb X}_{\tilde T,p},\prescript{}{0}{\mathbb F}_{\tilde T,p})} \to 0 \quad\text{for $\tilde T\searrow 0$}
    \end{equation}
    with the same estimates used to show \emph{(NL)(ii)} in the proof of \Cref{prop:ana-prob-local-wp}. By the dominated convergence theorem,
    \begin{equation}\label{eq:loc-wp-fullLin-3}
        \|\bar \gamma-\gamma\|_{\mathbb X_{\tilde T,p}}\to 0 \quad\text{as $\tilde T\searrow 0$}.
    \end{equation}
    Since $\bar\gamma - \gamma\in \prescript{}{0}{\mathbb X}_{\tilde T,p}$ for all $0<\tilde T\leq T$, the time-uniform embeddings in \Cref{prop:xtp-mappings-and-embeddings} and \eqref{eq:loc-wp-fullLin-3} together imply that, for $\tilde T\searrow 0$
    \begin{align}
        \|\bar\gamma-\gamma\|_{C^0([0,\tilde T],C^3(I))} \leq C(p) \|\bar\gamma-\gamma\|_{\mathrm{BUC}([0,\tilde T],W^{4(1-\frac1p),p}(I))} \to 0.
    \end{align}
    Especially, the $C^0([0,\tilde T]\times I)$-norm of the differences of the coefficients evaluated in $\bar\gamma$ resp.\ in $\gamma$ appearing in the above definition of $\mathcal K(\varphi)$ converge to $0$ as $\tilde T\searrow 0$. Using this key observation as well as the elementary estimate \Cref{lem:estimate-in-boundary-terms}, \eqref{eq:loc-wp-fullLin-2} follows.
\end{proof}

\begin{remark}[Controlling the dependency of $T'$ in \Cref{lem:loc-wp-of-full-lin}]\label{rem:dep-t'-in-loc-wp-of-full-lin}
    We work in the context of \Cref{lem:loc-wp-of-full-lin}. Suppose that $T>0$ and $\gamma$, $\tilde\gamma  \in \mathbb X_{T,p}$ are solutions of the analytic problem with initial data $\gamma_0$, $\tilde\gamma_0$. We claim that for all $\delta>0$, there exists $\zeta>0$ such that, if 
    \begin{equation}
        \|\gamma-\tilde\gamma\|_{\mathbb X_{T,p}}<\zeta,
    \end{equation}
    then $T'(T,\tilde\gamma)\geq T'(T,\gamma)-\delta$.

    This can be obtained by verifying two things in the above proof. First, the fact that the constant in \eqref{eq:loc-wp-fullLin-1} depends continuously on $\gamma_0=\mathrm{tr}_{t=0}(\gamma)$ in $W^{4(1-\frac1p),p}(I)$, see \Cref{prop:wp-lin-system}. Second, the fact that $\|\mathcal K_{\tilde\gamma}\|_{L(\prescript{}{0}{\mathbb X}_{\tilde T,p},\prescript{}{0}{\mathbb F}_{\tilde T,p})}$ in \eqref{eq:loc-wp-fullLin-2} is controlled by $\|\mathcal K_\gamma\|_{L(\prescript{}{0}{\mathbb X}_{\tilde T,p},\prescript{}{0}{\mathbb F}_{\tilde T,p})}$ and $\|\gamma-\tilde\gamma\|_{\mathbb X_{T,p}}$, uniformly in $\tilde T\in(0,T]$. 
\end{remark}

\begin{proof}[Proof of \Cref{prop:time-reg}]
    Let $t_0\in (0,T')$ where $T'\in(0,T)$ is fixed later, let $\varepsilon_0>0$ such that $[t_0-3\varepsilon_0,t_0+3\varepsilon_0]\subseteq(0,T')$ and let $\xi\in C^\infty(\R)$ with $\chi_{[t_0-\varepsilon_0,t_0+\varepsilon_0]}\leq\xi\leq\chi_{[t_0-2\varepsilon_0,t_0+2\varepsilon_0]}$. Using \cite[Proposition~2.2]{escherpruesssimonett2003}, there exists $r_0>0$ such that, for all $r\in(-r_0,r_0)$, $t\mapsto t+\xi(t)r$ is a smooth diffeomorphism of $[0,T']$. Furthermore, by \cite[Proposition~5.3]{escherpruesssimonett2003}, for each $r\in(-r_0,r_0)$, the mapping 
    \begin{align}
        \Psi_r^*\colon \mathbb{X}_{T',p} \to \mathbb{X}_{T',p},\quad
        g\mapsto \big[(t,x)\mapsto g(t+\xi(t)r,x)\big]
    \end{align}
    is well-defined, linear and bounded. Moreover, we have
    \begin{equation}\label{eq:time-reg-1}
        \partial_t\Psi_r^*g(t,x) = (1+\xi'(t)r) \Psi_r^*(\partial_tg)(t,x) = (1+\xi'(t)r)\partial_tg(t+\xi(t)r,x).
    \end{equation}
    Next, consider the parameter-dependent mapping
    \begin{align}
        \mathcal G&\colon (E\gamma_0 +\mathcal U^{\gamma_0}_{T',p}) \times (-r_0,r_0) \to \prescript{}{0}{\mathbb F}_{T',p},\\
        (g,r)&\mapsto \big(\partial_tg + (1+\xi'(t)r)\big(\mathcal A(g) -\lambda(g)\curv_{g} \big), \mathcal{B}_0(g-E\gamma_0),\mathcal{B}_1(g-E\gamma_0),\mathcal{B}_3(g-E\gamma_0) \big).
    \end{align}
    One checks that $\mathcal{G}$ is well defined and a smooth map using \eqref{eq:SetU}. We now fix $T'\in(0,T)$ so that, by \Cref{lem:loc-wp-of-full-lin}, $D_1\mathcal G(\gamma,0)\colon \prescript{}{0}{\mathbb X}_{T',p}\to \prescript{}{0}{\mathbb F}_{T',p}$ is invertible. 

    We have $\mathcal G(\Psi_r^*\gamma,r)=0$ for all $r\in(-r_0,r_0)$, using \eqref{eq:time-reg-1} and the fact that $\gamma$ solves \eqref{eq:ana-prob}. 

    Altogether, the implicit function theorem applied in $(\gamma,0)$ yields that there exists $r_1\in(0,r_0)$ 
    such that the map $(-r_1,r_1)\to \bar\gamma+\mathcal U^{\gamma_0}_{T',p}$, $r\mapsto \Psi_r^*\gamma$ is smooth. In particular, composing with the bounded linear evaluation map $\mathbb X_{T',p}\to W^{4(1-\frac1p),p}(I)$, $g\mapsto g(t_0,\cdot)$ (see \Cref{prop:xtp-mappings-and-embeddings}\eqref{it:buc-emb}), we get smoothness of 
    \begin{equation}
        (-r_1,r_1)\to W^{4(1-\frac1p),p}(I),\quad r\mapsto \gamma(t_0+\xi(t_0)r,\cdot)=\gamma(t_0+r,\cdot).
    \end{equation}
    Altogether, we obtain $\gamma\in C^\infty((0,T'),W^{4(1-\frac1p),p}(I))$ with $T'=T'(T,\gamma)$ as in \Cref{lem:loc-wp-of-full-lin}. 
    
    Let $\bar T\in(0,T]$ be maximal with the property $\gamma\in C^\infty((0,\bar T),W^{4(1-\frac1p),p}(I))$. For the sake of contradiction, suppose that $\tau\vcentcolon=T-\bar T>0$. Taking a sequence $t_j\nearrow \bar T$, we have $$\gamma(t_j+\cdot,\cdot)|_{[0,\tau]\times I}\to \gamma(\bar T+\cdot,\cdot)|_{[0,\tau]\times I}\quad\text{in $\mathbb X_{\tau,p}$.}$$ 
    In particular, applying what we proved above to each $\gamma(t_j+\cdot,\cdot)|_{[0,\tau]\times I}$ and using \Cref{rem:dep-t'-in-loc-wp-of-full-lin}, we have $$\gamma\in C^\infty\Big(\big(t_j,t_j+\frac12 T'(\tau,\gamma(\bar T+\cdot,\cdot)|_{[0,\tau]\times I})\big),W^{4(1-\frac1p),p}(I)\Big)$$ for $j$ sufficiently large. This contradicts the choice of $\bar T$ and thus $\bar T=T$.
\end{proof}

Using the time-regularity proved in \Cref{prop:time-reg}, an application of (linear) Schauder theory yields also smoothness in space and thus concludes the proof of instantaneous smoothing for solutions of the analytic problem.

\begin{theorem}\label{thm:para-smooth}
    Let $p>5$ and $\gamma_0\in W^{4(1-\frac1p),p}(I)$ be an immersion with $\E(\gamma_0)>0$ and \eqref{eq:ana-fbc}. Let $T>0$ and $\gamma\in\mathbb X_{T,p}$ solve the analytic problem with the initial datum $\gamma_0$. Then 
    \begin{equation}
        \gamma\in C^{\infty}((0,T)\times [-1,1],\R^n).
    \end{equation}
\end{theorem}
\begin{proof}
    For $0<\varepsilon<T$ arbitrary, consider a cut-off function $\zeta\in C^\infty(\R)$ with $\chi_{(\varepsilon,\infty)} \leq \zeta \leq \chi_{(\frac12\varepsilon,\infty)}$ and define $\tilde\gamma(t,x)=\zeta(t) \cdot \gamma(t,x)$. Then $\tilde\gamma$ solves the linear system
    \begin{equation}
        \begin{cases}
            \partial_t\tilde\gamma+\frac{2}{|\partial_x\gamma|^4}\partial_x^4\tilde\gamma - 12 \frac{\langle \partial_x^2\gamma,\partial_x\gamma\rangle}{|\partial_x\gamma|^6}\partial_x^3\tilde\gamma - 8 \frac{\langle \partial_x^3\gamma,\partial_x\gamma\rangle}{|\partial_x\gamma|^6}\partial_x^2\tilde\gamma  -5\frac{|\partial_x^2\gamma|^2}{|\partial_x\gamma|^6}\partial_x^2\tilde\gamma \\
            \quad + 35 \frac{\langle\partial_x^2\gamma,\partial_x\gamma\rangle^2}{|\partial_x\gamma|^8}\partial_x^2\tilde\gamma - \frac{\lambda(\gamma)}{|\partial_x\gamma|^2}\big(I_{2\times 2}-\frac{\partial_x\gamma\cdot(\partial_x\gamma)^t}{|\partial_x\gamma|^2}\big)\partial_x^2\tilde\gamma = \partial_t\zeta \cdot\gamma&\text{in $L^p((0,T)\times I)$}\\
            \tilde\gamma(0,\cdot)=0&\text{in $W^{4(1-\frac1p),p}(I)$}\\
            \tilde\gamma(t,\pm1)=\zeta(t)\gamma(t,\pm1)&\text{for all $0\leq t\leq T$}\\
            \partial_x\tilde\gamma(t,\pm1)=\zeta(t)\partial_x\gamma(t,\pm1)&\text{for all $0\leq t\leq T$}.
        \end{cases}
    \end{equation}
    By \Cref{prop:time-reg}, $\zeta(t)\gamma(t,\pm1),\zeta(t)\partial_x\gamma(t,\pm1)\in C^{\infty}([0,T))$ and, by \Cref{prop:xtp-mappings-and-embeddings}\eqref{it:hoeld-emb}, all coefficients of the linear system above belong to $C^{\frac{\alpha}{4},\alpha}([0,T']\times I)$ for some $\alpha\in(0,1)$, for all $0<T'<T$. 
    
    Therefore, the claimed smoothness follows with a bootstrapping argument using \cite[Theorem~4.9]{solonnikov1965MathPhysics}. Note that compatibility conditions of any order are satisfied by the choice of $\zeta$.
\end{proof}

\subsection{Existence and regularity of reparametrizations: from analytic problem to equation with zero tangential speed and back}\label{sec:wp-reparametrizations}

First, we show that a smooth solution of the geometric problem \eqref{eq:geo-ev-eq} can be reparametrized to a solution of the geometric problem with zero tangential speed, that is, to a solution of \eqref{eq:ev-eq-intro}. Recall from \Cref{rem:sol-of-anaprob-is-sol-of-geo-prob} that any solution of the analytic problem especially also solves \eqref{eq:geo-ev-eq}. Second, we start with a solution of \eqref{eq:ev-eq-intro} and show that it can be reparametrized to a solution of the analytic problem. This is a key step needed to transfer properties such as unique solvability from the analytic problem to \eqref{eq:ev-eq-intro}.

Recall that, for an initial datum $\gamma_0$ satisfying \eqref{eq:orth-bc}, we abbreviate $\tau_0=\langle\partial_s\gamma_0,\xi\circ\gamma_0\rangle$.

\begin{proposition}\label{prop:ex-rep-ana-to-geo}
    Consider a smooth family of immersions $\gamma\colon[0,T)\times[-1,1]\to\R^n$ solving the geometric equation 
    \begin{equation}\label{eq:geom-ev-eq}
        \begin{cases}
            (\partial_t\gamma(t,x))^\bot = -(\nabla\E(\gamma(t,\cdot))(x)-\lambda(\gamma(t,\cdot)) \cdot \curv(t,x))&\text{for all $(t,x)\in[0,T)\times[-1,1]$}\\
            \gamma(0,\cdot)=\gamma_0&\text{in $[-1,1]$}\\
            \gamma(t,\pm1)\in M&\text{for all $t\in[0,T)$}\\
            \partial_s\gamma(t,\pm1)\bot T_{\gamma(t,\pm1)}M&\text{for all $t\in[0,T)$}\\
            \partial_s^{\bot}\curv(t,\pm1) = -\tau_0(\pm1)S_{\gamma(t,\pm1)}\curv(t,\pm1) &\text{for all $t\in[0,T)$}.
        \end{cases}
    \end{equation}
    and let $\phi_0\colon[-1,1]\to[-1,1]$ be a smooth positive diffeomorphism. Then there exists a smooth family of (positive) diffeomorphisms $\phi\colon[0,T)\times[-1,1]\to[-1,1]$ with $\phi(0,\cdot)=\phi_0$ and such that $(t,y)\mapsto \gamma(t,\phi(t,y))$ is a FLFB-EF starting at $\gamma(0,\phi_0(\cdot))$.
\end{proposition}
\begin{proof}
    Consider the smooth function $F\colon[0,T)\times[-1,1]\to\R$ given by
    \begin{equation}
        F(t,z)=-\frac{\langle\partial_t\gamma(t,z),\partial_x\gamma(t,z)\rangle}{|\partial_x\gamma(t,z)|^2},
    \end{equation}
    and, for each $y\in[-1,1]$, consider the associated ODE
    \begin{equation}
        \begin{cases}
            \partial_t\phi(t,y)=F(t,\phi(t,y))\\
            \phi(0,y)=\phi_0(y).
        \end{cases}
    \end{equation}
    Using that $\langle\partial_t\gamma(t,\phi(t,y)),\partial_x\gamma(t,\phi(t,y))\rangle=0$ if $y=\pm1$, for all $0\leq t<T$, we get $F(t,\pm1)=0$ for all $0\leq t<T$. Thus, there exists a smooth family of maps $\phi\colon[0,T)\times[-1,1]\to[-1,1]$ solving the ODE such that $\phi(t,\pm1)=\pm1$ and $\phi(t,y)\in(-1,1)$ for all $y\in(-1,1)$ and $0\leq t<T$. We only need to show that $|\partial_y\phi(t,y)|>0$ for all $t\in[0,T)$ and $y\in[-1,1]$. Using $\partial_y\phi_0>0$ by assumption, this follows from the computation
    \begin{equation}
        \partial_t \log\partial_y\phi(t,y) = \frac{\partial_y(F(t,\phi(t,y)))}{\partial_y\phi(t,y)} = \partial_zF(t,\phi(t,y))
    \end{equation}
    as $\partial_zF$ is bounded on compact sets in $[0,T)\times[-1,1]$.
\end{proof}

\begin{lemma}\label{lem:reg-of-t-wise-inverse}
    Let $p>5$ and suppose that $\phi\colon[0,T]\times  I\to  I$ is a family of $C^1$-diffeomorphisms $I\to I$ with
    \begin{equation}
        \phi\in W^{1,p}(0,T;L^p(I,\R))\cap L^p(0,T;W^{4,p}(I,\R)).
    \end{equation}
    Defining $\eta\colon[0,T]\times  I\to  I$, $\eta(t,\cdot) \vcentcolon= \phi(t,\cdot)^{-1}$, it follows that also 
    \begin{equation}
        \eta\in W^{1,p}(0,T;L^p(I,\R))\cap L^p(0,T;W^{4,p}(I,\R)).
    \end{equation}
\end{lemma}
\begin{proof}
    Using the chain rule for Sobolev functions and the fact that $\phi(t,\cdot)\in C^3(I)$ for every $t\in [0,T]$ by \Cref{prop:xtp-mappings-and-embeddings}\eqref{it:hoeld-emb}, we get that $\eta(t,\cdot)$ is $C^3$ with    
    \begin{align}
        \partial_x\eta(t,x)&=\frac{1}{\partial_y\phi(t,y)},\quad \partial_x^2\eta(t,x)=-\frac{\partial_y^2\phi(t,y)}{(\partial_y\phi(t,y))^3},\\
        \partial_x^3\eta(t,x)&= - \frac{\partial_y^3\phi(t,y)}{(\partial_y\phi(t,y))^4}+3\frac{(\partial_y^2\phi(t,y))^2}{(\partial_y\phi(t,y))^5}\label{eq:eta-upto-third-order-der}
    \end{align}
    for $x=\phi(t,y)$. So for a.e.\ $t\in (0,T)$, $\partial_x^3\eta(t,\cdot)$ is weakly differentiable with weak derivative given by
    \begin{equation}\label{eq:g-un-1}
        \begin{aligned}
            \partial_x^4\eta(t,\phi(t,\cdot))&=-\frac{\partial_y^4\phi(t,\cdot)}{(\partial_y\phi(t,\cdot))^5}+10\frac{\partial_y^2\phi(t,\cdot)}{(\partial_y\phi(t,\cdot))^6}\partial_y^3\phi(t,\cdot)-15\frac{(\partial_y^2\phi(t,\cdot))^3}{(\partial_y\phi(t,\cdot))^7}.
        \end{aligned}
    \end{equation}
    Using the embedding in \Cref{prop:xtp-mappings-and-embeddings}\eqref{it:hoeld-emb}, we get $\eta\in L^p(0,T;W^{4,p}(I))$. Next, we verify that $\eta\in W^{1,p}(0,T;L^p(I))$.

    To this end, define $\Phi\colon [0,T]\times  I\to [0,T]\times  I$, $(t,x)\mapsto (t,\phi(t,x))$. As $\phi\in C^0([0,T],C^3( I))$ by \Cref{prop:xtp-mappings-and-embeddings}\eqref{it:hoeld-emb}, we particularly have that $\Phi$ is continuous. Moreover, since $[0,T]\times  I$ is compact and as $\Phi$ is bijective, we conclude that $\Phi$ is a homeomorphism. Further, for $\Omega\vcentcolon=(0,T)\times (-1,1)$, it is clear that
    \begin{equation}
        \Phi\in W^{1,p}(\Omega,\R^2)
    \end{equation}
    with non-vanishing Jacobian in $\Omega$. Hence, by \cite[Theorem 1.1]{heclkosela2006}, it follows that $\Psi\vcentcolon=\Phi^{-1}$ is weakly differentiable in $\Omega$. Proceeding as in \cite[Equations~(3.11) and (3.12)]{heclkosela2006}, we get
    \begin{equation}
        \int_{\Omega} |D\Psi(t,x)|^p\dd (t,x) \leq \int_{\Omega} \frac{|D\Phi(t,y)|^p}{J_\Phi(t,y)^{p-1}}\dd (t,y)
    \end{equation}
    and thus $\Psi\in W^{1,p}(\Omega)$. Now note 
    \begin{equation}
        (t,y)=\Psi\circ \Phi(t,y)=\Psi(t,\phi(t,y)) =\vcentcolon (\Psi^{(1)}(t,\phi(t,y)),\Psi^{(2)}(t,\phi(t,y))).
    \end{equation}
    Altogether, $\Psi^{(2)}=\eta$ and therefore $\eta\in W^{1,p}(0,T;L^p(I,\R))$.
\end{proof}

\begin{proposition}\label{lem:ex-rep}
    Let $p>5$, $T>0$ and suppose that $\gamma\in\X_{T,p}$ is a family of immersions solving \eqref{eq:ev-eq} with $\gamma_0$ satisfying \eqref{eq:ana-fbc}. 
    
    Then there exist $T'\in(0,T)$ and $\phi\in W^{1,p}(0,T';L^p((-1,1),\R))\cap L^p(0,T';W^{4,p}((-1,1),\R))$ such that $\phi(t,\cdot)$ is an orientation preserving $C^1$-diffeomorphism of $[-1,1]$ for each $0\leq t\leq T'$ and such that $\tilde{\gamma}(t,y)\vcentcolon= \gamma(t,\phi(t,y))$ satisfies $\tilde\gamma\in \mathbb X_{T',p}$ and solves \eqref{eq:ana-prob}. 

    Moreover, if $\gamma$ is smooth on $(0,T)\times[-1,1]$, then $\phi|_{(0,T')\times[-1,1]}$ is a smooth family of diffeomorphisms.
\end{proposition}
\begin{proof}
    Consider the linear PDE
    \begin{equation}\label{eq:pde-reparam}
        \begin{cases}
            \partial_t\eta = - \frac{2}{|\partial_x\gamma|^4}\partial_x^4\eta + \frac{12}{|\partial_x\gamma|^6}\langle\partial_x^2\gamma,\partial_x\gamma\rangle\partial_x^3\eta + \frac{8}{|\partial_x\gamma|^6}\langle\partial_x^3\gamma,\partial_x\gamma\rangle\partial_x^2\eta \\
            \quad - \frac{35}{|\partial_x\gamma|^8}(\langle\partial_x^2\gamma,\partial_x\gamma\rangle)^2 \partial_x^2\eta + \frac{5}{|\partial_x\gamma|^6}|\partial_x^2\gamma|^2\partial_x^2\eta +\frac{2}{|\partial_x\gamma|^6}\langle\partial_x^4\gamma,\partial_x\gamma\rangle\partial_x\eta\\
            \quad -\frac{20}{|\partial_x\gamma|^8}\langle\partial_x^3\gamma,\partial_x\gamma\rangle\langle\partial_x^2\gamma,\partial_x\gamma\rangle \partial_x\eta - \frac{5}{|\partial_x\gamma|^8}|\partial_x^2\gamma|^2\langle\partial_x^2\gamma,\partial_x\gamma\rangle \partial_x\eta \\
            \quad +\frac{35}{|\partial_x\gamma|^{10}}(\langle\partial_x^2\gamma,\partial_x\gamma\rangle)^3 \partial_x\eta 
            \quad \text{on $[0,T]\times [-1,1]$}\\
            \eta(0,\cdot)=\mathrm{id}_{[-1,1]}\\
            \eta(t,\pm1)=\pm 1\quad\text{for all $0\leq t\leq T$}\\
            \partial_x^3\eta = 3\frac{\langle\partial_x^3\gamma,\partial_x\gamma\rangle}{|\partial_x\gamma|^2}\partial_x^2\eta + \Big(\frac{\langle\partial_x^3\gamma,\partial_x\gamma\rangle}{|\partial_x\gamma|^2}-3\frac{(\langle\partial_x^2\gamma,\partial_x\gamma\rangle)^2}{|\partial_x\gamma|^4}\Big)\partial_x\eta \quad\text{on $[0,T]\times \{\pm1\}$}.
        \end{cases}
    \end{equation}
    Since $\frac{\langle\partial_x^3\gamma_0(\pm1),\partial_x\gamma_0(\pm1)\rangle}{|\partial_x\gamma_0(\pm1)|^2}-3\frac{(\langle\partial_x^2\gamma_0(\pm1),\partial_x\gamma_0(\pm1)\rangle)^2}{|\partial_x\gamma_0(\pm1)|^4}=0$ follows from \eqref{eq:ana-fbc}, the compatibility conditions are satisfied and \cite[Theorem~2.1]{denkhieberpruess2007} yields that there exists a unique solution 
    \begin{equation}
        \eta\in W^{1,p}(0,T;L^p((-1,1),\R))\cap L^p(0,T;W^{4,p}((-1,1),\R)).
    \end{equation} 
    Moreover, since $|\partial_x\eta(0,\cdot)|\equiv 1$ and $\eta(t,\pm1)=\pm1$ for all $0\leq t\leq T$, there exists $T'\in(0,T]$ such that $\eta(t,\cdot)\colon[-1,1]\to[-1,1]$ is a (positive) $C^1$-diffeomorphism for all $0\leq t\leq T'$, using \Cref{prop:xtp-mappings-and-embeddings}\eqref{it:hoeld-emb}. 
    
    For $\phi(t,\cdot)\vcentcolon=\eta(t,\cdot)^{-1}$, we have $\phi\in W^{1,p}(0,T';L^p((-1,1),\R))\cap L^p(0,T';W^{4,p}((-1,1),\R))$ by \Cref{lem:reg-of-t-wise-inverse}. Thus, by \cite[Lemma~5.4]{garckemenzelpluda2020}, we have $\tilde\gamma\in\mathbb X_{T',p}$.

    Furthermore, proceeding as in \cite[Proof of Lemma~C.1]{dallacqualinpozzi2021}, especially using \cite[Equation~(C7)]{dallacqualinpozzi2021}, we have that $\partial_t\tilde\gamma + \A(\tilde\gamma) - \lambda(\tilde\gamma)\curv_{\tilde\gamma} = 0$ in $L^p((0,T')\times(-1,1))$. It remains to show that $\tilde\gamma$ satisfies \eqref{eq:ana-fbc} on $[0,T']$. To this end, using that $\gamma$ solves 
    \eqref{eq:ev-eq}, we have for $0\leq t\leq T'$ and $y\in\{\pm1\}$
    \begin{equation}
        0=\partial_s^\bot \curv_\gamma + \tau_0 S_\gamma\curv_\gamma \big|_{(t,x)=(t,\phi(t,y))} = \partial_s^\bot \curv_{\tilde\gamma}+\tau_0 S_{\tilde\gamma}\curv_{\tilde\gamma}\big|_{(t,y)}.
    \end{equation}
    Thus, using \eqref{eq:dscurv}, in order to verify the third-order boundary condition in \eqref{eq:ana-fbc}, one only needs to check that
    \begin{equation}
        \langle \frac{\partial_y^3\tilde\gamma}{|\partial_y\tilde\gamma|^3}-3\frac{\langle \partial_y^2\tilde\gamma,\partial_y\tilde\gamma\rangle}{|\partial_y\tilde\gamma|^5}\partial_y^2\tilde\gamma,\partial_y\tilde\gamma\rangle = 0\quad\text{on $[0,T']\times\{\pm1\}$}.
    \end{equation}
    This can be reduced to the third-order boundary condition for $\eta$ in \eqref{eq:pde-reparam}, using that, for $0\leq t\leq T'$ and $x=\phi(t,y)$
    \begin{align}
        \partial_y\tilde\gamma(t,y)&=\partial_x\gamma(t,x)\partial_y\phi(t,y),\quad
        \partial_y^2\tilde\gamma(t,y)=\partial_x^2\gamma(t,x)(\partial_y\phi(t,y))^2+\partial_x\gamma(t,x)\partial_y^2\phi(t,y)\\
        \partial_y^3\tilde\gamma(t,y)&=\partial_x^3\gamma(t,x)(\partial_y\phi(t,y))^3+3\partial_x^2\gamma(t,x)\partial_y\phi(t,y)\partial_y^2\phi(t,y) +\partial_x\gamma(t,x)\partial_y^3\phi(t,y)
    \end{align}
    and \eqref{eq:eta-upto-third-order-der}.
    Finally suppose that $\gamma$ is smooth on $(0,T)\times I$. For $0<\varepsilon<T'$ arbitrary, consider a cut-off function $\zeta\in C^\infty(\R)$ with $\chi_{(\varepsilon,\infty)} \leq \zeta \leq \chi_{(\frac12\varepsilon,\infty)}$ and define $\tilde\eta(t,x)=\zeta(t) \cdot \eta(t,x)$. Then $\tilde\eta$ solves the linear system
    \begin{equation}
        \begin{cases}
            \partial_t\tilde\eta + \frac{2}{|\partial_x\gamma|^4}\partial_x^4\tilde\eta - \frac{12}{|\partial_x\gamma|^6}\langle\partial_x^2\gamma,\partial_x\gamma\rangle\partial_x^3\tilde\eta - \frac{8}{|\partial_x\gamma|^6}\langle\partial_x^3\gamma,\partial_x\gamma\rangle\partial_x^2\tilde\eta \\
            \quad + \frac{35}{|\partial_x\gamma|^8}(\langle\partial_x^2\gamma,\partial_x\gamma\rangle)^2 \partial_x^2\tilde\eta - \frac{5}{|\partial_x\gamma|^6}|\partial_x^2\gamma|^2\partial_x^2\tilde\eta -\frac{2}{|\partial_x\gamma|^6}\langle\partial_x^4\gamma,\partial_x\gamma\rangle\partial_x\tilde\eta\\
            \quad +\frac{20}{|\partial_x\gamma|^8}\langle\partial_x^3\gamma,\partial_x\gamma\rangle\langle\partial_x^2\gamma,\partial_x\gamma\rangle \partial_x\tilde\eta + \frac{5}{|\partial_x\gamma|^8}|\partial_x^2\gamma|^2\langle\partial_x^2\gamma,\partial_x\gamma\rangle \partial_x\tilde\eta \\
            \quad -\frac{35}{|\partial_x\gamma|^{10}}(\langle\partial_x^2\gamma,\partial_x\gamma\rangle)^3 \partial_x\tilde\eta 
            = (\partial_t\zeta)\eta\quad \text{on $[\frac14\varepsilon,T']\times [-1,1]$}\\
            \tilde\eta(\frac14\varepsilon,\cdot)=0\\
            \tilde\eta(t,\pm1)=\pm \zeta(t)\quad\text{for all $\frac14\varepsilon\leq t\leq T'$}\\
            \partial_x^3\tilde\eta - 3\frac{\langle\partial_x^3\gamma,\partial_x\gamma\rangle}{|\partial_x\gamma|^2}\partial_x^2\tilde\eta - \Big(\frac{\langle\partial_x^3\gamma,\partial_x\gamma\rangle}{|\partial_x\gamma|^2}-3\frac{(\langle\partial_x^2\gamma,\partial_x\gamma\rangle)^2}{|\partial_x\gamma|^4}\Big)\partial_x\tilde\eta  = 0\quad\text{on $[\frac14\varepsilon,T']\times \{\pm1\}$}.
        \end{cases}
    \end{equation}
    By assumption, all coefficients are smooth on $[\frac14\varepsilon,T']\times[-1,1]$. Moreover, using $\partial_t^j\zeta(\frac14\varepsilon)=0$ for all $j\in\N_0$, compatibility conditions of any order are satisfied. Thus, \cite[Theorem~4.9]{solonnikov1965MathPhysics} yields that $\tilde\eta$ is smooth on $[\frac14\varepsilon,T']$. In particular, $\eta$ is smooth on $(\varepsilon,T')$. Since $\varepsilon\in(0,T')$ was chosen arbitrary, this yields the claimed regularity of $\phi$.
\end{proof}

Combining \Cref{thm:wp-ana-prob,thm:para-smooth,prop:ex-rep-ana-to-geo}, and \cite[Lemma~D.1]{ruppspener2020}, we can give the proof of \Cref{thm:existence-of-solution,cor:ex-no-tang-velocity}.

\begin{proof}[Proofs of \Cref{thm:existence-of-solution,cor:ex-no-tang-velocity}] 
    Since $\gamma_0$ satisfies \eqref{eq:intro-free-bcs}, it solves \eqref{eq:orth-bc} and \eqref{eq:nat-bc} with $\tau_0=\langle\partial_s\gamma_0,\xi\circ\gamma_0\rangle$. 
    By \Cref{lem:geo-vs-ana-fbc}, there exists a smooth diffeomorphism $\phi_0\colon[-1,1]\to[-1,1]$ such that $\tilde\gamma_0=\gamma_0\circ\phi_0\in W^{4(1-\frac1p),p}(I)$ satisfies \eqref{eq:ana-fbc}. Using \Cref{thm:wp-ana-prob,thm:para-smooth}, there exists $T>0$ and 
    \begin{equation}
        \tilde\gamma\in W^{1,p}(0,T;L^p((-1,1)))\cap L^p(0,T;W^{4,p}((-1,1))) \cap C^\infty((0,T)\times [-1,1])
    \end{equation} 
    solving the analytic problem with initial datum $\tilde\gamma_0$. Since $\phi_0$ is smooth, so is $\psi_0=\phi_0^{-1}$. Thus, $\gamma(t,x)=\tilde\gamma(t,\psi_0(x))$ satisfies
    \begin{equation}
        \gamma \in W^{1,p}(0,T;L^p((-1,1)))\cap L^p(0,T;W^{4,p}((-1,1))) \cap C^\infty((0,T)\times [-1,1]).
    \end{equation}
    Moreover, $\gamma(0,\cdot)=\tilde\gamma(0,\psi_0(\cdot))=\gamma_0$. This proves \Cref{thm:existence-of-solution}, using \Cref{rem:sol-of-anaprob-is-sol-of-geo-prob}, \Cref{prop:xtp-mappings-and-embeddings}\eqref{it:buc-emb}, and the fact that the PDE in \eqref{eq:geo-ev-eq} is invariant with respect to reparametrization. For \Cref{cor:ex-no-tang-velocity}, letting $\varepsilon\in(0,T)$, $\gamma|_{[\varepsilon,T)\times[-1,1]}$ solves \eqref{eq:geom-ev-eq}. Thus, the corollary follows from \Cref{prop:ex-rep-ana-to-geo}, using \cite[Lemma~D.1]{ruppspener2020} as in \cite[Proof of Theorem~4.1]{ruppspener2020}.
\end{proof}

\section{Analysis of the long-time behavior}\label{sec:long-time}

\subsection{Geometric Gagliardo--Nirenberg-type interpolation estimates}

As in \cite{dziukkuwertschaetzle2002}, consider the following scaling-invariant norms.

\begin{equation}
    \|(\partial_s^{\bot})^i\curv\|_p\vcentcolon= \Ll(\gamma)^{i+1-\frac{1}{p}} \|(\partial_s^{\bot})^i\curv\|_{L^p(\dd s)}\quad \text{and}\quad \|\curv\|_{k,p} \vcentcolon= \sum_{i=0}^k \|(\partial_s^{\bot})^i\curv\|_{p}.
\end{equation}

\begin{proposition}[{\cite[Lemma~3.1 and Corollary~3.2]{dallacqualinpozzi2017}}]\label{prop:inter-dspk}
    Let $\gamma\colon[-1,1]\to\R^n$ be an immersion, $k\in\N$, $p\in[2,\infty]$ and $0\leq i<k$. Then
    \begin{equation}
        \|(\partial_s^{\bot})^i\curv\|_p \leq C \|\curv\|_2^{1-\alpha} \|\curv\|_{k,2}^{\alpha}
    \end{equation}
    for $\alpha=\frac{1}{k}(i+\frac12-\frac1p)$ and $C=C(n,k,p)$. Moreover,
    \begin{equation}
        \|\curv\|_{k,2}\leq C(n,k) \big(\|(\partial_s^{\bot})^k\curv\|_2+\|\curv\|_2\big).
    \end{equation}
\end{proposition}

\begin{proposition}\label{prop:inter-pbac}
    Consider an immersion $\gamma\colon[-1,1]\to\R^n$ and $a,b,c,k\in\N_0$ with $b,k\geq 2$ and $c\leq k$. 
    If $a<2k$, or if $c\leq k-1$,
    \begin{equation}
        \int_{-1}^1 |P_b^{a,c}|\dd s\leq C\Ll(\gamma)^{1-a-b}\|\curv\|_2^{b-\gamma}\|\curv\|_{k,2}^{\gamma}
    \end{equation}
    where $C=C(C(P_b^{a,c},0),n,k,a,b)$ with $\gamma=\frac{1}{k}(a+\frac12 b-1)$. If $a+\frac12 b < 2k+1$, then, for any $\varepsilon>0$, 
    \begin{align}
        \int_{-1}^1 |P_{b}^{a,c}|\dd s \leq \varepsilon \int_{-1}^1|(\partial_s^{\bot})^k\curv|^2\dd s + C \varepsilon^{-\frac{\gamma}{2-\gamma}}\E(\gamma)^{\frac{b-\gamma}{2-\gamma}} + C \Ll(\gamma)^{1-a-\frac12 b}\E(\gamma)^{\frac12b}
    \end{align}
    where $C=C(C(P_b^{a,c},0), n,k,a,b)$ and $\gamma=\frac{1}{k}(a+\frac12b-1)\in(0,2)$.
\end{proposition}
\begin{proof}
    Writing $P_{b}^{a,c}=T_{\gamma}(\partial_s\gamma,\dots,\partial_s\gamma,(\partial_s^{\bot})^{i_1},\dots,(\partial_s^{\bot})^{i_b}\curv)$, \eqref{eq:hyp-tensor-unif-bounds} yields
    \begin{equation}
        |P_{b}^{a,c}|\leq C(P_{b}^{a,c},0) \prod_{j=1}^b|(\partial_s^{\bot})^{i_j}\curv|. 
    \end{equation}
    Now the claim is proved with \cite[Lemmas~3.3 and 3.4]{dallacqualinpozzi2017}.
\end{proof}

\begin{corollary}\label{cor:inter-ppABC}
    Let $\gamma\colon[0,T)\times [-1,1]\to\R^n$ be a FLFB-EF, $L_0=\Ll(\gamma_0)$, and $A,B,c,k\in\N_0$ with $B,k\geq 2$ and $c\leq k$. 
    First, if $c\leq k-1$ and $\gamma=\frac1k(A+\frac12B-1)\in(0,2)$,
    \begin{equation}
        \int_{-1}^1 |\mathbb{P}_{B}^{A,c}| \dd s\leq C \big(1+\|(\partial_s^{\bot})^k\curv\|_{L^2(\dd s)}^{\gamma}\big)
    \end{equation}
    with $C$ depending only on $M$, $L_0$, $\E(\gamma_0)$, $A$, $B$, $c$, $k$ and the constants $C(P_{\centerdot}^{\centerdot,\centerdot},0)$ of the terms $P_{\centerdot}^{\centerdot,\centerdot}$ making up $\mathbb{P}_{B}^{A,c}$. Second, if $\gamma=\frac1k(A+\frac12B-1)\in(0,2)$ and $c\leq k$, for any $\varepsilon>0$,
    \begin{equation}
        \int_{-1}^1 |\mathbb{P}_{B}^{A,c}| \dd s\leq \varepsilon \int_{-1}^1|(\partial_s^{\bot})^k\curv|^2\dd s + C_{\varepsilon}
    \end{equation}
    where $C_{\varepsilon}$ only depends on $\varepsilon$ and again $M$, $L_0$, $\E(\gamma_0)$, $A$, $B$, $c$, $k$ and the constants $C(P_{\centerdot}^{\centerdot,\centerdot},0)$ of the terms $P_{\centerdot}^{\centerdot,\centerdot}$ making up $\mathbb{P}_{B}^{A,c}$.
\end{corollary}


\subsection{A priori estimates for the Lagrange multiplier}

Applying \Cref{prop:inter-pbac} to \eqref{eq:def-lambda} yields the following estimate.

\begin{lemma}\label{lem:lamb-brute-force}
    Let $\gamma\colon[0,T)\times[-1,1]\to\R^n$ be a FLFB-EF and $L_0=\Ll(\gamma_0)$. If there exists $\delta>0$ with $\delta < \E(\gamma(t)) < \delta^{-1}$, then, for any $m\in\N_0$,
    \begin{equation}
        |\lambda(\gamma(t,\cdot))|\leq C(L_0,\delta,n,m)\big(1+\|(\partial_s^{\bot})^{m+2}\curv\|_{L^2(\dd s)}^{\frac{2}{m+2}}\big)\quad\text{for all $0\leq t<T$}.
    \end{equation}
\end{lemma}

By a careful analysis of the scaling argument used in \cite{dziukkuwertschaetzle2002} to control $\lambda$, arguing as in \cite[Lemma~4.3]{dallacqualinpozzi2017}, we have the following.

\begin{lemma}\label{lem:lamb-a priori}
    Let $\gamma\colon[0,T)\times[-1,1]\to\R^n$ be a FLFB-EF and $L_0=\Ll(\gamma_0)$. Suppose that 
    \begin{equation}\label{eq:pos-curv-cond}
        L_0 - |\gamma(t,1)-\gamma(t,-1)| \geq \zeta >0
    \end{equation}
    for all $t\in[0,T)$. Then, for any $m\in\N_0$,
    \begin{equation}
        |\lambda(t)|\leq C(\zeta,n,m,L_0,\E(\gamma_0)) \big(1+\|\partial_t\gamma\|_{L^2(\dd s)} + \|(\partial_s^{\bot})^{m+2}\curv\|_{L^2(\dd s)}^{\frac{1}{m+2}}\big)\quad\text{for all $0\leq t<T$}.
    \end{equation}
\end{lemma}

\begin{remark}\label{rem:pos-curv-cond}
    Notice that we work with two related but not equivalent uniform non-flatness conditions in \Cref{lem:lamb-brute-force,lem:lamb-a priori}. More precisely, let $L_0>0$ and $\gamma\colon[0,T)\times [-1,1]\to\R^n$ be a family of immersions of length $L_0$. We consider the following two conditions:
    \begin{itemize}
        \item[(A)] There exists $\delta>0$ with $\E(\gamma(t,\cdot))>\delta$ for all $0\leq t<T$. 
        \item[(B)] There exists $\zeta>0$ with $|\gamma(t,1)-\gamma(t,-1)|\leq L_0-\zeta$ for all $0\leq t<T$.  
    \end{itemize}
    It is a simple application of the direct method to verify that (B) implies (A). In general, the converse is not as straightforward. For a discussion in our setting of curves with bounded elastic energy satisfying orthogonal free boundary conditions, see \Cref{prop:unif-non-flatness} below.

    We would like to point out that it is the \emph{uniformity}, i.e.\ the uniform estimates with $\delta$ resp.\ $\zeta$ in (A) resp.\ (B), which is crucial here. In fact, $\E(\gamma(t,\cdot))>0$ for all $0\leq t<T$ if and only if $|\gamma(t,1)-\gamma(t,-1)| < L_0$ for all $0\leq t<T$ is clearly true.
\end{remark}


\subsection{Energy estimates for $\|(\partial_s^{\bot})^3\curv\|_{L^2(\dd s)}$ and $\|\partial_s^3\curv\|_{L^2(\dd s)}$}


For the remainder of this section, the following hypothesis plays a major role in the analysis of the flow's long-time behavior.

\begin{hypothesis}\label{hyp:L0-zeta}
    Let $\gamma\colon[0,T)\times[-1,1]\to\R^n$ be a FLFB-EF, $L_0=\Ll(\gamma_0)$ and suppose that there exists $\zeta>0$ with
    \begin{equation}\label{eq:unif-non-flatness-condition}
        L_0-|\gamma(t,-1)-\gamma(t,1)|\geq \zeta\quad\text{for all $0\leq t<T$}.
    \end{equation}
\end{hypothesis}

We first prove the following proposition.

\begin{proposition}\label{prop:dsbot3curv-l2-bound}
    Assume \Cref{hyp:L0-zeta}. Then, we have
    \begin{equation}\label{eq:dsbot3-estimate}
        \|(\partial_s^{\bot})^3\curv\|_{L^2(\dd s)} \leq C(L_0,\E(\gamma_0),\zeta,n)\quad\text{on $[0,T)$}.
    \end{equation}
    Especially, $|\lambda(\gamma(t,\cdot))|\leq C(L_0,\E(\gamma_0),\zeta,n)$ for all $0\leq t<T$.
\end{proposition}

Before proving \Cref{prop:dsbot3curv-l2-bound}, we collect some technical computations. Let $\gamma$ be as in \Cref{hyp:L0-zeta} and define 
\begin{align}
	N_3(t,x) &= (\partial_s^{\bot})^3\curv(t,x) + \langle\partial_s\gamma(t,x),\xi(\gamma(t,x))\rangle\cdot (S_{\gamma(t,x)}(\partial_s^{\bot})^2\curv(t,x))^{\bot} \\
	&\quad + \langle \curv(t,x),\partial_s^{\bot}\curv(t,x)\rangle \curv(t,x)
\end{align}
with $\xi$ as in \eqref{eq:defxi} and using \Cref{rem:extshape}. We observe that 
\begin{equation}\label{eq:defN3}
N_3= (\partial_s^{\bot})^3\curv  +\mathbb{P}_1^{2,2}= \mathbb{P}_1^{3,3}
\end{equation}
and, most importantly, that $N_3(t,\pm1)=0$ for $0\leq t<T$ by \Cref{lem:n3}.

\begin{claim}
    We have 
    \begin{equation}
        \int_{\S^2} |(\partial_s^{\bot})^5\curv|^2\dd s  \leq C\int_{\S^2} |(\partial_s^{\bot})^2N_3|^2\dd s + C(L_0,\E(\gamma_0),n).
    \end{equation}
\end{claim}
\begin{proof}
    Using 
    \begin{equation}
        (\partial_s^{\bot})^5\curv = (\partial_s^{\bot})^2N_3 + \mathbb{P}_1^{4,4},
    \end{equation}
    the claim follows from \Cref{cor:inter-ppABC}.
\end{proof}

\begin{claim}
    Writing $Y=\partial_t^{\bot}N_3+(\partial_s^{\bot})^4N_3$ and using \eqref{eq:DefV}, we have
    \begin{equation}
        \int_{\S^2} \langle Y,N_3\rangle \dd s - \frac12\int_{-1}^1 |N_3|^2\langle\curv,V\rangle\dd s = \int_{-1}^1 \langle \mathbb{P}_1^{6,5},N_3\rangle \dd s+\lambda \int_{-1}^1 \langle N_3,(\partial_s^{\bot})^5\curv + \mathbb{P}_1^{4,4}\rangle \dd s.
    \end{equation}
\end{claim}
\begin{proof}
    Using \Cref{lem:control-Y}, 
    \begin{align}
        Y=\big[\partial_t^{\bot}+(\partial_s^{\bot})^4\big]\big((\partial_s^{\bot})^3\curv + P_{1}^{2,2} + P^{1,1}_3\big)= \lambda \big((\partial_s^{\bot})^5\curv + \mathbb{P}_1^{4,4}\big) + \mathbb{P}_1^{6,5}.
    \end{align}
    Then the claim immediately follows.
\end{proof}
\begin{claim}
    Finally, for each $\varepsilon\in(0,1)$, there is a constant $C_\varepsilon=C(\varepsilon,L_0,\E(\gamma_0),\zeta,n)$ with
    \begin{equation}
        \langle \partial_s^{\bot}N_3,(\partial_s^{\bot})^2N_3\rangle \Big|_{-1}^1 \leq \varepsilon \int_{-1}^1|(\partial_s^{\bot})^5\curv|^2\dd s + C_{\varepsilon}.
    \end{equation}
\end{claim}
\begin{proof}
    Throughout this proof, $C$ (resp.\ $C_\varepsilon$) is a constant depending only on $L_0$, $\E(\gamma_0)$, $\zeta$ and $n$ (and $\varepsilon$, resp.), changing from line to line. 
    Using \eqref{eq:5o-3}, on $[0,T)\times\{\pm1\}$,
    \begin{align}
        \langle \partial_s^{\bot}N_3,(\partial_s^{\bot})^2N_3\rangle\Big|_{-1}^1 &= \langle \mathbb{P}_1^{4,4}, (\partial_s^{\bot})^5\curv + \mathbb{P}_1^{4,4}\rangle\Big|_{-1}^1 = \langle \mathbb{P}_1^{4,4}, \mathbb{P}_1^{4,4}\rangle + \lambda \langle \mathbb{P}_1^{4,4},\mathbb{P}_3^{1,1}\rangle \ \Big|_{-1}^1 \\
        &= \int_{-1}^1 \langle \mathbb{P}_1^{5,5},\mathbb{P}_1^{4,4} \rangle \dd s + \lambda \int_{-1}^1 \langle \mathbb{P}_1^{5,5},\mathbb{P}_3^{1,1} \rangle+\langle \mathbb{P}_1^{4,4},\mathbb{P}_3^{2,2} \rangle\dd s.
    \end{align}
    We now focus on the terms involving $\lambda$. Notice that we can estimate $$|\langle \mathbb{P}_1^{5,5},\mathbb{P}_3^{1,1} \rangle|\leq  |\langle (\partial_s^\bot)^5\curv,\mathbb{P}_3^{1,1} \rangle | + |\mathbb P_4^{6,4}|.$$ 
    Therefore, applying Young's inequality yields 
    \begin{align}
        \lambda &\int_{-1}^1 \langle \mathbb{P}_1^{5,5},\mathbb{P}_3^{1,1} \rangle+\langle \mathbb{P}_1^{4,4},\mathbb{P}_3^{2,2} \rangle\dd s\\
        &\leq \frac13\varepsilon \int_{-1}^1 |(\partial_s^\bot)^5\curv|^2\dd s + C_{\varepsilon} \lambda^2\int_{-1}^1 |\mathbb P_6^{2,1}|\dd s + |\lambda| \int_{-1}^1 |\mathbb P_4^{6,4}|\dd s.
    \end{align}
    Using \Cref{lem:lamb-brute-force} (with \Cref{rem:pos-curv-cond} and \eqref{eq:unif-non-flatness-condition}) and \Cref{cor:inter-ppABC},
    \begin{align}
         C_{\varepsilon} &\lambda^2\int_{-1}^1 |\mathbb P_6^{2,1}|\dd s + |\lambda| \int_{-1}^1 |\mathbb{P}_4^{6,4}|\dd s \\
         &\leq C_\varepsilon\big(1+\|(\partial_s^{\bot})^5\curv\|_{L^2(\dd s)}^{\frac45}\big)\big(1+\|(\partial_s^{\bot})^5\curv\|_{L^2(\dd s)}^{\frac45}\big) \\
         &\qquad + C\big(1+\|(\partial_s^{\bot})^5\curv\|_{L^2(\dd s)}^{\frac25}\big)\big(1+\|(\partial_s^{\bot})^5\curv\|_{L^2(\dd s)}^{\frac75}\big) \leq \frac13\varepsilon \int_{-1}^1 |(\partial_s^{\bot})^5\curv|^2\dd s + C_{\varepsilon}.
    \end{align}
    The claim then follows after applying \Cref{cor:inter-ppABC} to the remaining term $\int_{-1}^1\langle \mathbb{P}_1^{5,5},\mathbb{P}_1^{4,4} \rangle \dd s$.
\end{proof}

\begin{proof}[Proof of \Cref{prop:dsbot3curv-l2-bound}]
    In the following, write $N=N_3$ and denote by $C$ a constant changing from line to line which may depend on $L_0$, $\E(\gamma_0)$, $\zeta$ and $n$. Combining \Cref{lem:ibp} with the above claims, especially using that $N=0$ on $[0,T)\times\{\pm1\}$, we get
    \begin{equation}
        \frac{\dd}{\dd t} \frac12\int_{-1}^1|N|^2\dd s + \frac34\int_{-1}^1 |(\partial_s^{\bot})^5\curv|^2\dd s\leq \int_{-1}^1\langle N,\mathbb{P}_1^{6,5}\rangle\dd s + \lambda \int_{-1}^1 \langle N,(\partial_s^{\bot})^5\curv + \mathbb{P}_1^{4,4}\rangle\dd s+C.
    \end{equation}
    Especially, using Young's inequality in the $\lambda$-term on the right hand side,
    \begin{align}
        \frac{\dd}{\dd t} \frac12\int_{-1}^1|N|^2\dd s & + \frac12\int_{-1}^1 |(\partial_s^{\bot})^5\curv|^2\dd s\\
        &\leq \int_{-1}^1\langle N,\mathbb{P}_1^{6,5}\rangle\dd s + C \lambda^2 \int_{-1}^1 |N|^2\dd s + C \int_{-1}^1\langle \mathbb{P}_1^{4,4},\mathbb{P}_1^{4,4}\rangle\dd s+C.
    \end{align}
    By \Cref{cor:inter-ppABC} and \eqref{eq:defN3}, this yields
    \begin{align}
        \frac{\dd}{\dd t} \frac12\int_{-1}^1|N|^2\dd s + \frac14\int_{-1}^1 |(\partial_s^{\bot})^5\curv|^2\dd s \leq C \lambda^2 \int_{-1}^1 |N|^2\dd s + C.\label{eq:int3o-1}
    \end{align}
    With \Cref{lem:lamb-a priori,cor:inter-ppABC}, using \Cref{hyp:L0-zeta},
    \begin{align}
        C\lambda^2 \int_{-1}^1|N|^2\dd s &\leq C (1+\|V\|_{L^2(\dd s)}^2 + \|(\partial_s^{\bot})^5\curv\|_{L^2(\dd s)}^{\frac25}) \int_{-1}^1|N|^2\dd s\\
        &\leq C\|V\|_{L^2(\dd s)}^2 \int_{-1}^1|N|^2\dd s + C (1+ \|(\partial_s^{\bot})^5\curv\|_{L^2(\dd s)}^{\frac25})(1+\|(\partial_s^\bot)^5\curv\|^{\frac65}_{L^2(\dd s)})\\
        &\leq \frac18 \int_{-1}^1|(\partial_s^{\bot})^5\curv|^2\dd s + C\|V\|_{L^2(\dd s)}^2 \int_{-1}^1|N|^2\dd s + C
    \end{align}
    with Young's inequality so that \eqref{eq:int3o-1} yields
    \begin{align}
        \frac{\dd}{\dd t} \frac12\int_{-1}^1|N|^2\dd s + \frac18\int_{-1}^1 |(\partial_s^{\bot})^5\curv|^2\dd s \leq C \|V\|^2_{L^2(\dd s)} \int_{-1}^1 |N|^2\dd s + C.\label{eq:int3o-2}
    \end{align}
    Finally, adding $\int_{-1}^1|N|^2\dd s$ on both sides of \eqref{eq:int3o-2} and again using \Cref{cor:inter-ppABC} to estimate $$\int_{-1}^1|N|^2\dd s\leq \frac18 \int_{-1}^1|(\partial_s^{\bot})^5\curv|^2\dd s+C$$ on the right hand side of \eqref{eq:int3o-2} yields
    \begin{equation}
        \frac{\dd}{\dd t} \int_{-1}^1|N|^2\dd s \leq \big(C\|V\|_{L^2(\dd s)}^2-1\big) \int_{-1}^1|N|^2\dd s + C.
    \end{equation}
    Now $\int_{-1}^1|N|^2\dd s\leq C$ on $[0,T)$ follows with Gronwall's inequality, using $\int_0^T\|V\|_{L^2(\dd s)}^2\dd t \leq \E(\gamma_0)$, cf.\ \Cref{rem:en-mon}. The estimate \eqref{eq:dsbot3-estimate} follows from \Cref{cor:inter-ppABC}, arguing as in Claim~1. Finally, the bound for $\lambda$ follows from \Cref{lem:lamb-brute-force} with $m=1$ (using \Cref{rem:pos-curv-cond} and \eqref{eq:unif-non-flatness-condition}).
\end{proof}

\begin{corollary}\label{cor:ds3curv-l2-bound}
    Assume \Cref{hyp:L0-zeta}. Then,
    \begin{equation}
        \sum_{k=0}^3\|\partial_s^k\curv\|_{L^2(\dd s)}\leq C(L_0,\E(\gamma_0),\zeta,n)\quad\text{on $[0,T)$}.
    \end{equation}
\end{corollary}
\begin{proof}
    The case $k=0$ is clear since $\|\curv\|_{L^2(\dd s)}^2=\E(\gamma)$. Let $k\in\{1,2,3\}$. By \eqref{eq:tens-alg-4},
    \begin{equation}\label{eq:ds3curv-l2-bound-help}
        \partial_s^k\curv = (\partial_s^\bot)^k\curv + \mathbb{P}_2^{k-1,k-1}.
    \end{equation}
    Moreover, combining \Cref{prop:inter-dspk,prop:dsbot3curv-l2-bound}, $$\|(\partial_s^\bot)^k\curv\|_{L^2(\dd s)}\leq C(\E(\gamma_0),L_0,n)(1+\|(\partial_s^\bot)^3\curv\|_{L^2(\dd s)}) \leq C(L_0,\E(\gamma_0),\zeta,n).$$
    Thus, the claim follows from \eqref{eq:ds3curv-l2-bound-help} with \Cref{cor:inter-ppABC} since $\frac{2k-1}{k}<2$.
\end{proof}

\subsection{Global existence and subconvergence}\label{sec:asymptotics}
Let $\gamma\colon[0,T)\times[-1,1]\to\R^n$ be a time-maximal FLFB-EF and $L_0=\Ll(\gamma_0)$. Consider the smooth family of positive (that is, orientation-preserving) diffeomorphisms of $[-1,1]$ given by
\begin{equation}\label{eq:phipsi}
    \phi\colon[0,T)\times[-1,1]\to[-1,1], \quad \phi(t,\cdot)^{-1}\big|_y \vcentcolon= \psi(t,y) \vcentcolon= \frac{2}{L_0}\int_{-1}^y |\partial_x\gamma(t,x)|\dd x - 1.
\end{equation}
Throughout this section, denote $\tilde\gamma(t,x)=\gamma(t,\phi(t,x))$ and note that $|\partial_x\tilde\gamma|\equiv \frac{L_0}{2}$, so $\tilde \gamma$ is the family of (positive) reparametrizations by constant speed. The following a priori bounds for this special parametrization can be concluded from the previous section's energy estimates.

\begin{lemma}\label{lem:prop-of-arclengthrep}
    Assume \Cref{hyp:L0-zeta}. Then $\|\partial_x^k\tilde\gamma(t,\cdot)\|_{L^{2}((-1,1))}\leq C(k,L_0,\E(\gamma_0),\zeta,n)$ for all $k=1,\dots,5$ and for all $0\leq t<T$. Moreover, for any sequence $t_j\nearrow T$, $j\in \N$, $$p_j=\begin{cases}
        0&\text{if $M$ is compact}\\
        \gamma(t_j,1)&\text{otherwise}
    \end{cases}$$ and $\delta>0$, we have
    \begin{equation}\label{eq:prop-of-arclengthrep}
        \|\tilde \gamma(t_j+\cdot,\cdot)-p_j\|_{L^\infty(0,\min\{\delta,T-t_j\};W^{5,2}([-1,1]))} \leq C(\delta,L_0,\E(\gamma_0),\zeta,n)\quad\text{for all $j\in\N$}.
    \end{equation}
\end{lemma}
\begin{proof}
    Since $|\partial_x\tilde\gamma|\equiv \frac{L_0}2$, one computes $\partial_{\tilde s}^j \curv_{\tilde\gamma} = (2/L_0)^{j+2}\partial_x^{j+2}\tilde\gamma$ for $j\in\N_0$. In particular, since
    \begin{equation}
        \int_{-1}^{1} |\partial_s^k\curv_{\gamma}(t,\cdot)|^2\dd s = \int_{-1}^{1} |\partial_{\tilde s}^k\curv_{\tilde \gamma}(t,\cdot)|^2\dd \tilde s
    \end{equation}
    for all $k\in\N_0$, \Cref{cor:ds3curv-l2-bound} yields the desired $L^2$-bound for the space-derivatives of $\tilde\gamma$. 

    Using \Cref{prop:dsbot3curv-l2-bound}, interpolating with \Cref{prop:inter-dspk} with $p=\infty$ and using \eqref{eq:DefV}
    , we find 
    \begin{equation}\label{eq:prop-of-arclengthrep-2}
        \|\partial_t\gamma\|_{L^\infty((0,T)\times I)}\leq C(L_0,\E(\gamma_0),\zeta,n)
    \end{equation}
    so that 
    \begin{equation}\label{eq:prop-of-arclengthrep-1}
        \gamma\in L^\infty((0,T')\times(-1,1))\quad\text{for all $T'\in (0,T]\cap (0,\infty)$}
    \end{equation}
    by the fundamental theorem of calculus since $T'$ is finite. Moreover, using $\gamma(t_j,1)\in M$ and $|\gamma(t_j,x)-\gamma(t_j,1)|\leq L_0$ for all $j\in\N$ and $x\in[-1,1]$, there exists $R=R(\delta,L_0,\E(\gamma_0),\zeta,n,M)$ such that $|\gamma(t,x)-p_j|\leq R$ for all $t_j\leq t < \min\{t_j+\delta,T\}$, for all $j\in\N$. This yields \eqref{eq:prop-of-arclengthrep}.
\end{proof}

\begin{proposition}\label{prop:lte}
    Let $\gamma\colon[0,T)\times[-1,1]\to\R^n$ be a time-maximal FLFB-EF and assume \Cref{hyp:L0-zeta}. Then $T=\infty$.
\end{proposition}
\begin{proof}
    For the sake of contradiction, suppose that $T<\infty$. 
    %
    Using \Cref{lem:prop-of-arclengthrep} and \eqref{eq:prop-of-arclengthrep-1}, we have $\tilde\gamma\in L^\infty(0,T;W^{5,2}(I))$. 
    
    For each $t\in[0,T)$, denote by $\theta(t,\cdot)$ the diffeomorphism in \Cref{lem:geo-vs-ana-fbc}(b). In particular, $\sup_{0\leq t<T}\|\theta(t,\cdot)\|_{C^5([-1,1])}<\infty$ so that $\sup_{0\leq t<T}\|\tilde\gamma(t,\theta(t,\cdot))\|_{W^{5,2}(I)}<\infty$. 
    
    Let $\alpha\in(0,\frac12)$. The compact Sobolev embedding $W^{5,2}((-1,1))\hookrightarrow C^{4+\alpha}([-1,1])$ shows that there exist $t_j\nearrow T$ and $\tilde\gamma_T\in C^{4+\alpha}(I)$ with $\tilde\gamma(t_j,\theta(t_j,\cdot))\to \tilde\gamma_T$ in $C^{4+\alpha}(I)$. In particular, since each $\tilde\gamma(t_j,\theta(t_j,\cdot))$ satisfies \eqref{eq:ana-fbc}, so does $\tilde\gamma_T$. Furthermore, using \Cref{hyp:L0-zeta}, we have $\E(\tilde\gamma_T)>0$.

    Applying \Cref{thm:wp-ana-prob,thm:para-smooth} at the ``initial datum'' $\tilde\gamma_T$ and using $\tilde\gamma(t_j,\theta(t_j,\cdot))\to\tilde\gamma_T$, there exist $J\in\N$ and $\delta>0$ and a solution 
    \begin{equation}
        \hat\gamma \in W^{1,p}(t_J,T+\delta;L^p(I))\cap L^p(t_J,T+\delta;W^{4,p}(I)) \cap C^\infty((t_J,T+\delta)\times I)
    \end{equation}
    of the analytic problem \eqref{eq:ana-prob} with $\hat\gamma(t_J,\cdot)=\tilde\gamma(t_J,\theta(t_J,\cdot))$. By \Cref{lem:ex-rep}, there exists $t_J'\in (t_J,T)$ and a family 
    \begin{equation}
        \alpha\in W^{1,p}(t_J,t_J';L^p(I))\cap L^p(t_J,t_J';W^{4,p}(I))
    \end{equation}
    of positive $C^1$-diffeomorphisms of $I$ which is smooth on $(t_J,t_J')\times I$ such that $(t,y)\mapsto \gamma(t,\alpha(t,y))$ solves the analytic equation \eqref{eq:ana-prob} and $\gamma(t_J,\alpha(t_J,\cdot))=\tilde\gamma(t_J,\theta(t_J,\cdot))$. Using that the solution of the analytic equation is unique by \Cref{thm:wp-ana-prob}, we have 
    \begin{equation}\label{eq:pf-lte-1}
        \hat\gamma(t,y)=\gamma(t,\alpha(t,y))\quad\text{for all $t\in[t_J,t_J']$ and $y\in I$}.
    \end{equation}
    Let $\tau_J=\frac12(t_J+t_J')$. By \Cref{prop:ex-rep-ana-to-geo}, there exists $\eta\colon[\tau_J,T+\delta)\times I\to I$, a smooth family of diffeomorphisms, such that $(t,x)\mapsto \hat\gamma(t,\eta(t,x))$ solves \eqref{eq:ev-eq} with 
    \begin{equation}\label{eq:pf-lte-2}
        \eta(\tau_J,\cdot) = \alpha(\tau_J,\cdot)^{-1}.
    \end{equation}
    Using \eqref{eq:pf-lte-1} and the fact that $\hat\gamma(t,\eta(t,x))$ solves \eqref{eq:ev-eq}, we have 
    \begin{equation}
        \partial_t\big(\gamma(t,\alpha(t,\eta(t,\cdot)))\big) = - \big(\nabla\E(\gamma(t,\alpha(t,\eta(t,\cdot)))) -\lambda(\gamma(t,\alpha(t,\eta(t,\cdot))))\curv_{\gamma(t,\alpha(t,\eta(t,\cdot)))}\big)
    \end{equation}
    for all $t\in[\tau_J,t_J')$. Thus, since $\partial_t\gamma = -\big(\nabla\E(\gamma) -\lambda(\gamma)\curv_{\gamma}\big)$ and using that the right hand side is geometric, i.e.\ invariant with respect to reparametrizations, the chain rule yields
    \begin{equation}\label{eq:pf-lte-3}
        \partial_x\gamma(t,\alpha(t,\eta(t,\cdot)))\cdot\partial_t\big(\alpha(t,\eta(t,\cdot))\big) = 0\quad\text{for all $t\in [\tau_J,t_J')$}.
    \end{equation}
    Combined with \eqref{eq:pf-lte-2}, \eqref{eq:pf-lte-3} yields $\alpha(t,\cdot)^{-1}=\eta(t,\cdot)$ for all $t\in[\tau_J,t_J')$ and therefore, by \eqref{eq:pf-lte-1}, $\hat\gamma(t,\eta(t,\cdot))=\gamma(t,\cdot)$ for all $t\in [\tau_J,t_J')$. So we can smoothly extend $\gamma$ onto $[0,T+\delta)\times I$ by gluing with $\hat\gamma(t,\eta(t,\cdot))$ on $[\tau_J,T+\delta)$ which shows that $\gamma$ is not maximal, a contradiction!
\end{proof}

\begin{proposition}\label{prop:subc}
    Let $\gamma\colon[0,\infty)\times[-1,1]\to\R^n$ be a FLFB-EF and assume \Cref{hyp:L0-zeta}. Then the flow subconverges to solutions of the stationary problem as follows. For every $t_j\nearrow \infty$, there exists a smooth elastica $\tilde\gamma_\infty$ satisfying $\Ll(\tilde\gamma_\infty)=L_0$ such that the (positive) reparametrizations of $\gamma(t_j,\cdot)$ by constant speed converge to $\tilde\gamma_\infty$ in $C^{4+\alpha}([-1,1])$ up to subsequence and translation, for $j\to\infty$ where $\alpha\in(0,\frac12)$ is arbitrary.

    If $M$ is compact, $\tilde\gamma_\infty$ satisfies \eqref{eq:orth-bc} and \eqref{eq:nat-bc} and $\gamma(t_j,\cdot)\to\tilde\gamma_\infty$ in $C^{4+\alpha}([-1,1])$ up to subsequence and reparametrization.
\end{proposition}
\begin{proof}
    %
    Fix any $T\in(0,\infty)$, recall the notation of \Cref{lem:prop-of-arclengthrep} and define the sequence
    \begin{equation}
        \tilde\gamma_j\colon[0,T]\times[-1,1]\to\R^n,\quad\tilde \gamma_j(t,x)=\tilde\gamma(t_j+t,x)-p_j.
    \end{equation}
    Recall that $p_j=0$ for all $j\in\N$ if $M$ is compact.

    \textbf{Step 1.} We have that 
    \begin{equation}\label{eq:glob-ex-step1}
        \sup_{j\in\N} \|\tilde\gamma_j\|_{L^\infty(0,T;W^{5,2}(-1,1))}+\|\partial_t\tilde\gamma_j\|_{L^\infty(0,T;L^\infty(-1,1))} < \infty.
    \end{equation}

    By \eqref{eq:prop-of-arclengthrep}, $(\tilde\gamma_j)_{j\in\N}$ is bounded in $L^\infty(0,T;W^{5,2}(-1,1))$. Next, note that by \eqref{eq:phipsi}
    \begin{align}
        0&=\partial_t \big(\phi(t,\psi(t,y))\big)=\partial_t\phi(t,\psi(t,y)) 
        +\partial_x\phi(t,\psi(t,y))\partial_t\psi(t,y) \\
        &= \partial_t\phi(t,\psi(t,y)) + \frac{L_0}{2|\partial_x\gamma(t,y)|}\partial_t\psi(t,y) .
    \end{align}
    Moreover, 
    \begin{equation}\label{eq:sc-s1-1}
        \partial_t\psi(t,y) = \frac{2}{L_0} \partial_t \int_{-1}^y 1 \dd s = -\frac{2}{L_0} \int_{-1}^y \langle\curv,\partial_t\gamma\rangle \dd s
    \end{equation}
    and thus, by \eqref{eq:prop-of-arclengthrep-2}, $\|\partial_t\psi(t,\cdot)\|_{L^\infty([-1,1])} \leq C(L_0,\E(\gamma_0),\zeta,n)$. In particular, we obtain
    \begin{align}
        \partial_t\tilde\gamma(t,x)&= \partial_t\gamma(t,\phi(t,x)) + \partial_x\gamma(t,\phi(t,x))\partial_t\phi(t,x)\\
        &= \partial_t\gamma(t,\phi(t,x)) - \frac{L_0}{2}\frac{\partial_x\gamma(t,\phi(t,x))}{|\partial_x\gamma(t,\phi(t,x))|} \partial_t\psi(t,\phi(t,x))\label{eq:sc-s1-2}
    \end{align}
    so that $\|\partial_t\tilde\gamma(t,x)\|_{L^\infty((0,\infty)\times(-1,1))}\leq C(L_0,\E(\gamma_0),\zeta,n)$ by the above and \eqref{eq:prop-of-arclengthrep-2}. This yields \emph{Step~1}. 

    \textbf{Step 2.} Fix any $\alpha\in(0,\frac12)$. There exists a family of immersions $\tilde\gamma_\infty\in C^0([0,T];C^{4+\alpha}([-1,1]))$ parametrized by constant speed such that, after passing to a subsequence without relabeling, 
    \begin{equation}
        \tilde\gamma_j \to \tilde\gamma_\infty \quad\text{in $C^{0}([0,T];C^{4+\alpha}([-1,1]))$}.
    \end{equation}

    \emph{Step 2} is a consequence of the Aubin--Lions--Simon Lemma \cite[Corollary~4]{simon1987}. Indeed, we have $W^{5,2}(-1,1)\hookrightarrow C^{4+\alpha}([-1,1])$ compactly and $ C^{4+\alpha}([-1,1])\hookrightarrow L^\infty(-1,1)$. Moreover, \eqref{eq:glob-ex-step1} especially yields for $q>1$
    \begin{equation}
        \sup_{j\in\N}\|\tilde\gamma_j\|_{L^\infty(0,T;W^{5,2}(-1,1))}+\|\partial_t\tilde\gamma_j\|_{L^q(0,T;L^\infty(-1,1))}< \infty
    \end{equation}
    so that we can take $p=\infty$, $q>1$, $X=W^{5,2}(-1,1)$, $B=C^{4+\alpha}([-1,1])$, and $Y=L^\infty(-1,1)$ in \cite[Corollary~4]{simon1987}. The fact that $\tilde\gamma_\infty$ is a family of immersions follows from $|\partial_x\tilde\gamma_j|\equiv \frac{L_0}{2}$.

    \textbf{Step 3.} We have that $\tilde\gamma_\infty(t,\cdot)$ is a smooth elastica for all $0\leq t\leq T$.

    By \emph{Step~2}, we have
    \begin{equation}\label{eq:sc-s3-1}
        \nabla\E(\tilde\gamma_j)-\lambda(\tilde\gamma_j(t,\cdot))\curv_{\tilde\gamma_j} \to \nabla\E(\tilde\gamma_\infty)-\lambda(\tilde\gamma_\infty(t,\cdot))\curv_{\tilde\gamma_\infty} \quad\text{in $C^0([0,T]\times[-1,1])$}.
    \end{equation}
    Moreover, we compute
    \begin{align}
        &\int_0^T \int_{-1}^{1} |\nabla\E(\tilde\gamma_j)(t,x)-\lambda(\tilde\gamma_j(t,\cdot))\curv_{\tilde\gamma_j}(t,x)|^2\dd\tilde s_j\dd t \\
        &= \int_{t_j}^{t_j+T}\int_{-1}^{1} |\nabla\E(\gamma(t,\cdot))-\lambda(\gamma(t,\cdot))\curv_{\gamma}(t,\cdot)|^2\dd s\dd t = \int_{t_j}^{t_j+T} \int_{-1}^{1} |\partial_t\gamma|^2\dd s\dd t \label{eq:sc-s3-2}\\
        &= \int_{t_j}^{t_j+T} (-\partial_t\E(\gamma(t,\cdot)))\dd t = \E(\gamma(t_j,\cdot))-\E(\gamma(t_j+T,\cdot)) \to 0 \label{eq:sc-s3-3}.
    \end{align}
    Together, \eqref{eq:sc-s3-1} and \eqref{eq:sc-s3-3} imply
    \begin{equation}\label{eq:sc-s3-5}
        \nabla\E(\tilde\gamma_\infty)(t,x)-\lambda(\tilde\gamma_\infty(t,\cdot))\curv_{\tilde\gamma_\infty}(t,x)=0 \quad\text{for all $(t,x)\in [0,T]\times[-1,1]$}.
    \end{equation}
    The claim follows up to smoothness of $\tilde\gamma_\infty(t,\cdot)$ for each $0\leq t\leq T$. The smoothness is obtained from \eqref{eq:sc-s3-5} via bootstrapping, using $\tilde\gamma_\infty(t,\cdot)\in C^{4+\alpha}([-1,1])$ and $|\partial_x\tilde\gamma_\infty(t,\cdot)|\equiv \frac12L_0$.
\end{proof}

\begin{proof}[Proof of \Cref{thm:intro-main-asymptotics}] 
The claim follows combining \Cref{prop:lte,prop:subc}.
\end{proof}

\subsection{Discussion of the uniform non-flatness hypothesis}\label{sec:non-flatness}

In order to make conclusive statements about the asymptotics of \eqref{eq:ev-eq}, we require the uniform non-flatness assumption in \Cref{hyp:L0-zeta} to obtain control of the nonlocal Lagrange multiplier. In this section, we discuss assumptions on the data or on $M$ under which this hypothesis is automatically satisfied.

First, recall the data for our problem: The hypersurface $M$ with unit normal field $\xi$, a mapping $\tau_0\colon\{\pm1\}\to\{\pm1\}$, a fixed length $L_0>0$ and an initial energy $E>0$. An immersion $\gamma\colon[-1,1]\to\R^n$ with $\Ll(\gamma)=L_0$, $\E(\gamma)\leq E$ satisfies the orthogonal boundary conditions determined by this data if
\begin{equation}\label{eq:orth-bc-tau0}
    \gamma(\pm1)\in M\quad\text{and}\quad\partial_s\gamma(\pm1)=\tau_0(\pm1)\xi(\gamma(\pm1)).
\end{equation}
We always assume that $E>0$ is chosen sufficiently large such that there exists an immersion $\gamma\in W^{2,2}([-1,1])$ with $\Ll(\gamma)=L_0$, $\E(\gamma)\leq E$ and \eqref{eq:orth-bc-tau0}.

\begin{proposition}\label{prop:unif-non-flatness}
    Let $M$ be flat or compact. Then
    \begin{align}
        E_{\min}\vcentcolon=\min \{\E(\gamma)\mid\ &\gamma\in W^{2,2}([-1,1],\R^n)\text{ is an immersion} \\
        &\text{ with $\Ll(\gamma)=L_0$ and satisfying \eqref{eq:orth-bc-tau0}}\}
    \end{align}
    exists and the following are equivalent.
    \begin{enumerate}[(a)]
        \item We have that $E_{\min}>0$.
        \item There exists $\zeta>0$ such that, for any immersion $\gamma\in W^{2,2}([-1,1],\R^n)$ with $\Ll(\gamma)=L_0$, $\E(\gamma)\leq E$ and \eqref{eq:orth-bc-tau0}, we have
        \begin{equation}
            L_0-|\gamma(1)-\gamma(-1)| \geq \zeta\quad\text{for all $t\in(0,T)$}.
        \end{equation}
    \end{enumerate}
    If $M$ is flat, then $E_{\min} = \frac{\pi^2}{L_0}$. Furthermore, if $M=\partial\Omega$ is the boundary of a compact convex set $\Omega$, $\xi$ is outward-pointing and $\tau_0(-1)=1$, then $E_{\min}>0$.
\end{proposition}

If $M$ is compact, notice that  $E_{\min}=0$ if and only if there exists a straight line of length $L_0$ satisfying \eqref{eq:orth-bc-tau0}. In many cases, not only in the flat or convex case considered in \Cref{prop:unif-non-flatness},  $E_{\min}>0$ is satisfied. For instance if $M$ is compact and $L_0>\diam (M)$. An example where $E_{\min}=0$ only for one value of the length is illustrated in \Cref{fig:example-nonconvex}. 

\begin{proof}[Proof of \Cref{prop:unif-non-flatness}]
    Proving the equivalence of (a) and (b) is a simple application of the direct method and compact Sobolev embeddings.

    Indeed, consider a sequence of immersions $(\gamma_j)_{j\in\N}\subseteq W^{2,2}([-1,1])$ with $\E(\gamma_j)\leq E$, $\Ll(\gamma_j)=L_0$, all satisfying \eqref{eq:orth-bc-tau0} and suppose that each $\gamma_j$ is parametrized by constant speed as at the beginning of \Cref{sec:asymptotics}, i.e., $|\partial_x \gamma_j(x)|=L_0/2$ for all $x \in [-1,1]$. If $M$ is compact, \eqref{eq:orth-bc-tau0} and $\Ll(\gamma_0)=L_0$ yield $\gamma_j([-1,1])\subseteq \bar B_{L_0}(M)$. If $M$ is flat, fix any $z\in M$ and notice that $z+\gamma_j-\gamma_j(-1)$ still satisfies \eqref{eq:orth-bc-tau0} and $z+\gamma_j(x)-\gamma_j(-1)\in \bar B_{L_0}(z)$ for all $j\in\N$. So replacing $\gamma_j$ by $z+\gamma_j(x)-\gamma_j(-1)$ in the noncompact case, we may without loss of generality suppose 
    \begin{equation}
        \|\gamma_j\|_{L^\infty}\leq C(M,L_0)\quad \text{for all $j\in\N$}.
    \end{equation}
    Moreover, using $\E(\gamma_j)\leq E$ and $\Ll(\gamma_j)=L_0$ and arguing as in the proof of \Cref{lem:prop-of-arclengthrep}, we thus have $\|\gamma_j\|_{W^{2,2}([-1,1])}\leq C(M,L_0,E)$ for all $j\in\N$. Fix $\alpha\in(0,\frac12)$. By the compact Sobolev embedding $W^{2,2}([-1,1])\hookrightarrow C^{1+\alpha}([-1,1])$, there exists $\gamma_\infty \in W^{2,2}([-1,1])$ such that, after passing to a subsequence, 
    \begin{equation}\label{eq:on-hyp-2}
        \gamma_j\rightharpoonup \gamma\text{ in $W^{2,2}([-1,1])$}\quad\text{and}\quad \gamma_j\to \gamma\text{ in $C^{1+\alpha}([-1,1])$}.
    \end{equation}
    In particular, $\gamma$ satisfies $|\partial_x\gamma|\equiv \frac{L_0}{2}>0$, $\Ll(\gamma)=L_0$, \eqref{eq:orth-bc-tau0} and we have
    \begin{equation}\label{eq:on-hyp-1}
        \E(\gamma)\leq\liminf_{j\to\infty}\E(\gamma_j)\leq E.
    \end{equation}
    Applying these arguments to a minimizing sequence $(\gamma_j)_{j\in\N}$ for $E_{\min}$, especially using \eqref{eq:on-hyp-1}, the minimum in the definition of $E_{\min}$ is indeed attained.

    \textbf{Claim 1.} We have $E_{\min}=0$ if and only if there exists an immersion $\gamma\in W^{2,2}(I)$ of length $L_0$ that parametrizes a straight line and satisfies \eqref{eq:orth-bc-tau0}.

    Let $\gamma\in W^{2,2}(I)$ be an immersion of length $L_0$ that satisfies \eqref{eq:orth-bc-tau0}. Taking $\gamma$ to be parametrized by constant speed, we have $\E(\gamma)=0$ if and only if $\partial_x^2\gamma=0$ which is true if and only if $\gamma$ parametrizes a line segment, i.e.\ if and only if $L_0-|\gamma(1)-\gamma(-1)|=0$.

    \textbf{Claim 2.} Statement (b) is false if and only if there exists an immersion $\gamma\in W^{2,2}([-1,1])$ of length $L_0$ parametrizing a straight line and satisfying \eqref{eq:orth-bc-tau0}.

    If (b) is false, we obtain a sequence $(\gamma_j)$ as considered above that additionally satisfies 
    \begin{equation}
        L_0-|\gamma_j(1)-\gamma_j(-1)|\to 0.
    \end{equation}
    Using \eqref{eq:on-hyp-1}, we get $L_0-|\gamma(1)-\gamma(-1)|=0$ for the limit $\gamma$ obtained above. In particular, $\gamma$ is a geodesic, i.e.\ parametrizes a straight line which yields \emph{Claim~2}.

    The equivalence of (a) and (b) is a direct consequence of \emph{Claims~1 and 2}. 

    Next, suppose that $M$ is flat and let $P\in\R^{n\times n}$ denote the reflection by $M$. Let $\gamma\in W^{2,2}([-1,1])$ be an immersion with $\Ll(\gamma)=L_0$ satisfying \eqref{eq:orth-bc-tau0}. Identify $\S^1=\R/4\Z$ and define $\mathcal R\gamma\colon\S^1\to\R^n$,
    \begin{equation}\label{eq:def-R-gamma}
        \mathcal R\gamma(x)=\begin{cases}
        \gamma(x)&x\in[-1,1]\\
        P\gamma(2-x)&x\in[1,3].
    \end{cases}
    \end{equation}
    Then \eqref{eq:orth-bc-tau0} yields $\mathcal R\gamma\in W^{2,2}(\S^1,\R^n)$ and, since $P$ is an isometry of $\R^n$, 
    \begin{equation}
       \E(\mathcal R\gamma) = 2\E(\gamma)\quad\text{and}\quad \Ll(\mathcal R\gamma)=2L_0.
    \end{equation}
    It is a classical application of Cauchy--Schwarz paired with Fenchel's Theorem in $\R^n$ 
    \cite{fenchel1929,borsuk1947} that the (up to reparametrization and isometries) unique minimizer of $\E$ among all $W^{2,2}(\S^1)$-immersions of length $2L_0$ is a round planar circle of radius $\frac{L_0}{\pi}$. So $\E(\gamma)\geq \frac{\pi^2}{L_0}$ follows.

    Finally, let $M=\partial\Omega$ be the boundary of a compact convex set $\Omega$, $\xi$ be outward pointing and $\tau_0(-1)=1$. Suppose that either statement (a) or (b) fails, i.e.\ by the above claims, there exists $z\in M$ such that $\gamma\colon[-1,1]\to\R^n$ with $\gamma(x)=z+\frac{L_0}{2}\xi(z)(x+1)$ satisfies \eqref{eq:orth-bc-tau0}. In particular, $\gamma(1)=z+L_0\xi(z)\in M$ which contradicts the convexity of $\Omega$.
\end{proof}

\subsection{Full convergence and classification of the limit in a special case}

In this section, we consider the special case where $M=\R\times\{0\}\subseteq\R^2$ is the $x$-axis in the plane. Moreover, choose $\xi(x)=(0,1)$ for $x\in M$. Then, using a reflection argument similar to \cite{kuwertlamm2021} and the classification of closed planar elastica \cite{langersinger1985}, we can obtain full convergence of the FLFB-EF and characterize the limit. 

To make this more precise, let $P\colon\R^2\to\R^2$ be the reflection by $M$, that is,
\begin{equation}
    P(x)= (x^{(1)},-x^{(2)})\quad\text{for $x=(x^{(1)},x^{(2)})\in\R^2$}.
\end{equation}
Identify $\S^1=\R / 4\Z$ and, for an immersion $\gamma\colon[-1,1]\to\R^2$ satisfying the orthogonal free boundary conditions \eqref{eq:orth-bc}, consider the reflection $\mathcal R\gamma\colon\S^1\to\R^2$ given by \eqref{eq:def-R-gamma}.

\begin{remark}[Regularity of the reflection]\label{rem:reg-of-reflection}
    Let $\gamma\colon[-1,1]\to\R^2$ be smooth and satisfy \eqref{eq:orth-bc}. Then $\mathcal R\gamma\in W^{2,\infty}(\S^1)$. Moreover, for $m\in\N$, we have
    \begin{equation}\label{eq:reg-of-refl-1}
        \partial_x^{2k}\gamma^{(2)}(\pm1)=\partial_x^{2k-1}\gamma^{(1)}(\pm1)=0 \text{ for all $1\leq k\leq m$},\quad \text{if and only if $\mathcal R\gamma\in W^{2m+1,\infty}(\S^1)$}.
    \end{equation}
    Now suppose that $\gamma$ is parametrized by constant speed, i.e., $|\partial_x\gamma|\equiv\mathrm{const}$. Let $m\in\N$.
    \begin{equation}\label{eq:reg-of-refl-2}
        \text{If } (\partial_s^\bot)^{2k-1}\curv(\pm1) = 0  \text{ for all $1\leq k\leq m$},\quad \text{then $\mathcal R\gamma\in W^{2m+3,\infty}(\S^1)$}.
    \end{equation}
    Indeed, without loss of generality, we may arrange that $|\partial_x\gamma|\equiv 1$. By \cite[Equation~(2.20)]{dziukkuwertschaetzle2002}, we have
    \begin{equation}\label{eq:reg-of-refl-4}
        \langle \partial_s^{2k}\curv,\partial_s\gamma\rangle = \sum_{i=1}^{k} Q_{2i}^{2k+1-2i}(\curv)
    \end{equation}
    where each term $Q_\nu^\mu$ is a scalar multilinear map in $(\partial_s^\bot)^{i_1}\curv, \dots ,(\partial_s^\bot)^{i_\nu}\curv$ where $i_1+\cdots+i_\nu=\mu$. In particular, if $k\leq m$, $\mu=2k+1-2i$ and $\nu=2i$, in each such term, there exists $1\leq j\leq \nu$ such that $i_j$ is odd and thus $(\partial_s^\bot)^{i_j}\curv(\pm1) = 0$ by \eqref{eq:reg-of-refl-2}. That is, \eqref{eq:reg-of-refl-4} yields
    \begin{equation}\label{eq:reg-of-refl-3}
        \langle \partial_s^{2k}\curv,\partial_s\gamma\rangle(\pm1) = 0\quad\text{for all $k=0,\dots,m$}.
    \end{equation}
    Moreover, again using \cite[Equation~(2.20)]{dziukkuwertschaetzle2002}, we have
    \begin{equation}
        (\partial_s^{2k-1}\curv)^\bot(\pm1) = (\partial_s^\bot)^{2k-1}\curv(\pm1) + \sum_{i=1}^{k-1} \vec Q_{2i+1}^{2k-1-2i} (\curv)(\pm1) \underset{\eqref{eq:reg-of-refl-2}}{=} \sum_{i=1}^{k-1} \vec Q_{2i+1}^{2k-1-2i} (\curv)(\pm1)
    \end{equation}
    where $\vec Q_\nu^\mu$ is the vectorial analog of $Q_\nu^\mu$. Again, if $k\leq m$, $\mu=2k-1-2i$ and $\nu=2i+1$, in each such term, there exists $1\leq j\leq \nu$ such that $i_j$ is odd and thus, using \eqref{eq:reg-of-refl-2},
    \begin{equation}\label{eq:reg-of-refl-5}
        (\partial_s^{2k-1}\curv)^\bot(\pm1) = 0\quad\text{for all $k=1,\dots,m$}.
    \end{equation}
    Using $\partial_s=\partial_x$, $\curv=\partial_x^2\gamma$ and \eqref{eq:orth-bc}, \eqref{eq:reg-of-refl-3} and \eqref{eq:reg-of-refl-5}, we can apply \eqref{eq:reg-of-refl-1} (with $m$ replaced by $m+1$).
\end{remark}

\begin{theorem}\label{thm:special-case}
    Consider a smooth immersion $\gamma_0\colon[-1,1]\to\R^2$ satisfying \eqref{eq:orth-bc} and \eqref{eq:nat-bc}. 
    Let $\gamma\colon[0,T)\times[-1,1]\to\R^2$ be a maximal FLFB-EF starting at $\gamma_0$. Then $T=\infty$ and, for the positive reparametrizations $\tilde\gamma(t,\cdot)$ by constant speed, we have $\tilde\gamma(t,\cdot)\to\tilde\gamma_\infty$ in the $C^\infty$-topology for $t\to\infty$ where $\tilde\gamma_\infty$ is an elastica that satisfies \eqref{eq:orth-bc} and \eqref{eq:nat-bc} and, either
    \begin{enumerate}
        \item $\mathcal R\tilde\gamma_\infty$ is a (multiple cover of a) round circle of length $2\cdot \Ll(\gamma_0)$; or
        \item $\mathcal R\tilde\gamma_\infty$ is a (multiple cover of a) figure-eight elastica of length $2\cdot\Ll(\gamma_0)$.
    \end{enumerate}
    Moreover, the turning numbers of $\mathcal R\tilde\gamma_\infty$ and $\mathcal R\gamma_0$ coincide. In particular, Case~(2) occurs if and only if the turning number of $\mathcal R\gamma_0$ is zero.
\end{theorem}
For some background on the figure-eight elastica, see \cite[Theorem~0.1]{langersinger1985}.
\begin{proof}[Proof of \Cref{thm:special-case}]
    Using \Cref{prop:lte,prop:unif-non-flatness}, $T=\infty$.  
    We have $\partial_s^\bot\curv(t,\pm1)=0$ by \eqref{eq:nat-bc} for all $t\geq 0$. Moreover, \Cref{lem:n3} yields $(\partial_s^\bot)^3\curv(t,\pm1)=0$ for all $t\geq 0$. So using \cite[Lemma~2.3]{dziukkuwertschaetzle2002}, differentiating in time and arguing by induction yields
    \begin{equation}
        (\partial_s^\bot)^{2k-1}\curv(t,\pm1)=0\quad\text{for all $k\in\N$ and $t\geq 0$}.
    \end{equation}
    Thus, \eqref{eq:reg-of-refl-2} shows that $\tilde u(t,\cdot)=\mathcal R\tilde\gamma(t,\cdot)$ is smooth for all $t\geq 0$. Using that $P$ is an isometry, one finds $\lambda(\tilde u(t,\cdot))=\lambda(\tilde\gamma(t,\cdot))$ for all $t\geq 0$. So, using
    \begin{equation}
        \partial_t^\bot \tilde\gamma = -\big(\nabla\E(\tilde\gamma)-\lambda(\tilde\gamma)\curv_{\tilde\gamma}\big),
    \end{equation}
    $\tilde u\in C^\infty([0,\infty)\times \S^1)$ is a family of reparametrizations by constant speed of the length-preserving elastic flow of closed curves starting at $\tilde u_0\vcentcolon=\mathcal R\tilde\gamma(0,\cdot)$.  

    Using \cite[Theorem~3.3]{dziukkuwertschaetzle2002}, $\tilde u(t,\cdot)$ subconverges to elastica (in the $C^\infty$-topology) up to horizontal translation for $t\to\infty$. Applying a length-constrained \L ojasiewicz--Simon inequality for \emph{closed} curves (see \cite{mantegazzapozzetta2021} for a version without length-constraint, and \cite[Theorem~4.10]{dallacquamuellerruppschlierf2025} for a version with length-constraint in $\H^2$), this improves to full convergence. Indeed, adapting \cite[proof of Theorem~4.8]{dallacquamuellerruppschlierf2025}, we get that $\tilde u(t,\cdot)$ converges to an elastica $\tilde u_\infty=\mathcal R\tilde\gamma_\infty$ in the $C^\infty$-topology as $t\to\infty$. Finally, \cite[Theorem~0.1]{langersinger1985} yields the claim, using that the turning number is preserved along $t\mapsto \tilde u(t,\cdot)$ by the Whitney--Graustein Theorem.
\end{proof}

\appendix
\section{Background on function spaces and embeddings, and technical proofs}

\subsection{Function spaces and embeddings}\label{app:fct-spaces}

In this section, we provide references for facts concerning the function spaces used in the short-time existence theory. To this end, let $E$ be a complex Banach space \emph{of class} $\mathcal{H}\mathcal{T}$. For our purposes, it is sufficient to note that finite dimensional spaces, Hilbert spaces and $L^p$, Sobolev, Slobodetskii, Besov and Bessel potential spaces in the reflexive range, provided they take values in a space of class $\mathcal{H}\mathcal{T}$, are of this class. For more details, refer to \cite[Appendix A.3]{meyries2010}. 

Consider any bounded open interval $J\subseteq\R$, $p,q\in[1,\infty)$ and $s\in (0,\infty)$. If $[s]\in\N_0$ denotes the unique integer with $[s]\leq s<[s]+1$, we define the \emph{Besov space}
\begin{equation}
    B^s_{p,q}(J,E) \vcentcolon= \left(L^p(J,E),W^{[s]+1,p}(J,E)\right)_{\frac{s}{[s]+1},q}
\end{equation}
resp.\ the \emph{Bessel potential space}
\begin{equation}
    H^{s,p}(J,E) \vcentcolon= \left[ L^p(J,E),W^{[s]+1,p}(J,E) \right]_{\frac{s}{[s]+1}}
\end{equation}
as real resp.\ complex interpolation spaces where $W^{k,p}(J,E)$ is the usual Bochner--Sobolev space for $k\in\N_0$. By \cite[(A.4.1)]{meyries2010}, the above assumptions on $E$ yield
\begin{equation}
    W^{k,p}(J,E) = H^{k,p}(J,E)\quad\text{for any }k\in\N,\,p\in[1,\infty).
\end{equation}
Furthermore, for $p\in [1,\infty)$ and $s>0$, the \emph{Sobolev--Slobodetskii space} $W^{s,p}(J,E)$ is defined by
\begin{equation}\label{eq:SobSlobSpaces}
    W^{s,p}(J,E)\vcentcolon=\begin{cases}
        W^{k,p}(J,E)&\text{if $s=k\in\N$}\\
        B^s_{p,p}(J,E)&\text{if $s\notin\N$}.
    \end{cases}
\end{equation}
By \cite[Theorem 2.3.2(d)]{triebel1978}, $B^k_{2,2}(J,\R^N)=W^{k,2}(J,\R^N)$ for $k\in\N$. Moreover, for $f\in L^{p}(J,E)$ and $\theta\in (0,1)$, define
\begin{equation}\label{eq:def-sobslob-seminorm}
    [f]_{W^{\theta,p}} \vcentcolon= \left( \int_{J}\int_J \frac{|f(x)-f(y)|_E^p}{|x-y|^{1+\theta p}} \dd x\dd y \right)^{\frac{1}{p}}.
\end{equation}
By \cite[Appendix A.4]{meyries2010}, for $s\in (0,\infty)\setminus\N$, we have
\begin{equation}
    W^{s,p}(J,E) = \{ u\in W^{[s],p}(J,E): [\partial_x^{[s]}u]_{W^{s-[s],p}}<\infty \}.
\end{equation}
Further, for $u\in W^{[s],p}(J,E)$,  
\begin{equation}
    \|u\|_{W^{s,p}(J,E)} \vcentcolon= \left( \|u\|_{W^{[s],p}(J,E)}^p + [\partial_x^{[s]}u]_{W^{s-[s],p}}^p \right)^{\frac{1}{p}}
\end{equation}
defines an equivalent norm on $W^{s,p}(J,E)$. Again by \cite[Appendix A.4]{meyries2010}, the usual Sobolev embeddings apply. That is
\begin{equation}
    W^{s,p}(J,E)\hookrightarrow W^{\tau,q}(J,E)\quad\text{for }s-\frac{1}{p}\geq \tau-\frac{1}{q},\;s\geq\tau,\;p\geq q,
\end{equation}
and
\begin{equation}\label{eq:morrey}
    W^{s,p}(J,E)\hookrightarrow C^{k,\alpha}(J,E)\quad\text{for }s-\frac{1}{p}\geq k+\alpha,\,k\in\N,\alpha\in [0,1).
\end{equation}

\subsubsection{Parabolic maximal $L^p$-regularity: Spaces, traces and embeddings}

For $T\in(0,\infty)$, $p\in(1,\infty)$ and a compact interval $I\subseteq \R$, define
\begin{equation}
    \mathbb X_{T,p} = W^{1,p}(0,T;L^p(I,\R^n))\cap L^p(0,T;W^{4,p}(I,\R^n)).
\end{equation}
This space equipped with the norm
\begin{equation}
    \|\cdot\|_{\mathbb X_{T,p}} \vcentcolon= \|\cdot \|_{W^{1,p}(0,T;L^p(I,\R^n))} + \|\cdot\|_{L^p(0,T;W^{4,p}(I,\R^n))},
\end{equation}
is a Banach space. 
As usual, restricting to subspaces of functions with vanishing temporal trace $\mathrm{tr}_{t=0}$ is indicated by a pre-script $0$, for instance $\prescript{}{0}{\mathbb X}_{T,p}$ or $\prescript{}{0}{W}^{s,p}([0,T],\R^n\times\R^n)$. 

The following extension lemma for functions with vanishing temporal trace is fundamental, see \cite[Lemma~3.10]{garckemenzelpluda2020} and \cite[definition of $Ev(t)$ on p.~441]{pruesssimonett2016}.
\begin{lemma}[Extension lemma]\label{lem:0xtp-extension-lemma}
    Let $p>5$ and $\bar T\in(0,\infty)$. For each $T\in(0,\bar T]$ and $\varphi\in \prescript{}{0}{\mathbb X}_{T,p}$, there exists $\bar\varphi\in \prescript{}{0}{\mathbb X}_{\bar T,p}$ with $\bar\varphi|_{[0,T]\times [-1,1]}=\varphi$ and 
    \begin{equation}
        \|\varphi\|_{\mathbb X_{T,p}}\leq \|\bar\varphi\|_{\mathbb X_{\bar T,p}}\leq 2 \|\varphi\|_{\mathbb X_{T,p}}.
    \end{equation}
\end{lemma}

\begin{proposition}\label{prop:xtp-mappings-and-embeddings}
    Let $T\in (0,\infty)$, $p\geq 2$ and let $0\leq k\leq 4$ be an integer. Then
    \begin{enumerate}[(i)]
        \item\label{it:buc-emb} $\mathbb X_{T,p}\hookrightarrow \mathrm{BUC}([0,T],B_{p,p}^{4(1-\frac1p)}(I,\R^n))$ with the estimate 
        \begin{equation}
            \|f\|_{\mathrm{BUC}([0,T],B_{p,p}^{4(1-\frac1p)}(I,\R^n))} \leq C(p) \|f\|_{\mathbb X_{T,p}}\quad\text{for all $f\in\prescript{}{0}{\mathbb X}_{T,p}$},
        \end{equation}
        \item\label{it:kder-emb} the $k$-th spacial derivative is continuous as a map
        \begin{align}
            \partial_x^k \colon\mathbb X_{T,p} &\to H^{1-\frac{1}{4}k,p}(0,T;L^p(I,\R^n))\cap L^p(0,T;H^{4-k,p}(I,\R^n))\\
            &\hookrightarrow W^{\theta(1-\frac14k),p}(0,T;W^{(1-\theta)(4-k),p}(I,\R^n))\quad\text{for all $\theta\in(0,1)$}\label{eq:mixed-sobolev-embedding}
        \end{align}
        with the estimate
        \begin{equation}
            \|\partial_x^kf\|_{H^{1-\frac{1}{4}k,p}(0,T;L^p(I,\R^n))\cap L^p(0,T;H^{4-k,p}(I,\R^n))} \leq C(p,k) \|f\|_{\mathbb X_{T,p}}\quad\text{for all $f\in\prescript{}{0}{\mathbb X}_{T,p}$},
        \end{equation}
        \item\label{it:trkder-emb} the spatial trace of the $k$-th spatial derivative is continuous as a map 
        \begin{equation}
            \mathrm{tr}_{\partial I}\partial_x^k \colon\mathbb{X}_{T,p} \to 
             W^{1-\frac14k-\frac{1}{4p},p}(0,T;\R^n\times\R^n)
        \end{equation}
        with the estimate
        \begin{equation}
            \|\mathrm{tr}_{\partial I}\partial_x^kf\|_{W^{1-\frac14k-\frac{1}{4p},p}(0,T;\R^n\times\R^n)} \leq C(p,k) \|f\|_{\mathbb X_{T,p}}\quad\text{for all $f\in\prescript{}{0}{\mathbb X}_{T,p}$},
        \end{equation}
        \item\label{it:hoeld-emb} for $p>5$ and $\sigma\in(\frac1p,\frac14(1-\frac1p))$, $\mathbb X_{T,p}\hookrightarrow C^{\sigma-\frac1p}([0,T],C^{3+(1-4\sigma-\frac1p)}(I,\R^n))$. Let $T'\in(0,\infty)$. For all $T\in(0,T')$, we have the estimate
        \begin{equation}
            \|f\|_{C^{\sigma-\frac1p}([0,T],C^{3+(1-4\sigma-\frac1p)}(I,\R^n))} \leq C(\sigma,p,T') \|f\|_{\mathbb X_{T,p}}\quad\text{for all $f\in\prescript{}{0}{\mathbb X}_{T,p}$}.
        \end{equation}
    \end{enumerate}
\end{proposition}
\begin{proof}
    Item \eqref{it:buc-emb} is a special case of \cite[Theorem~4.2]{meyriesschnaubelt2012}, Item \eqref{it:kder-emb} (except for the embedding \eqref{eq:mixed-sobolev-embedding}) follows from \cite[Lemma~3.4]{meyriesschnaubelt2012} and Item \eqref{it:trkder-emb} is a consequence of \eqref{it:kder-emb} together with \cite[(4.13) and Theorem~4.5]{meyriesschnaubelt2012}. The embedding \eqref{eq:mixed-sobolev-embedding} follows from \cite[Proposition~3.2]{meyriesschnaubelt2012}. Finally, \eqref{it:hoeld-emb} follows by combining \eqref{it:kder-emb} with $k=0$ and \eqref{eq:morrey} where the time-uniform estimate is due to \Cref{lem:0xtp-extension-lemma}.
\end{proof}

\begin{lemma}\label{lem:estimate-in-boundary-terms}
    Let $p\in(1,\infty)$, $s\in (\frac1p,1)$, and $T'>0$. Then, there exists $C=C(T',s,p)>0$ with
    \begin{equation}\label{eq:end1}
        \|\langle f, g\rangle \|_{W^{s,p}(0, T)} \leq C(T',s,p)([f]_{W^{s,p}(0, T)} + \|f\|_{C^0([0, T])}) \|g \|_{W^{s,p}(0, T)}
    \end{equation}
    for all $T\in (0,T']$, $f\in W^{s,p}(0,T;\R^N)$, and $g\in \prescript{}{0}{W}^{s,p}(0,T;\R^N)$. If $f,g\in \prescript{}{0}{W}^{s,p}(0,T';\R^N)$, then
    \begin{equation}\label{eq:end2}
        \|\langle f, g\rangle \|_{W^{s,p}(0, T)} \leq o(1) \|g\|_{W^{s,p}(0, T)}\quad\text{for $ T\searrow 0$}.
    \end{equation}
\end{lemma}
\begin{proof}
    Throughout this proof, $C$ is a constant that changes from line to line and depends at most on $T'$, $p$, and $s$. Let $T\in(0,T']$, $f\in W^{s,p}(0,T;\R^N)$, and $g\in \prescript{}{0}{W}^{s,p}(0,T;\R^N)$. Clearly,
    \begin{equation}
        \|\langle f,g\rangle\|_{L^p(0, T)} \leq \|f\|_{C^0([0, T])} \|g\|_{L^p(0, T)}. 
    \end{equation}
    Since $g(0)=0$, arguing as in \cite[Lemma~3.12]{garckemenzelpluda2020}, there exists $\tilde g\in W^{s,p}(0,T')$ with $\tilde g|_{[0,T]}=g$ and $\|\tilde g\|_{W^{s,p}(0,T')}\leq C(p,s) \|g\|_{W^{s,p}(0,T)}$. Using \eqref{eq:morrey} with $J=[0,T']$, we have
    \begin{equation}
        \|g\|_{C^0([0,T])} \leq  \|\tilde g\|_{C^0([0,T'])} \leq C(T',s,p) \|\tilde g\|_{W^{s,p}(0,T')} \leq C(T',s,p)\|g\|_{W^{s,p}(0,T)}.
    \end{equation}
    Thus, working with \eqref{eq:def-sobslob-seminorm},
    \begin{align}
        [\langle f,g\rangle]_{W^{s,p}(0, T)} &\leq \|g\|_{C^0([0, T])} [f]_{W^{s,p}(0, T)} + \|f\|_{C^0([0, T])} [g]_{W^{s,p}(0, T)}\\
        &\leq C(T',s,p) ([f]_{W^{s,p}(0, T)} + \|f\|_{C^0([0, T])}) \|g\|_{W^{s,p}(0, T)},
    \end{align}
    which yields \eqref{eq:end1}. Furthermore, \eqref{eq:end2} follows as $[f]_{W^{s,p}(0, T)} + \|f\|_{C^0([0, T])} \to |f(0)|$ for $T \searrow 0$. 
\end{proof}

\subsection{Well-posedness of the linear system: Proof of \Cref{prop:wp-lin-system}}\label{app:wp-lin-system}

In this section, we prove well-posedness of \eqref{eq:lin-system} and the estimate \eqref{eq:time-indep-maxreg-estimate}. To this end, we use \cite[Theorem~5.4]{solonnikov1965MathPhysics} and start by rewriting the linear system in the notation of \cite{solonnikov1965MathPhysics}. With 
\begin{align}
    \mathcal L(t,x,\partial_t,\partial_x) = \Big(\partial_t &+ \frac2{|\partial_x\bar\gamma(t,x)|^4}\partial_x^4 + \sum_{j=1}^3 a_{j}(|\partial_x\bar\gamma(t,x)|^{-1},\partial_x\bar\gamma(t,x),\dots,\partial_x^3\bar\gamma(t,x))\partial_x^j \\
    &- \Lambda(0) \cdot \sum_{j=1}^2b_{j}(|\partial_x\bar\gamma(t,x)|^{-1},\partial_x\bar\gamma(t,x),\partial_x^2\bar\gamma(t,x))\partial_x^j \Big) \mathrm{Id}_{n\times n}
\end{align}
and 
\begin{align}
    \mathcal B(t,x,\partial_t,\partial_x) = \begin{pmatrix}
        \xi(\bar\gamma(t,x))^t\\
        T_1(\bar\gamma(t,x))^t\partial_x+c^{(1)}(\bar\gamma(t,x))\\
        \vdots\\
        T_{n-1}(\bar\gamma(t,x))^t\partial_x+c^{(n-1)}(\bar\gamma(t,x))\\
        \frac{1}{|\partial_x\bar\gamma(t,x)|^3}\partial_x^3\mathrm{Id}_{n\times n} + \sum_{j=0}^2 d_{j,\bar\gamma(t,x)}(|\partial_x\bar\gamma(t,x)|^{-1},\partial_x\bar\gamma(t,x),\dots,\partial_x^3\bar\gamma(t,x))\partial_x^j
    \end{pmatrix},
\end{align}
\eqref{eq:lin-system} can be rewritten as $\mathcal L(t,x,\partial_t,\partial_x)\varphi = f$, $\varphi|_{t=0}=0$ and $\mathcal B(t,x,\partial_t,\partial_x)\varphi = (h_0,h_1,h_3)$.

With the notation of \cite[Section~1]{solonnikov1965MathPhysics}, the associated principal parts are given by
\begin{align}
    \mathcal L_0(t,x,p,\iu\zeta) &= (p+\frac{2}{|\partial_x\bar\gamma(t,x)|^4}\zeta^4) \mathrm{Id}_{n\times n},\\ 
    L(t,x,p,\iu \zeta) &= \det\mathcal L(t,x,p,\iu x)= (p+\frac{2}{|\partial_x\bar\gamma(t,x)|^4}\zeta^4)^n
\end{align}
and 
\begin{equation}
    \mathcal B_0(t,x,p,\iu\zeta)= \begin{pmatrix}
        \xi(\bar\gamma(t,x))^t\\
        \iu\zeta T_1(\bar\gamma(t,x))^t\\
        \vdots\\
        \iu\zeta T_{n-1}(\bar\gamma(t,x))^t\\
        -\iu\zeta^3 \frac{1}{|\bar\gamma(t,x)|^3}\mathrm{Id}_{n\times n} \end{pmatrix}.
\end{equation}

\subsubsection*{Parabolicity condition} Clearly, all zeros of $L$ as a polynomial in $p$ satisfy $p = -\frac2{|\partial_x\bar\gamma(t,x)|^4}\zeta^4$, so we have uniform parabolicity in the sense of \cite[p.~9]{solonnikov1965MathPhysics} with $\delta=\min_{t\in[0,T_0],\ x\in[-1,1]}\frac2{|\partial_x\bar\gamma(t,x)|}>0$, using that $\bar\gamma(t,\cdot)$ is an immersion for all $0\leq t\leq T_0$ by our choice of $T_0$.

\subsubsection*{Complementary condition} For the complementary condition as on \cite[p.~11]{solonnikov1965MathPhysics}, fix $x\in\partial I=\{\pm1\}$. W.l.o.g.\ we only consider the case $x=-1$ so that $\zeta(x)=0$, tangential vector, and $\nu(x)=1$ for the interior normal. The case $x=1$ is completely analogous with $\nu(x)=-1$. Let $t>0$ and $p\in\C$ with $p\neq 0$ and $\mathrm{Re} \ p\geq 0$. Writing 
\begin{equation}
    -p=|p|e^{\iu\theta}\quad\text{for some $\theta\in[\frac12\pi,\frac32\pi]$},
\end{equation}
the polynomial $M^+(t,x,p,\zeta,\tau)$ on \cite[p.~11]{solonnikov1965MathPhysics} turns out to be given by
\begin{equation}
    M^+(t,x,p,\zeta,\tau)=M^+(t,x,\tau) = (\tau-\tau_1^+)^n(\tau-\tau_2^+)^n
\end{equation}
where 
\begin{equation}
    \tau_1^+=|\partial_x\bar\gamma(t,x)|\sqrt[4]{|p|/2}\ e^{\iu \frac{\theta}{4}}\quad\text{and}\quad\tau_2^+=|\partial_x\bar\gamma(t,x)|\sqrt[4]{|p|/2}\ e^{\iu( \frac{\theta}{4}+\frac\pi2)}.
\end{equation}
\begin{remark}\label{rem:tau^2-and-tau^3}
    One computes
    \begin{equation}
        (\tau_1^+)^2 = - (\tau_2^+)^2\quad\text{and}\quad (\tau_1^+)^3 = - \iu (\tau_2^+)^3.
    \end{equation}
\end{remark}
Writing $\widetilde{\mathcal{A}}(t,x,p,\iu(\zeta+\tau\nu))=\mathcal B_0(t,x,p,\iu(\zeta+\tau\nu))\hat{\mathcal L}_0(t,x,p,\iu(\zeta+\tau\nu))$ where $\hat{\mathcal L}_0=L(\mathcal L_0)^{-1}$, the complementary condition is satisfied if and only if the rows of $\widetilde{\mathcal{A}}(t,x,p,\iu(\zeta+\tau\nu))$ are linearly independent modulo $M^+(t,x,p,\zeta,\tau)$ if $|p|^2>0$. Simplifying common factors in $\widetilde{\mathcal{A}}$ and $M^+$, we have that this is equivalent to the following. For all $w_1,w_2\in\C^n$,
\begin{equation}\label{eq:complementary-condition}
    (w_1^t,w_2^t) \mathcal B_0(t,x,p,\iu \tau) = 0 \text{ modulo }(\tau-\tau_1^+)(\tau-\tau_2^+) \implies w_1=w_2=0.
\end{equation}
Suppose $w_1,w_2\in\C^n$ satisfy $(w_1^t,w_2^t) \mathcal B_0(t,x,p,\iu \tau) = 0$ modulo $(\tau-\tau_1^+)(\tau-\tau_2^+)$. Plugging in $\tau=\tau_k^+$ for $k=1,2$ yields
\begin{equation}\label{eq:comp-cond-1}
    w_1^{(1)}\xi(\bar\gamma(t,x)) + \sum_{j=2}^n (\iu\tau_k^+w_1^{(j)}) T_{j-1}(\bar\gamma(t,x)) =  (\tau_k^+)^3 \sum_{j=1}^n \frac{\iu w_2^{(j)}}{|\partial_x\bar\gamma(t,x)|^3} e_j.
\end{equation}
Dividing by $(\tau_k^+)^3$, the right hand side no longer depends on $k$. So using that the collection of vectors $\{\xi(\bar\gamma(t,x)),T_1(\bar\gamma(t,x)),\dots,T_{n-1}(\bar\gamma(t,x))\}$ is linearly independent, we get
\begin{equation}
    \begin{cases}
        w_1^{(1)}\frac{1}{(\tau_1^+)^3} = w_1^{(1)}\frac{1}{(\tau_2^+)^3} &\text{and}\\
        w_1^{(j)}\frac{\iu}{(\tau_1^+)^2} = w_1^{(j)}\frac{\iu}{(\tau_2^+)^2}&\text{for all $j=2,\dots,n$}.
    \end{cases}
\end{equation} 
Using \Cref{rem:tau^2-and-tau^3}, this yields $w_1=0$. By \eqref{eq:comp-cond-1}, we find $w_2=\sum_{j=1}^n w_2^{(j)}e_2=0$ so that the complementary condition \eqref{eq:complementary-condition} is satisfied.

\subsubsection*{Regularity of the coefficients}

In the following, we verify the assumptions in \cite[Theorem~5.4]{solonnikov1965MathPhysics} on the regularity of the coefficients of $\mathcal L$ and $\mathcal B$. In the setting of \cite[Theorem~5.4]{solonnikov1965MathPhysics}, see also page 10 in the same reference, we have $k=0$, $b=2$, $l=0$, $m=n$, $t_1,\dots,t_m=4$, $s_1,\dots,s_m=0$, $r=m=n$, $\rho_1,\dots,\rho_m=-4$, and for the boundary conditions $\sigma_1=-4$, $\sigma_2,\dots,\sigma_n=-3$ and $\sigma_{n+1}=\dots=\sigma_{2n}=-1$. 
\begin{itemize}
    \item We need to check that any $\partial_t^\mu\partial_x^\nu$ derivative of the coefficients of $\mathcal L$ with $4\mu+\nu\leq 0$ is continuous, that is, $\mathcal L$ has continuous coefficients. This is clearly satisfied by \Cref{prop:xtp-mappings-and-embeddings}\eqref{it:hoeld-emb} since the coefficients are smooth functions of at most third-order spatial derivatives of $\bar\gamma$. 
    
    For the reader's convenience we point out that there is a typo in the english translation \cite[Theorem~5.4]{solonnikov1965MathPhysics} of the original russian version of \cite[Theorem~5.4]{solonnikov1965}.
    \item The coefficients of $\mathcal B_{qj}$ for $1\leq q\leq n$ and $1\leq j\leq n$ are constant in time and thus trivially belong to any (parabolic) Hölder space $C^{\frac14s,s}(\Gamma)=C^{\frac14s}([0,T_0])$ of the boundary cylinder $\Gamma=[0,T_0]\times \{\pm1\}$. Indeed, we have $\mathrm{tr}_{\partial I}\bar\gamma(t,\cdot)=\mathrm{tr}_{\partial I}\gamma_0$ by \eqref{eq:def-bargamma}. 
    \item Additionally using $\mathrm{tr}_{\partial I}\partial_x^3\bar\gamma(t,\cdot)=\mathrm{tr}_{\partial I}\partial_x^3\gamma_0$ by \eqref{eq:def-bargamma}, time-dependence of the coefficients of $\mathcal B_{qj}$ for $n+1\leq q\leq 2n$ and $1\leq j\leq n$ is due to $\mathrm{tr}_{\partial I}\partial_x\bar\gamma$ and $\mathrm{tr}_{\partial I}\partial_x^2\bar\gamma$ which both belong to $C^{\frac12-\frac{1}{4p}-\frac1p}([0,T_0])$ by \Cref{prop:xtp-mappings-and-embeddings}\eqref{it:trkder-emb} and \eqref{eq:morrey}. Using that 
    \begin{equation}
        \frac12-\frac{1}{4p}-\frac1p = \frac{1}{4} - \frac{1}{4p} + \underbrace{\frac14-\frac1p}_{>\frac14-\frac15>0},
    \end{equation}
    and that $C^{\frac12-\frac{1}{4p}-\frac1p}([0,T_0])$ is an algebra, also the coefficients of $\mathcal B_{qj}$ have the required regularity for $n+1\leq q\leq 2n$ and $1\leq j\leq n$.
\end{itemize}

\subsubsection*{
Compatibility conditions}

The compatibility conditions of order $-1$, i.e.\ $i_q=0$ in \cite[Equation~(4.19)]{solonnikov1965MathPhysics} for all $1\leq q\leq 2n$, are clearly satisfied for zero initial data due to the zero time-trace-spaces used for the boundary data in $\prescript{}{0}{\mathbb F}_{T,p}$.

\subsubsection*{Conclusion}

Altogether, \cite[Theorem~5.4]{solonnikov1965MathPhysics} applies and yields the claim of \Cref{prop:wp-lin-system} except for the $T_0$-dependence of the constant in \eqref{eq:unif-bd-derivatives}. The fact that the constant only depends on $T_0$, not on $T$ can be shown by applying what we proved above on the interval $[0,T_0]$ and, for $T\leq T_0$, extending the data onto $[0,T_0]$ as in \cite[Lemmas~3.10 and 3.12]{garckemenzelpluda2020}, critically using that the temporal trace vanishes.

Finally, the continuous dependence of the constant on $\gamma_0$ in $W^{4(1-\frac1p),p}(I)$ follows from \Cref{prop:xtp-mappings-and-embeddings}\eqref{it:hoeld-emb} combined with \eqref{eq:bargamma-gamma0-estimate} applied to $E\gamma_0-E\tilde\gamma_0$.




\printbibliography


\end{document}